\pdfoutput=1
\documentclass[11pt,a4paper,reqno]{amsart}

\usepackage[british]{babel}

\usepackage[DIV=9,oneside,BCOR=0mm]{typearea}
\usepackage{microtype}
\usepackage[T1]{fontenc}
\usepackage{setspace}
\usepackage{amssymb}
\usepackage{amsmath}
\usepackage{amsthm}
\usepackage{enumitem}
\usepackage{mathrsfs}
\usepackage[colorlinks,citecolor=blue,urlcolor=black,linkcolor=blue,linktocpage]{hyperref}
\usepackage{mathrsfs}
\usepackage{mathtools}
\usepackage{xcolor}
\usepackage{xfrac}
\usepackage{multicol}
\usepackage{ifsym}

\newtheorem{thm}{Theorem}[section]

\newtheorem{cor}[thm]{Corollary}
\newtheorem{lem}[thm]{Lemma}
\newtheorem{prop}[thm]{Proposition}

\theoremstyle{definition}

\newtheorem{definition}[thm]{Definition}
\newtheorem{remark}[thm]{Remark}

\renewcommand{\epsilon}{\varepsilon}
\renewcommand{\phi}{\varphi}
\newcommand{\defeq}{\mathrel{\mathop:}=}
\newcommand{\eqdef}{\mathrel{\mathopen={\mathclose:}}}

\DeclareMathOperator{\free}{F}
\DeclareMathOperator{\I}{I}
\DeclareMathOperator{\lat}{L}
\DeclareMathOperator{\latop}{L^{\!{}^{{}_\circlearrowleft}}\!}
\DeclareMathOperator{\E}{E}
\DeclareMathOperator{\GL}{GL}
\DeclareMathOperator{\SL}{SL}

\DeclareMathOperator{\id}{id}
\DeclareMathOperator{\im}{im}
\DeclareMathOperator{\Aut}{Aut}
\DeclareMathOperator{\rAnn}{Ann_r}
\DeclareMathOperator{\lAnn}{Ann_l}
\DeclareMathOperator{\ZZ}{Z}

\DeclareMathOperator{\R}{\mathbb{R}}
\DeclareMathOperator{\Q}{\mathbb{Q}}

\DeclareMathOperator{\N}{\mathbb{N}}

\DeclareMathOperator{\Pfin}{\mathcal{P}_{fin}}
\DeclareMathOperator{\rk}{rk}
\DeclareMathOperator{\rank}{rank}
\DeclareMathOperator{\cha}{char}
\DeclareMathOperator{\Nil}{Nil}
\DeclareMathOperator{\M}{M}

\DeclareMathOperator{\Homeo}{Homeo}

\makeatletter
\def\moverlay{\mathpalette\mov@rlay}
\def\mov@rlay#1#2{\leavevmode\vtop{%
		\baselineskip\z@skip \lineskiplimit-\maxdimen
		\ialign{\hfil$\m@th#1##$\hfiincr#2\crcr}}}
\newcommand{\charfusion}[3][\mathord]{
	#1{\ifx#1\mathop\vphantom{#2}\fi
		\mathpalette\mov@rlay{#2\cr#3}
	}
	\ifx#1\mathop\expandafter\displaylimits\fi}

\def\smallunderbrace#1{\mathop{\vtop{\m@th\ialign{##\crcr
				$\hfil\displaystyle{#1}\hfil$\crcr
				\noalign{\kern3\p@\nointerlineskip}%
				\tiny\upbracefill\crcr\noalign{\kern3\p@}}}}\limits}
\makeatother

\begin{document}
	
\author{Josefin Bernard}
\address{J.B., Institute of Discrete Mathematics and Algebra, TU Bergakademie Freiberg, 09596 Freiberg, Germany}
\email{josefin.bernard@math.tu-freiberg.de}
\author{Friedrich Martin Schneider}
\address{F.M.S., Institute of Discrete Mathematics and Algebra, TU Bergakademie Freiberg, 09596 Freiberg, Germany}
\email{martin.schneider@math.tu-freiberg.de}
\thanks{This research is funded by the Deutsche Forschungsgemeinschaft (DFG, German Research Foundation) -- Projektnummer 561178190}
	
\title[Width bounds and Steinhaus property]{Width bounds and Steinhaus property for\\ unit groups of continuous rings}
\date{\today}
	
\begin{abstract} We prove an algebraic decomposition theorem for the unit group $\GL(R)$ of an arbitrary non-discrete irreducible, continuous ring $R$ (in von Neumann's sense), which entails that every element of $\GL(R)$ is both a product of $7$ commutators and a product of $16$ involutions. Combining this with further insights into the geometry of involutions, we deduce that $\GL(R)$ has the so-called Steinhaus property with respect to the natural rank topology, thus every homomorphism from $\GL(R)$ to a separable topological group is necessarily continuous. Due to earlier work, this has further dynamical ramifications: for instance, for every action of $\GL(R)$ by homeomorphisms on a non-void metrizable compact space, every element of $\GL(R)$ admits a fixed point in the latter. In particular, our results answer two questions by Carderi and Thom, even in generalized form. \end{abstract}
	
\subjclass[2020]{22A05, 37B02, 06C20, 16E50}
	
\keywords{Continuous ring, involution width, commutator width, verbal width, perfect group, topological group, Steinhaus property, automatic continuity}
	
\maketitle
	
\allowdisplaybreaks	
	
	
\tableofcontents
	
\section{Introduction}\label{section:introduction}

The topic of the present paper is the unit group $\GL(R)$, i.e., the group of invertible elements, of an arbitrary irreducible, continuous ring $R$. The study of such rings was initiated and profoundly advanced by John von Neumann in his fundamental work on \emph{continuous geometry}~\cite{VonNeumannBook}, a continuous-dimensional counterpart of finite-dimensional projective geometry. Inspired by conversations with Garrett Birkhoff on the lattice-theoretic formulation of projective geometry~\cite{BirkhoffBulletin}, von Neumann introduced \emph{continuous geometries} 
as complete, complemented, modular lattices possessing a certain continuity property, and he established a distinguished algebraic representation for these objects through his \emph{coordinatization theorem}~\cite[II.XIV, Theorem~14.1, p.~208]{VonNeumannBook}: for any complemented modular lattice $L$ of an order at least $4$ there exists a regular ring $R$, unique up to isomorphism, such that $L$ is isomorphic to the lattice $\lat(R)$ of principal right ideals of~$R$. Many algebraic properties and constructions translate conveniently via the resulting correspondence. For instance, a regular ring~$R$ is (directly) irreducible if and only if the lattice $\lat(R)$ is. A \emph{continuous ring} is a regular ring $R$ whose associated lattice $\lat(R)$ is a continuous geometry.
	 
Another deep result due to von Neumann~\cite{VonNeumannBook} asserts that every irreducible, continuous ring $R$ admits a unique \emph{rank function}---naturally corresponding to a unique \emph{dimension function} on the continuous geometry $\lat(R)$. This rank function gives rise to a complete metric on $R$, and the resulting topology, which we call the \emph{rank topology}, turns $R$ into a topological ring. While an irreducible, continuous ring is \emph{discrete} with respect to its rank topology if and only if it is isomorphic to a finite-dimensional matrix ring over some division ring (see Remark~\ref{remark:rank.function.general}\ref{remark:characterization.discrete}), the class of \emph{non-discrete} irreducible, continuous rings appears ineffably large. In addition to the original example given by the ring of densely defined, closed, linear operators affiliated with an arbitrary $\mathrm{II}_{1}$ factor~\cite{MurrayVonNeumann}, such objects have emerged in the study of Kaplansky's direct finiteness conjecture~\cite{ElekSzabo,linnell} and the Atiyah conjecture~\cite{LinnellSchick,elek2013,elek}. As discovered already by von Neumann~\cite{NeumannExamples}, a rather elementary example may be constructed from any field $K$ via the following procedure. Consider the inductive limit of the matrix rings \begin{displaymath}
	K \, \cong \, \M_{2^{0}}(K) \, \lhook\joinrel\longrightarrow \, \ldots \, \lhook\joinrel\longrightarrow \, \M_{2^{n}}(K) \, \lhook\joinrel\longrightarrow \, \M_{2^{n+1}}(K) \, \lhook\joinrel\longrightarrow \, \ldots
\end{displaymath} relative to the embeddings \begin{displaymath}
	\M_{2^n}(K)\,\lhook\joinrel\longrightarrow \,  \M_{2^{n+1}}(K), \quad a\,\longmapsto\, 
	\begin{pmatrix}
		a & 0\\
		0 & a
	\end{pmatrix} \qquad (n \in \N) .
\end{displaymath} Those embeddings are isometric with respect to the normalized rank metrics \begin{displaymath}
	\M_{2^{n}}(K) \times \M_{2^{n}}(K) \, \longrightarrow \, [0,1], \quad (a,b) \, \longmapsto \, \tfrac{\rank(a-b)}{2^{n}} \qquad (n \in \N) ,
\end{displaymath} whence the latter jointly extend to a metric on the inductive limit. The corresponding metric completion, which we denote by $\M_{\infty}(K)$, is a non-discrete irreducible, continuous ring~\cite{NeumannExamples,Halperin68}.

The paper at hand explores algebraic and dynamical properties of the group $\GL(R)$ for an arbitrary non-discrete irreducible, continuous ring $R$. The center $\ZZ(R)$ of such a ring constitutes a field, and viewing $R$ as a unital $\ZZ(R)$-algebra naturally leads us to considering matricial subalgebras of $R$ and studying \emph{simply special} and \emph{locally special} elements of $\GL(R)$ (see Definition~\ref{definition:simply.special} and Definition~\ref{definition:locally.special} for details). Our first main result is the following uniform decomposition theorem.
	
\begin{thm}[Theorem~\ref{theorem:decomposition}]\label{theorem:decomposition.introduction} Let $R$ be a non-discrete irreducible, continuous ring. Every element $a\in\GL(R)$ admits a decomposition \begin{displaymath}
	a \, = \, bu_{1}v_{1}v_{2}u_{2}v_{3}v_{4}
\end{displaymath} where \begin{enumerate}
	\item[---\,] $b \in \GL(R)$ is simply special,
	\item[---\,] $u_{1},u_{2} \in \GL(R)$ are locally special,
	\item[---\,] $v_{1},v_{2},v_{3},v_{4} \in \GL(R)$ are locally special involutions.
\end{enumerate} \end{thm}

The proof of Theorem~\ref{theorem:decomposition.introduction} combines a decomposition argument inspired by Fathi's work on algebraic properties of the group $\Aut([0,1],\lambda)$~\cite{Fathi} with our Corollary~\ref{corollary:simply.special.dense}, which establishes density of the set of simply special elements in the unit group of any non-discrete irreducible, continuous ring $R$. The latter result rests on two pillars: the work of von Neumann~\cite{VonNeumann37} and Halperin~\cite{Halperin62} about the density of the set of \emph{algebraic} elements in $R$, and our characterization of such elements as those contained in some matricial subalgebra of $R$ (Theorem~\ref{theorem:matrixrepresentation.case.algebraic}).

Thanks to results of Gustafson, Halmos and Radjavi~\cite{GustafsonHalmosRadjavi76}, Thompson~\cite{Thompson61,Thompson62,ThompsonPortugaliae}, and Borel~\cite{Borel}, our Theorem~\ref{theorem:decomposition.introduction} has some notable consequences. For a group $G$ and a natural number $m$, let us recall that any map $g \in G^{m}$ naturally extends to a unique homomorphism $\free(m) \to G, \, w \mapsto w(g)$ from the free group $\free(m)$ on $m$ generators to~$G$, and let us define $w(G) \defeq \{ w(g) \mid g \in G^{m} \}$. A group $G$ is said to be \emph{verbally simple} if, for every $m \in \N$ and every $w \in \free(m)$, the subgroup of $G$ generated by~$w(G)$ is trivial or coincides with $G$.

\begin{cor}[Theorem~\ref{theorem:width}]\label{corollary:width.introduction} Let $R$ be a non-discrete irreducible, continuous~ring. 
\begin{enumerate}
	\item\label{corollary:width.introduction.involutions} Every element of $\GL(R)$ is a product of $16$ involutions.
	\item\label{corollary:width.introduction.commutators} Every element of $\GL(R)$ is a product of $7$ commutators. In particular, $\GL(R)$ is perfect.
	\item\label{corollary:width.introduction.word} Suppose that $\ZZ(R)$ is algebraically closed. For all $m \in \N$ and $w \in \free(m)\setminus \{ \epsilon \}$, \begin{displaymath}
					\qquad \GL(R) \, = \, w(\GL(R))^{14} .
				\end{displaymath} In particular, $\GL(R)$ is verbally simple.
\end{enumerate} \end{cor}

In the language of~\cite{Liebeck}, the corollary above establishes uniform finite upper bounds for \emph{involution width}, \emph{commutator width}, and \emph{$w$-width} for any non-trivial reduced word $w$ in finitely many letters and their inverses. For every non-discrete irreducible, continuous ring $R$, the commutator subgroup had been known to be dense in $\GL(R)$ with respect to the rank topology, due to Smith's work~\cite[Theorems~1 and~2]{Smith57}.
	
Our algebraic results have non-trivial dynamical ramifications.	To be more precise, we recall that a topological group $G$ has~\emph{automatic continuity}~\cite{KechrisRosendal} if every homomorphism from $G$ to any separable topological group is continuous. Examples of topological groups with this property include the automorphism group of the ordered rational numbers endowed with the topology of pointwise convergence~\cite{RosendalSolecki}, the unitary group of an infinite-dimensional separable Hilbert space equipped with the strong operator topology~\cite{Tsankov}, and the full group of an ergodic measure-preserving countable equivalence relation furnished with the uniform topology~\cite{KittrellTsankov}. The reader is referred to~\cite{RosendalSuarez} for a recent survey on this subject. One route towards automatic continuity is provided by the Steinhaus property: given some $n \in \N$, a topological group $G$ is called \emph{$n$-Steinhaus}~\cite[Definition~1]{RosendalSolecki} if, for every symmetric and countably syndetic\footnote{A subset $W$ of a group $G$ is called \emph{countably syndetic} if there exists a countable subset $C\subseteq G$ such that $G=CW$.} subset $W\subseteq G$, the set $W^{n}$ is an identity neighborhood in $G$.

Equipped with the associated rank topology, the unit group of any irreducible, continuous ring constitutes a topological group. Using Corollary~\ref{corollary:width.introduction}\ref{corollary:width.introduction.involutions} along with further insights into the geometry of involutions, we establish the Steinhaus property for all of these topological groups in a uniform way.
	
\begin{thm}[Theorem~\ref{theorem:194-Steinhaus}]\label{theorem:194-Steinhaus.introduction} Let $R$ be an irreducible, continuous ring. Then the topological group $\GL(R)$ is $194$-Steinhaus. In particular, $\GL(R)$ has automatic continuity. \end{thm}

Since a Polish topological group with automatic continuity admits a unique Polish group topology, we obtain the following.
	
\begin{cor}\label{corollary:unique.polish.topology} Let $R$ be an irreducible, continuous ring. If $R$ is separable with respect to the rank topology\footnote{equivalently, if $\GL(R)$ is separable with respect to the rank topology (Remark~\ref{remark:separable})}, then the latter restricts to the unique Polish group topology on~$\GL(R)$. \end{cor}

In~\cite{CarderiThom}, Carderi and Thom asked, given any finite field $K$, whether $\GL(\M_{\infty}(K))$ would have a unique Polish topology, or even automatic continuity. In particular, Theorem~\ref{theorem:194-Steinhaus.introduction} and Corollary~\ref{corollary:unique.polish.topology} provide affirmative answers to both of these questions, even in generalized form. Moreover, in combination with results of~\cite{SchneiderGAFA} and~\cite{CarderiThom}, our Theorem~\ref{theorem:194-Steinhaus.introduction} turns out to have non-trivial dynamical consequences for the underlying abstract groups, too.

\begin{cor}\label{corollary:inert} Let $R$ be a non-discrete irreducible, continuous ring. For every action of $\GL(R)$ by homeomorphisms on a non-empty metrizable compact space  $X$ and every~$g\in\GL(R)$, there exists $x\in X$ with $gx=x$. \end{cor}

\begin{cor}\label{corollary:fixpoint} If $K$ is a finite field, then every action of $\GL(\M_{\infty}(K))$ by homeomorphisms on a non-empty metrizable compact space has a fixed point. \end{cor}

Beyond the above, Theorem~\ref{theorem:194-Steinhaus.introduction} has found a representation-theoretic application in combination with more recent work of the second-named author and Thom~\cite{SchneiderThom}. The latter establishes that, for any non-discrete irreducible, continuous ring $R$, the topological group $\GL(R)$ is \emph{exotic}, i.e., it does not admit any non-trivial strongly continuous unitary representation~\cite[Theorem~1.1]{SchneiderThom}. Thanks to Theorem~\ref{theorem:194-Steinhaus.introduction}, this immediately entails that every unitary representation of group $\GL(R)$ on a separable Hilbert space must be trivial~\cite[Corollary~1.3]{SchneiderThom}.
	
This article is organized as follows. The first three sections provide some necessary preliminaries concerning continuous geometries (Section~\ref{section:continuous.geometries}), regular rings and their principal ideal lattices (Section~\ref{section:regular.rings}), as well as continuous rings and their rank functions (Section~\ref{section:continuous.rings.and.rank.functions}). In Section~\ref{section:subgroups.of.the.form.Gamma_e(R)}, we introduce and describe a natural family of subgroups of the unit group of a unital ring induced by idempotent ring elements. The subsequent Section~\ref{section:geometry.of.involutions} is dedicated to the study of involutions in irreducible, continuous rings, their conjugacy classes, and further structural properties. In Section~\ref{section:dynamical.independence}, we revisit an argument due to von Neumann~\cite{NeumannExamples} and Halperin~\cite{Halperin62} and extract from it a notion of dynamical independence for ideals. The latter constitutes a key concept in our analysis of algebraic elements pursued in Section~\ref{section:algebraic.elements}. These insights are then put into use for proving density of simply matricial elements in any irreducible, continuous ring (Section~\ref{section:approximation.by.matrix.algebras}), establishing our main decomposition theorem for the unit group of such (Section~\ref{section:decomposition.into.locally.special.elements}), and deducing the Steinhaus property for the associated rank topology (Section~\ref{section:steinhaus.property}). The resulting Corollaries~\ref{corollary:unique.polish.topology}, \ref{corollary:inert} and \ref{corollary:fixpoint} are proved in the final Section~\ref{section:proofs.of.corollaries}.
	
\section{Continuous geometries}\label{section:continuous.geometries}

In this section, we give a brief overview of the basic concepts related to von Neumann's continuous geometry~\cite{VonNeumannBook}. For a start, let us recall some order-theoretic terminology. A \emph{lattice} is a partially ordered set $L$ in which every pair of elements $x,y\in L$ admits both a (necessarily unique) supremum $x\vee y \in L$ and a (necessarily unique) infimum $x\wedge y \in L$. Equivalently, a lattice may be defined as an algebraic structure carrying two commutative, associative binary operations satisfying the two absorption laws. We will freely use aspects of both of these definitions. Moreover, a lattice $L$ is called \begin{enumerate}
	\item[---\,] \emph{complete} if for every $S \subseteq L$ there is a (necessarily unique) supremum $\bigvee S \in L$, or equivalently, if for every subset $S \subseteq L$ there exists a (necessarily unique) infimum $\bigwedge S \in L$,
	\item[---\,]  \emph{(directly) irreducible} if $\vert L \vert \geq 2$ and $L$ is not isomorphic to a direct product of two lattices of cardinality at least 2,
	\item[---\,] \emph{bounded} if $L$ has a (necessarily unique) least element $0\in L$ and a (necessarily unique) greatest element $1\in L$,
	\item[---\,] \emph{complemented} if $L$ is bounded and
			\begin{displaymath}
					\qquad	\forall x\in L\ \exists y\in L\colon \quad x\wedge y=0,\ \, x\vee y=1,
			\end{displaymath}
	\item[---\,]\emph{modular} if
			\begin{displaymath}
					\qquad	\forall x,y,z\in L\colon\quad x\leq y\ \Longrightarrow\ x\vee (y\wedge z)=y\wedge (x\vee z).
			 \end{displaymath}
\end{enumerate} A \emph{continuous geometry} is a complete, complemented, modular lattice $L$ such that, for every chain\footnote{A subset $C$ of a partially ordered set $P$ is called a \emph{chain} if the restriction of the order of $P$ to $C$ is a linear order on $C$.} $C\subseteq L$ and every element $x \in L$, \begin{displaymath}
	x\wedge \bigvee C \, = \, \bigvee\{x\wedge y\mid y\in C\},\qquad x\vee\bigwedge C \, = \, \bigwedge\{x\vee y\mid y\in C\}.
\end{displaymath}

One of the main results of von Neumann's work~\cite{VonNeumannBook} about irreducible continuous geometries asserts the existence of a unique dimension function on such (Theorem~\ref{theorem:dimension.function.lattice}). Before providing details on this, let us record two general facts about modular lattices.
	
\begin{lem}\label{lemma:complement} Let $L$ be a complemented, modular lattice. If $x,y,z \in L$ satisfy $x \wedge y = 0$ and $x \vee y \leq z$, then there exists $x' \in L$ with $x \leq x'$, $x' \wedge y = 0$ and $x' \vee y = z$. \end{lem}

\begin{proof} Let $x,y,z \in L$ with $x \wedge y = 0$ and $x \vee y \leq z$. Consider $y' \defeq x \vee y$ and note that $x \leq y' \leq z$. By~\cite[I.I, Theorem~1.3, p.~5]{VonNeumannBook} (see also \cite[VIII.1, Theorem~1, p.~114]{BirkhoffBook}), there exists $x' \in L$ such that $x' \wedge y' = x$ and $x' \vee y' = z$. Then $x = x' \wedge y' \leq x'$ and $x' \vee y = z$. Moreover, $x' \wedge y = x' \wedge (y' \wedge y) = (x' \wedge y') \wedge y = x \wedge y = 0$. \end{proof}

\begin{remark}\label{remark:independence.equivalence.intersection} Let $L$ be a bounded lattice, let $n\in\N$ and $(x_{1},\dots,x_{n})\in L^{n}$. Then $(x_{1},\dots,x_{n})$ is said to be \emph{independent} in $L$ and we write $(x_{1},\dots,x_{n})\perp$ if \begin{displaymath}
	\forall I,J \subseteq \{ 1,\ldots,n\} \colon \quad I\cap J=\emptyset \ \Longrightarrow \ \left(\bigvee\nolimits_{i\in I}x_{i}\right)\!\wedge\! \left(\bigvee\nolimits_{j\in J}x_{j}\right)\! =0 .
\end{displaymath} If $L$ is modular, then~\cite[I.II, Theorem~2.2, p.~9]{VonNeumannBook} (see also~\cite[I.1, Satz 1.8, p.~13]{MaedaBook}) asserts that \begin{displaymath}
	(x_{1},\ldots,x_{n}) \perp \quad \Longleftrightarrow \quad \forall i\in\{1,\ldots,n-1\}\colon \ (x_{1} \vee \ldots \vee x_{i}) \wedge x_{i+1} = 0.
\end{displaymath} \end{remark}

In order to clarify some terminology needed for Theorem~\ref{theorem:dimension.function.lattice}, let $L$ be a bounded lattice. A \emph{pseudo-dimension function} on $L$ is a map $\delta \colon L\to [0,1]$ such that \begin{enumerate}
	\item[---\,] $\delta(0)=0$ and $\delta(1)=1$,
	\item[---\,] $\delta$ is \emph{monotone}, i.e., \begin{displaymath}
					\qquad \forall x,y \in L \colon \quad x\leq y \ \Longrightarrow \ \delta(x)\leq \delta(y),
				\end{displaymath}
	\item[---\,] $\delta$ is \emph{modular}, i.e., \begin{displaymath}
					\qquad \forall x,y \in L \colon \quad \delta(x \vee y) + \delta(x \wedge y) = \delta(x)+\delta(y).
				\end{displaymath}
\end{enumerate} A \emph{dimension function} on $L$ is a pseudo-dimension function $\delta \colon L \to [0,1]$ such that \begin{displaymath}
	\forall x,y \in L \colon \quad x < y \ \Longrightarrow \ \delta(x) < \delta(y).
\end{displaymath} For any pseudo-dimension function $\delta$ on $L$, the map \begin{displaymath}
	d_{\delta}\colon\, L\times L\,\longrightarrow\,[0,1],\quad	(x,y)\,\longmapsto\,\delta(x\vee y)-\delta(x\wedge y).
\end{displaymath} constitutes a pseudo-metric on $L$~\cite[V.7, Theorem~9, p.~77]{BirkhoffBook}, and we see that $d_{\delta}$ is a metric if and only if $\delta$ is a dimension function on $L$.

\begin{thm}[\cite{VonNeumannBook}]\label{theorem:dimension.function.lattice} Let $L$ be an irreducible, continuous geometry. Then there exists a unique dimension function $\delta_{L}$ on $L$. Moreover, the metric $d_{L} \defeq d_{\delta_{L}}$ is complete. \end{thm}

\begin{proof} See~\cite[I.VI, Theorem~6.9, p.~52]{VonNeumannBook} for existence and~\cite[I.VII, p.~60, Corollary~1]{VonNeumannBook} for uniqueness. Completeness of $d_{L}$ is proved in~\cite[I.6, Satz~6.4, p.~48]{MaedaBook}. \end{proof}

\begin{prop}[\cite{MaedaBook}]\label{proposition:dimension.function.continuous} Let $L$ be an irreducible continuous geometry. If $C$ is a chain in $L$, then \begin{displaymath}
	\delta_{L}\!\left(\bigvee C\right)\! \, = \, \sup \delta_{L}(C), \qquad \delta_{L}\!\left(\bigwedge C\right)\! \, = \, \inf \delta_{L}(C) .
\end{displaymath} \end{prop}

\begin{proof} This is due to~\cite[V.2, Satz~2.1, p.~118]{MaedaBook} (see~\cite[V.1, Satz~1.8, p.~117]{MaedaBook} for a more general result). \end{proof}

\section{Regular rings}\label{section:regular.rings}

The purpose of this section is to recollect some background material about regular rings from~\cite{VonNeumannBook} (see also~\cite{GoodearlBook, MaedaBook}). We start off with some bits of notation that will be of central relevance to the entire manuscript. To this end, let $R$ be a unital ring. Then we consider its \emph{center} \begin{displaymath}
	\ZZ(R) \, \defeq \, \{a\in R\mid \forall b\in R\colon\, ab=ba\} ,
\end{displaymath} which is a commutative unital subring of $R$, its \emph{unit group} \begin{displaymath}
	\GL(R) \, \defeq \, \{a\in R\mid \exists b\in R\colon\, ab=ba=1\} ,
\end{displaymath} equipped with the multiplication inherited from $R$, and the sets \begin{displaymath}
	\lat(R) \, \defeq \, \{aR\mid a\in R\}, \qquad \latop(R) \, \defeq \, \{Ra\mid a\in R\}
\end{displaymath} of \emph{principal right ideals} (resp., \emph{principal left ideals}) of $R$, partially ordered by inclusion. As would seem natural, we will frequently view $R$ as a unital $\ZZ(R)$-algebra.

\begin{remark}\label{remark:quantum.logic} Let $R$ be a unital ring. The set \begin{displaymath}
	\E(R) \, \defeq \, \{ e \in R \mid ee = e \}
\end{displaymath} of \emph{idempotent} elements of $R$ naturally carries a partial order defined by \begin{displaymath}
	e\leq f \quad :\Longleftrightarrow \quad ef=fe=e \qquad (e,f\in\E(R)).
\end{displaymath} Two elements $e,f\in\E(R)$ are called \emph{orthogonal} and we write $e\perp f$ if $ef=fe=0$. Let $e,f \in \E(R)$. The following hold. \begin{enumerate}
	\item\label{remark:quantum.logic.1} If $f \leq e$, then $e-f\in\E(R)$ and $f\perp (e-f)\leq e$.
	\item\label{remark:quantum.logic.2} If $e\perp f$, then $e+f \in \E(R)$ satisfies $e \leq e+f$ and $f \leq e+f$.
\end{enumerate} \end{remark}
 
We now turn to the class of regular rings. A unital ring $R$ is called \emph{(von Neumann) regular} if, for every $a \in R$, there exists $b \in R$ such that $aba = a$.
	
\begin{remark}[\cite{VonNeumannBook}, II.II, Theorem~2.2, p.~70]\label{remark:regular.idempotent.ideals} Let $R$ be a unital ring. Then the following are equivalent. \begin{enumerate}
	\item[---\,] $R$ is regular.
	\item[---\,] $\lat(R) = \{ eR \mid e \in \E(R) \}$.
	\item[---\,] $\latop(R) = \{ Re \mid e \in \E(R) \}$.
\end{enumerate} \end{remark}

Our next remark requires some additional notation. If $R$ is a unital ring, then, for any $S \subseteq R$ and $a \in R$, the corresponding \emph{right annihilators} in $R$ are defined as \begin{displaymath}
	\rAnn(S) \defeq \{x\in R\mid \forall s\in S\colon\, sx=0 \}, \qquad \rAnn(a) \defeq \rAnn(\{ a \}),
\end{displaymath} and the corresponding \emph{left annihilators} in $R$ are defined as \begin{displaymath}
	\lAnn(S) \defeq \{x\in R\mid \forall s\in S\colon\, xs=0 \}, \qquad \lAnn(a) \defeq \lAnn(\{ a \}).
\end{displaymath}		
		
\begin{remark}\label{remark:bijection.annihilator} Let $R$ be a regular ring. By~\cite[II.II, Corollary~2, p.~71 and II.II, Lemma~2.3, p.~72]{VonNeumannBook}, the maps \begin{align*}
	&(\lat(R),{\subseteq}) \, \longrightarrow \, (\latop(R),{\subseteq}), \quad I \, \longmapsto \, \lAnn(I) ,\\
	&(\latop(R),{\subseteq}) \, \longrightarrow \, (\lat(R),{\subseteq}), \quad I \, \longmapsto \, \rAnn(I)
\end{align*} are mutually inverse order-reversing bijections. Due to~\cite[II.II, Theorem~2.4, p.~72]{VonNeumannBook}, the partially ordered set $(\lat(R),{\subseteq})$ (resp., $(\latop(R),{\subseteq})$) is a complemented, modular lattice, in which \begin{displaymath}
	I\vee J\,= \, I+J, \qquad I\wedge J \, = \, I\cap J
\end{displaymath} for all $I,J \in \lat(R)$ (resp., $I,J \in \latop(R)$). Moreover, by \cite[II.II, Lemma~2.2(i), p.~71]{VonNeumannBook}, \begin{align*}
	\forall e\in\E(R)\colon\qquad	R(1-e)=\lAnn(eR),\quad (1-e)R=\rAnn(Re).
\end{align*} \end{remark}

\begin{lem}\label{lemma:partial.inverse} Let $R$ be a regular ring, $f\in\E(R)$ and $a\in R$ with $\rAnn(a) \subseteq (1-f)R$. Then there exists $b\in R$ such that $ba=f$. \end{lem}

\begin{proof} We observe that \begin{displaymath}
	Ra \, \stackrel{\ref{remark:bijection.annihilator}}{=} \, \lAnn(\rAnn(Ra)) \, = \, \lAnn(\rAnn(a)) \, \supseteq \, \lAnn((1-f)R) \, \stackrel{\ref{remark:bijection.annihilator}}{=} \, Rf .
\end{displaymath} In particular, there exists $b\in R$ such that $ba=f$. \end{proof}

In the principal right ideal lattice of a regular ring, independence can be described by means of orthogonal idempotents as follows.

\begin{remark}\label{remark:independence.ideals.idempotents} Let $R$ be a regular ring, let $n \in \N$ and let $I_{1},\dots,I_{n} \in \lat(R)$. It is straightforward to check that $(I_{1},\dots, I_{n})\perp$ if and only if \begin{displaymath}
	\forall (x_{1},\ldots,x_{n}) \in I_{1} \times \dots \times I_{n} \colon \quad x_{1}+\ldots+x_{n}=0 \ \Longrightarrow \ x_{1}=\ldots=x_{n}=0.
\end{displaymath} As is customary, if $(I_{1},\dots, I_{n})\perp$, then we write \begin{displaymath}
	I_{1}\oplus \ldots \oplus I_{n} \, \defeq \, I_{1}+\ldots+I_{n} \, = \, \sum\nolimits_{j=1}^{n} I_{j} \, \eqdef \, \bigoplus\nolimits_{j=1}^{n} I_{j} .
\end{displaymath} By~\cite[II.III, Lemma~3.2, p.~94 and its Corollary, p.~95]{VonNeumannBook}, the following are equivalent. \begin{enumerate}
	\item[---\,] $I_{1} \oplus \ldots \oplus I_{n} = R$.
	\item[---\,] There exist $e_{1},\dots,e_{n}\in\E(R)$ pairwise orthogonal such that $e_{1}+\dots+e_{n}=1$ and $I_{j}=e_{j}R$ for each $j \in \{1,\ldots,n\}$.
\end{enumerate} \end{remark}

A ring $R$ is called \emph{(directly) irreducible} if $R$ is non-zero and not isomorphic to the direct product of two non-zero rings. The following characterization of irreducibility for regular rings will be used frequently throughout the entire manuscript.

\begin{remark}[\cite{VonNeumannBook}, II.II, Theorem~2.7, p.~75 and II.II, Theorem~2.9, p.~76]\label{remark:irreducible.center.field} Let $R$ be a regular ring. The following are equivalent. \begin{enumerate}
	\item[---\,] $R$ is irreducible.
	\item[---\,] $\ZZ(R)$ is a field.
	\item[---\,] $\lat(R)$ is irreducible.
\end{enumerate} \end{remark}

A basic ingredient in later decomposition arguments will be the following standard construction of subrings induced by idempotent elements of a given ring.

\begin{remark}\label{remark:corner.rings} Let $R$ be a unital ring, $e \in \E(R)$. Then $eRe$ is a subring of $R$, with multiplicative unit $e$. If $R$ is regular, then so is $eRe$~\cite[II.II, Theorem~2.11, p.~77]{VonNeumannBook}. \end{remark}

\begin{lem}\label{lemma:right.multiplication} Let $R$ be a unital ring and let $e \in \E(R)$. The following hold. \begin{enumerate}
	\item\label{lemma:right.multiplication.1} If $I$ is a right ideal of $R$, then $Ie = I \cap Re$.
	\item\label{lemma:right.multiplication.2} If $I$ and $J$ are right ideals of $R$, then $(I \cap J)e = Ie \cap Je$.	
	\item\label{lemma:annihilator.eRe} Let $e \in \E(R)$ and $a \in R$. If $ae=ea$, then \begin{displaymath}
				\qquad e\rAnn(a) = \rAnn(ae) \cap eR, \qquad e\rAnn(a)e = \rAnn(ae) \cap eRe.
			\end{displaymath}	
\end{enumerate} \end{lem}

\begin{proof} \ref{lemma:right.multiplication.1} Let $I$ be a right ideal of $R$. As $x=xe \in Ie$ for every $x \in I \cap Re$, we~see that $I \cap Re \subseteq Ie$. Since $I$ is a right ideal of $R$, we have $Ie \subseteq I$, thus $Ie \subseteq I \cap Re$.
	
\ref{lemma:right.multiplication.2} If $I$ and $J$ are right ideals of $R$, then \begin{displaymath}
	(I \cap J)e \, \stackrel{\ref{lemma:right.multiplication.1}}{=} \, (I \cap J) \cap Re \, = \, (I \cap Re) \cap (J \cap Re) \, \stackrel{\ref{lemma:right.multiplication.1}}{=} \, Ie \cap Je .
\end{displaymath}

\ref{lemma:annihilator.eRe} Suppose that $ae=ea$. First let us show that $e\rAnn(a) = \rAnn(ae) \cap eR$. Evidently, if $x \in \rAnn(a)$, then $aeex = aex = eax = 0$, thus $ex \in \rAnn(ae) \cap eR$. Conversely, if $x \in \rAnn(ae) \cap eR$, then $ax = aex = 0$, hence $x = ex \in e\rAnn(a)$. This proves the first equality. We conclude that \begin{displaymath}
	e\rAnn(a)e \, \stackrel{\ref{lemma:right.multiplication.1}}{=} \, e\rAnn(a) \cap Re \, = \, \rAnn(ae) \cap eR \cap Re \, \stackrel{e \in \E(R)}{=} \, \rAnn(ae) \cap eRe . \qedhere
\end{displaymath} \end{proof}

\begin{lem}\label{lemma:right.multiplication.regular} Let $R$ be a regular ring and let $e \in \E(R)$. The following hold. \begin{enumerate}
	\item\label{lemma:right.multiplication.4} If $I \in \lat(R)$ and $I \subseteq eR$, then $Ie \in \lat(eRe)$.
	\item\label{lemma:right.multiplication.3} Suppose that $R$ is regular, let $n \in \N_{>0}$ and $I_{1},\ldots,I_{n} \in \lat(R)$. If $(I_{1},\ldots,I_{n})\perp$ and $I_{i} \subseteq eR$ for each $i \in \{ 1,\ldots,n \}$, then $(I_1e,\ldots,I_ne)\perp$.
\end{enumerate} \end{lem}

\begin{proof} \ref{lemma:right.multiplication.4} Let $I \in \lat(R)$ with $I \subseteq eR$. By~\cite[Lemma~7.5]{SchneiderGAFA}, there exists $f \in \E(R)$ with $f \leq e$ and $I = fR$, which implies that $Ie = fRe = (efe)eRe \in \lat(eRe)$. 
	
\ref{lemma:right.multiplication.3} If $(I_{1},\ldots,I_{n})\perp$ and $I_{i} \subseteq eR$ for each $i \in \{ 1,\ldots,n \}$, then \begin{displaymath}
	(I_{1}e + \ldots + I_{i}e) \cap I_{i+1}e \, \subseteq \, (I_{1} + \ldots + I_{i}) \cap I_{i+1} \, = \, \{ 0 \} 
\end{displaymath} for every $i \in \{ 1,\ldots,n-1 \}$, whence $(I_{1}e,\ldots,I_{n}e)\perp$ by Remark~\ref{remark:independence.equivalence.intersection}. \end{proof}	

Our later considerations, especially in Section~\ref{section:algebraic.elements} and Section~\ref{section:approximation.by.matrix.algebras}, rely on a study of embeddings of matrix algebras. We conclude this section with a clarifying remark on this matter. To recollect some terminology from~\cite[II.III, Definition~3.5, p.~97]{VonNeumannBook}, let $R$ be a unital ring and let $n \in \N_{>0}$. Then $s = (s_{ij})_{i,j\in\{1,\ldots,n\}} \in R^{n\times n}$ is called a \emph{family of matrix units} for $R$ if \begin{displaymath}
	\forall i,j,k,\ell\in\{1,\ldots,n\}\colon\qquad s_{ij}s_{k\ell} \, = \,
		\begin{cases}
			\, s_{i\ell} & \text{if } j=k, \\
			\, 0 & \text{otherwise}
		\end{cases}
\end{displaymath} and $s_{11}+\ldots+s_{nn}=1$.

\begin{remark}\label{remark:matrixunits.ring.homomorphism} Let $K$ be a field and let $n \in \N_{>0}$. Since the matrix ring $\M_{n}(K)$ is simple (see, e.g.,~\cite[IX.1, Corollary~1.5, p.~361]{GrilletBook}), any unital ring homomorphism from $\M_{n}(K)$ to a non-zero unital ring is an embedding. Let $R$ be a unital $K$-algebra. If $s\in R^{n\times n}$ is any family of matrix units for $R$, then \begin{align*}
	\M_{n}(K)\,\longrightarrow\, R,\quad (a_{ij})_{i,j\in\{1,\ldots,n\}}\,\longmapsto\, \sum\nolimits_{i,j=1}^{n}a_{ij}s_{ij} 
\end{align*} constitutes a unital $K$-algebra homomorphism. Using the standard basis of the $K$-vector space $\M_{n}(K)$, it is easy to see that this construction induces a bijection between the set of families of matrix units for $R$ in $R^{n \times n}$ and the set of unital ring homomorphisms from $\M_{n}(K)$ to $R$. \end{remark}

\section{Rank functions and continuous rings}\label{section:continuous.rings.and.rank.functions}

This section provides some background on von Neumann's work on continuous rings, regarding in particular their description via rank functions. To begin with, we recall some convenient terminology from~\cite[Chapter~16]{GoodearlBook}. Let $R$ be a regular ring. A \emph{pseudo-rank function} on $R$ is a map $\rho \colon R \to [0,1]$ such that \begin{enumerate}
	 \item[---\,] $\rho(1) = 1$,
	 \item[---\,] $\rho(ab) \leq \min\{\rho(a),\rho(b)\}$ for all $a,b \in R$, and
	 \item[---\,] $\rho(e+f) = \rho(e) + \rho(f)$ for any two orthogonal $e,f \in \E(R)$.
\end{enumerate} Of course, the third condition entails that $\rho(0) = 0$ for any pseudo-rank function $\rho$ on $R$. A \emph{rank function} on $R$ is a pseudo-rank function $\rho$ on $R$ such that $\rho(a)>0$ for each $a \in R\setminus\{0\}$. A \emph{rank ring} is a pair consisting of a regular ring and a rank function on it. 

\begin{remark}[\cite{VonNeumannBook}]\label{remark:properties.pseudo.rank.function} Let $\rho$ be a pseudo-rank function on a regular ring $R$. \begin{enumerate}
	\item\label{remark:rank.difference.smaller.idempotent} If $e,f\in\E(R)$ and $e\leq f$, then $e \perp (f-e) \in \E(R)$ by Remark~\ref{remark:quantum.logic}\ref{remark:quantum.logic.1} and therefore $\rho(f) = \rho(e+(f-e)) = \rho(e)+\rho(f-e)$, thus $\rho(f-e) = \rho(f)-\rho(e)$.
	\item\label{remark:rank.order.isomorphism} If $\rho$ is a rank function and $E$ is a chain in $(\E(R),{\leq})$, then~\ref{remark:rank.difference.smaller.idempotent} entails that \begin{displaymath}
				\qquad \forall e,f \in E \colon \quad e \leq f \ \Longleftrightarrow \ \rho(e) \leq \rho(f)
			\end{displaymath} and \begin{displaymath}
				\qquad \forall e,f \in E \colon \quad \lvert \rho(e)-\rho(f) \rvert = \rho(e-f).
			\end{displaymath}
	\item\label{remark:rank.matrixunits} Let $n \in \N_{>0}$. If $s \in R^{n\times n}$ is a family of matrix units for $R$, then $\rho(s_{ij}) = \tfrac{1}{n}$ for all $i,j \in \{1,\ldots,n\}$. This follows by straightforward calculation.
	\item\label{remark:inequation.sum.pseudo.rank} For all $a,b \in R$, \begin{displaymath}
				\qquad \rho(a+b) \, \leq \, \rho(a) + \rho(b).
			\end{displaymath} (Proofs of this are contained in~\cite[II.XVIII, p.~231, Corollary~$(\overline{f})$]{VonNeumannBook}, as well as in~\cite[VI.5, Hilfssatz~5.1(3°), p.~153]{MaedaBook}, and~\cite[Proposition~16.1(d), p.~227]{GoodearlBook}.)
	\item\label{remark:rank.continuous} The map \begin{align*}
				\qquad d_{\rho} \colon \, R \times R \, \longrightarrow \, [0,1], \quad (a,b) \, \longmapsto \, \rho(a-b)
			\end{align*} is a pseudo-metric on $R$. Evidently, $d_{\rho}$ is a metric if and only if $\rho$ is a rank function. Moreover, the map $\rho \colon R \to [0,1]$ is $1$-Lipschitz, hence continuous, with respect to $d_{\rho}$. (For proofs, we refer to~\cite[II.XVIII, Lemma~18.1, p.~231]{VonNeumannBook}, \cite[VI.5, Satz~5.1, p.~154]{MaedaBook}, or~\cite[Proposition~19.1, p.~282]{GoodearlBook}.)
	\item\label{remark:rank.group.topology} Equipped with the \emph{$\rho$-topology}, i.e., the topology generated by $d_{\rho}$, the ring $R$ constitutes a topological ring (see~\cite[Remark~7.8]{SchneiderGAFA}). It follows from~\ref{remark:rank.continuous} that the $\rho$-topology is Hausdorff if and only if $\rho$ is a rank function. Furthermore, $\GL(R)$ is a topological group with respect to the relative $\rho$-topology, as the latter is generated by the bi-invariant pseudo-metric ${d_{\rho}}\vert_{\GL(R)\times \GL(R)}$.
\end{enumerate} \end{remark}

The following remark recollects several general facts about rank functions and unit groups of the underlying rings from the literature.	

\begin{remark}\label{remark:properties.rank.function} Let $(R,\rho)$ be a rank ring. \begin{enumerate}
	\item\label{remark:directly.finite} $R$ is \emph{directly finite}, i.e., \begin{displaymath}
			\qquad \forall a,b \in R \colon \quad ab=1 \ \Longrightarrow \ ba=1 .
		\end{displaymath} (For proofs, we refer to~\cite[Corollary~6(1)]{Handelman76}, \cite[Proposition~16.11(b), p.~234]{GoodearlBook}, or~\cite[Lemma~7.13(2)]{SchneiderGAFA}.)
	\item\label{remark:invertible.rank} $\GL(R) = \{ a \in R \mid \rho(a) = 1 \}$. (For a proof, see~\cite[Lemma~7.13(3)]{SchneiderGAFA}.)
	\item\label{remark:GL(R).closed} $\GL(R)$ is closed in $(R,d_{\rho})$ by~\ref{remark:invertible.rank} and Remark~\ref{remark:properties.pseudo.rank.function}\ref{remark:rank.continuous}.
\end{enumerate} \end{remark} 		

For later use, we record some additional basic observations. 

\begin{lem}[\cite{MaedaBook,SchneiderGAFA}]\label{lemma:pseudo.dimension.function} Let $\rho$ be a pseudo-rank function on a regular ring $R$. Then \begin{displaymath}
	\delta_{\rho} \colon \, \lat(R) \, \longrightarrow \, [0,1], \quad aR \, \longmapsto \, \rho(a)
\end{displaymath} is a pseudo-dimension function. Moreover, the following hold. \begin{enumerate}
	\item\label{lemma:rank.dimension.function} $\rho$ rank function on $R$ $\ \Longleftrightarrow\ $ $\delta_{\rho}$ dimension function on $\lat(R)$. 
	\item\label{lemma:rank.dimension.annihilator} $\rho(a)=1-\delta_{\rho}(\rAnn(a))$ for every $a \in R$.
	\item\label{lemma:multplication.Lipschitz} For every $a \in R$, the mapping $\lat(R) \to \lat(R),\, I \mapsto aI$ is $1$-Lipschitz, hence continuous, with respect to $d_{\delta_{\rho}}$.
\end{enumerate} \end{lem}
	
\begin{proof} The fact that $\delta \defeq \delta_{\rho}$ is a pseudo-dimension function and assertion~\ref{lemma:rank.dimension.function} follow by the proof of \cite[VI.5, Satz~5.2, p.~154]{MaedaBook}.
		
\ref{lemma:rank.dimension.annihilator} Let $a \in R$. Since $R$ is regular, Remark~\ref{remark:regular.idempotent.ideals} asserts the existence of $e\in\E(R)$ such that $Ra=Re$. Hence, \begin{displaymath}
	\rAnn(a) \, = \, \rAnn(Ra) \, = \, \rAnn(Re) \, \stackrel{\ref{remark:bijection.annihilator}}{=} \, (1-e)R
\end{displaymath} and therefore \begin{displaymath}
	\rho(a) \, = \, \rho(e) \, \stackrel{\ref{remark:properties.pseudo.rank.function}\ref{remark:rank.difference.smaller.idempotent}}{=} \, 1-\rho(1-e) \, = \, 1-\delta((1-e)R) \, = \, 1-\delta(\rAnn(a)) .
\end{displaymath}
		
\ref{lemma:multplication.Lipschitz} The argument follows the lines of~\cite[Proof of Lemma~9.10(2)]{SchneiderGAFA}. Let $a \in R$. If $I, J \in \lat(R)$ and $I\subseteq J$, then Remark~\ref{remark:bijection.annihilator} and Lemma~\ref{lemma:complement} together assert the existence of $I' \in \lat(R)$ such that $J = I\oplus I'$, whence \begin{align*}
	d_{\delta}(aI,aJ) \, &= \, \delta(aJ)-\delta(aI) \, \leq \, \delta(aI)+\delta(aI')-\delta(aI) \\
			&\leq \, \delta(I') \, = \, \delta(J)-\delta(I) \, = \, d_{\delta}(I,J) .
\end{align*} Consequently, for all $I,J \in \lat(R)$, \begin{align*}
	d_{\delta}(aI,aJ) \, &\leq \, d_{\delta}(aI,a(I+J)) + d_{\delta}(a(I+J),aJ) \\
			&\leq \, d_{\delta}(I,I+J) + d_{\delta}(I+J,J) \, = \, 2\delta(I+J)-\delta(I)-\delta(J) \stackrel{(\ast)}{=} \, d_{\delta}(I,J),
\end{align*} where $(\ast)$ follows by~\cite[Lemma~6.3(3)]{SchneiderGAFA}. \end{proof}

We now turn to von Neumann's continuous rings. A \emph{continuous ring} is a regular ring $R$ such that the lattice $\lat(R)$ is a continuous geometry.

\begin{thm}[\cite{VonNeumannBook}]\label{theorem:unique.rank.function} Let $R$ be an irreducible, continuous ring. Then \begin{displaymath}
	\rk_{R} \colon \, R \, \longrightarrow \, [0,1] , \quad a \, \longmapsto \, \delta_{\lat(R)}(aR).
\end{displaymath} is the unique rank function on $R$. Moreover, the metric $d_{R} \defeq d_{\rk_{R}}$ is complete. \end{thm}

\begin{proof} By~\cite[II.XVII, Theorem~17.1, p.~224]{VonNeumannBook} and~\cite[II.XVII, Theorem~17.2, p.~226]{VonNeumannBook}, the map $\rk_{R}$ is the unique rank function on $R$. The metric space $(R,d_{R})$ is complete according to~\cite[II.XVII, Theorem~17.4, p.~230]{VonNeumannBook}. (Alternatively, proofs may be found in~\cite[VII.2, pp.~162--165]{MaedaBook}.) \end{proof}

A rank ring $(R,\rho)$ is called \emph{complete} if the metric space $(R,d_{\rho})$ is complete. If $R$ is an irreducible, continuous ring, then $(R,\rk_{R})$ is a complete rank ring by Theorem~\ref{theorem:unique.rank.function}. Conversely, if $(R,\rho)$ is a complete rank ring, then $R$ is a continuous ring according to~\cite[VI.5, Satz~5.3, p.~156]{MaedaBook} (see also~\cite[II.XVIII, Proof of Theorem~18.1, p.~237]{VonNeumannBook}).

The remainder of this section records several useful facts about the rank function of an irreducible, continuous ring.

\begin{remark}\label{remark:rank.function.general} Let $R$ be an irreducible, continuous ring. \begin{enumerate}
	\item\label{remark:characterization.discrete} The work of von Neumann~\cite{VonNeumannBook} implies that the following are equivalent. \begin{enumerate}
				\item[---\,] $R$ is \emph{discrete}, i.e., the topology generated by~$d_{R}$ is discrete.
				\item[---\,] $R\cong \M_{n}(D)$ for some division ring $D$ and $n \in \N_{>0}$.
				\item[---\,] $\rk_{R}(R) \ne [0,1]$.
		\end{enumerate} For a proof of this, see~\cite[Remark~3.4]{SchneiderIMRN} and~\cite[Remark~3.6]{SchneiderIMRN}.
	\item\label{remark:uniqueness.rank.embedding} Suppose that $\rho$ is a pseudo-rank function on a regular ring $S$ and let $\phi \colon R \to S$ be a unital ring homomorphism. Then $\rho \circ \phi$ is a pseudo-rank function on~$R$, so $(\rho \circ \phi)^{-1}(\{ 0 \})$ is a proper two-sided ideal of $R$ by~\cite[Proposition~16.7(a), p.~231]{GoodearlBook}. Since the ring $R$ is simple according to~\cite[VII.3, Hilfssatz~3.1, p.~166]{MaedaBook} (see also~\cite[Corollary~13.26, p.~170]{GoodearlBook}), it follows that $(\rho \circ \phi)^{-1}(\{ 0\}) = \{ 0\}$, i.e., $\rho \circ \phi$ is a rank function on $R$. Thus, $\rho \circ \phi = \rk_{R}$ by Theorem~\ref{theorem:unique.rank.function}, whence \begin{displaymath}
		\qquad \forall a,b \in R \colon \quad d_{R}(a,b) \, = \, d_{\rho}(\phi(a),\phi(b)) .
	\end{displaymath}
	\item\label{remark:duality} By Remark~\ref{remark:bijection.annihilator} and Remark~\ref{remark:irreducible.center.field}, both $\lat(R)$ and $\latop(R)$ are irreducible continuous geometries. Moreover, by~\cite[II.XVII, Lemma~17.2, p.~223]{VonNeumannBook}, \begin{displaymath}
			\qquad \forall I \in \latop(R) \colon \quad \delta_{\latop(R)}(I) \, = \, 1-\delta_{\lat(R)}(\rAnn(I)) .
		\end{displaymath}
\end{enumerate} \end{remark}

\begin{lem}[\cite{VonNeumannBook}]\label{lemma:matrixunits.idempotents} Let $R$ be an irreducible, continuous ring, let $n \in \N_{>0}$, and let $e_{1},\ldots,e_{n} \in \E(R)$ be pairwise orthogonal with $e_{1} + \ldots + e_{n} = 1$ and $\rk_{R}(e_{1}) = \rk_{R}(e_{i})$ for each $i \in \{ 1,\ldots,n\}$. Then there exists a family of matrix units $s\in R^{n\times n}$ for $R$ such that $s_{ii} = e_{i}$ for each $i\in\{1,\ldots,n\}$. \end{lem}

\begin{proof} Note that $R = e_{1}R \oplus \ldots \oplus e_{n}R$ by Remark~\ref{remark:independence.ideals.idempotents} and, for each $i \in \{ 1,\ldots, n\}$, \begin{displaymath}
	\delta_{\lat(R)}(e_{1}R) \, \stackrel{\ref{theorem:unique.rank.function}}{=} \, \rk_{R}(e_{1}) \, = \, \rk_{R}(e_{i}) \, \stackrel{\ref{theorem:unique.rank.function}}{=} \, \delta_{\lat(R)}(e_{i}R) .
\end{displaymath} Thus, by~\cite[II.III, Lemma~3.6, p.~97]{VonNeumannBook} and~\cite[I.VI, Theorem~6.9(iii)'', p.~52]{VonNeumannBook}, there is a family of matrix units $s\in R^{n\times n}$ for $R$ such that $e_{i}R = s_{ii}R$ for each $i \in \{1,\ldots,n\}$. Using the argument from~\cite[II.II, Corollary~(iv) after Definition~2.1, p.~69]{VonNeumannBook}, we see that, for each $i \in \{1,\ldots,n\}$, \begin{align*}
	s_{ii}(1-e_{i}) \, &= \, s_{ii}(e_{1} + \ldots + e_{i-1} + e_{i+1} + \ldots + e_{n}) \\
		&= \, s_{ii}e_{1} + \ldots + s_{ii}e_{i-1} + s_{ii}e_{i+1} + \ldots + s_{ii}e_{n} \\
		&= \, s_{ii}s_{11}e_{1} + \ldots + s_{ii}s_{i-1,\,i-1}e_{i-1} + s_{ii}s_{i+1,\,i+1}e_{i+1} + \ldots + s_{ii}s_{nn}e_{n} \, = \, 0,
\end{align*} i.e., $s_{ii} = s_{ii}e_{i} = e_{i}$, as desired. \end{proof}

\begin{lem}\label{lemma:order} Let $R$ be an irreducible, continuous ring and let $e \in \E(R)$. \begin{enumerate}
	\item\label{lemma:order.1} For every $e' \in \E(R)$ with $e \leq e'$ and every $t \in [\rk_{R}(e),\rk_{R}(e')] \cap \rk_{R}(R)$, there exists $f \in \E(R)$ such that $\rk_{R}(f) = t$ and $e \leq f \leq e'$. 
	\item\label{lemma:order.2} If $R$ is non-discrete and $\rk_{R}(e) = \tfrac{1}{2}$, then we find $(e_{n})_{n\in\N_{>0}} \in \E(R)^{\N_{>0}}$ pairwise orthogonal such that $e_{1} = e$ and $\rk_{R}(e_{n}) = 2^{-n}$ for each $n\in\N_{>0}$.
\end{enumerate} \end{lem}

\begin{proof} Let $e' \in \E(R)$ with $e \leq e'$. The Hausdorff maximal principle asserts the existence of a maximal chain $E$ in $(\E(R),{\leq})$ such that $\{ e,e' \} \subseteq E$. By~\cite[Corollary~7.19]{SchneiderGAFA}, $\rk_{R}(E) = \rk_{R}(R)$. We deduce from Remark~\ref{remark:properties.pseudo.rank.function}\ref{remark:rank.order.isomorphism} that ${{\rk_{R}}\vert_{E}} \colon (E,{\leq}) \to (\rk_{R}(R),{\leq})$ is an isomorphism of linearly ordered sets.
	
\ref{lemma:order.1} Now, let $t \in [\rk_{R}(e),\rk_{R}(e')] \cap \rk_{R}(R)$. We define $f \defeq ({{\rk_{R}}\vert_{E}})^{-1}(t)$. Then, from $\rk_{R}(e) \leq t \leq \rk_{R}(e')$, we infer that $e \leq f \leq e'$.

\ref{lemma:order.2} Let $e' \defeq 1$. Suppose that $R$ is non-discrete and $\rk_{R}(e) = \tfrac{1}{2}$. Note that $\rk_{R}(R) = [0,1]$ by Remark~\ref{remark:rank.function.general}\ref{remark:characterization.discrete}. Considering \begin{displaymath}
	f_{n} \, \defeq \, ({\rk_{R}}\vert_{E})^{-1}\!\left(\tfrac{2^{n}-1}{2^{n}}\right) \qquad (n\in\N),
\end{displaymath} we conclude that $(f_{n})_{n \in \N}$ is a monotone sequence in $(\E(R),\leq)$ with $f_{0} = 0$ and~$f_{1} = e$. For each $n \in \N_{>0}$, let us define $e_{n} \defeq f_{n}-f_{n-1}$ and observe that \begin{displaymath}
	\rk_{R}(e_{n}) \, \stackrel{\ref{remark:properties.pseudo.rank.function}\ref{remark:rank.difference.smaller.idempotent}}{=} \, \rk_{R}(f_{n}) - \rk_{R}(f_{n-1}) \, = \, \tfrac{2^{n}-1}{2^{n}}-\tfrac{2^{n-1}-1}{2^{n-1}} \, = \, 2^{-n}.
\end{displaymath} Evidently, $e_{1} = f_{1}-f_{0} = e$. Moreover, for any two $m,n\in\N_{>0}$ with $m>n$, \begin{displaymath}
	e_{m}e_{n} \, = \, f_{m}(1-f_{m-1})f_{n}(1-f_{n-1}) \, = \, f_{m}(f_{n}-f_{n})(1-f_{n-1}) \, = \, 0 . \qedhere
\end{displaymath} \end{proof}
	
Finally, we turn once more to the construction mentioned in Remark~\ref{remark:corner.rings}. 

\begin{remark}\label{remark:eRe.non-discrete.irreducible.continuous} Let $R$ be a continuous ring and let $e\in\E(R)$. Then $eRe$ is continuous due to~\cite[Proposition~13.7, p.~162]{GoodearlBook}. Suppose now that $R$ is irreducible and $e \ne 0$. Then $\ZZ(eRe) = \ZZ(R)e$ by~\cite[Lemma~2.1]{Halperin62}, and therefore $\ZZ(R) \cong \ZZ(eRe)$.\footnote{We usually identify $\ZZ(eRe)$ with $\ZZ(R)$.} In turn, $eRe$ is irreducible by Remark~\ref{remark:irreducible.center.field}. Moreover, using Theorem~\ref{theorem:unique.rank.function}, we see that \begin{displaymath}
	{\rk_{eRe}} \, = \, \tfrac{1}{\rk_{R}(e)}{{\rk_{R}}\vert_{eRe}} .
\end{displaymath} In particular, if $R$ is non-discrete, then it follows by Remark~\ref{remark:rank.function.general}\ref{remark:characterization.discrete} and Lemma~\ref{lemma:order}\ref{lemma:order.1} that $eRe$ is non-discrete, too. \end{remark}

\section{Subgroups induced by idempotents}\label{section:subgroups.of.the.form.Gamma_e(R)}

In this section, we isolate and study a certain natural family of subgroups of the unit group of a unital ring. Our observations about these subgroups will be fundamental to both Section~\ref{section:decomposition.into.locally.special.elements} and Section~\ref{section:steinhaus.property}. The construction, detailed in Lemma~\ref{lemma:subgroup.unit.group}, is based on the following type of ring embeddings.

\begin{lem}\label{lemma:sum.embedding} Let $R$ be a unital ring, let $n \in \N$, let $e_{1},\ldots,e_{n} \in \E(R)$ be pairwise orthogonal and $e \defeq e_{1} + \ldots + e_{n}$. Then \begin{displaymath}
	\prod\nolimits_{i=1}^{n} e_{i}Re_{i} \, \longrightarrow \, eRe, \quad (x_{1},\dots,x_{n}) \,\longmapsto \, x_{1}+\ldots+x_{n}
\end{displaymath} is a unital $\ZZ(R)$-algebra embedding. Moreover, the following hold. \begin{enumerate}
	\item\label{lemma:sum.embedding.3} If $a \in \prod\nolimits_{i=1}^{n} e_{i}Re_{i}$, then \begin{displaymath}
		\qquad \sum\nolimits_{i=1}^{n} a_{i} + 1 - \sum\nolimits_{i = 1}^{n} e_{i} \, = \, \prod\nolimits_{i=1}^{n} a_{i}+1-e_{i} ,
	\end{displaymath} where the factors in the product commute.
	\item\label{lemma:sum.embedding.2} Suppose that $R$ is regular ring and let $\rho$ be a pseudo-rank function on $R$. If $a_{1} \in e_{1}Re_{1},\dots,a_{n}\in e_{n}Re_{n}$, then $\rho(a_{1}+\dots+a_{n}) = \rho(a_{1})+\dots+\rho(a_{n})$.
\end{enumerate} \end{lem}

\begin{proof} It is straightforward to check that $\prod\nolimits_{i=1}^{n} e_{i}Re_{i} \to eRe, \, x \mapsto x_{1}+\ldots+x_{n}$ is a unital $\ZZ(R)$-algebra homomorphism. Furthermore, this mapping is injective: indeed, if $x \in \prod\nolimits_{i=1}^{n} e_{i}Re_{i}$ and $x_{1}+\ldots+x_{n}=0$, then it follows that $x_{i} = e_{i}(x_{1}+\ldots +x_n) = 0$ for every $i \in \{ 1,\ldots,n \}$, i.e., $x=0$.

\ref{lemma:sum.embedding.3} Let $a \in \prod\nolimits_{i=1}^{n} e_{i}Re_{i}$. For any two distinct $i,j \in \{ 1,\ldots,n\}$, from $e_{i} \perp e_{j}$ we infer that \begin{displaymath}
	(a_{i}+1-e_{i})(a_{j}+1-e_{j}) \, = \, a_{i}+a_{j}+1-e_{i}-e_{j} \, = \, (a_{j}+1-e_{j})(a_{i}+1-e_{i}) .
\end{displaymath} Hence, the factors in the product commute. By induction, we show that \begin{displaymath}
	\forall j \in \{ 1,\ldots,n \} \colon \qquad \prod\nolimits_{i=1}^{j} a_{i}+1-e_{i} = \sum\nolimits_{i=1}^{j} a_{i} + 1 - \sum\nolimits_{i = 1}^{j} e_{i} .
\end{displaymath} For $j=1$, the statement is trivial. Now, if $\prod\nolimits_{i=1}^{j} a_{i}+1-e_{i} = \sum\nolimits_{i=1}^{j} a_{i} + 1 - \sum\nolimits_{i = 1}^{j} e_{i}$ for some $j \in \{ 1,\ldots,n-1\}$, then from $\bigl(\sum_{i=1}^{j} e_{i}\bigr) \perp e_{j+1}$ we deduce that \begin{align*}
	\prod\nolimits_{i=1}^{j+1}a_{i}+1-e_{i} \, &= \, \! \left(\sum\nolimits_{i=1}^{j}a_{i} +1 - \sum\nolimits_{i=1}^{j}e_{i}\right)\!(a_{j+1}+1-e_{j+1})\\
	&=\, \sum\nolimits_{i=1}^{j+1}a_{i}+1-\sum\nolimits_{i=1}^{j+1}e_{i} .
\end{align*} This completes the induction, which entails the desired statement for $j=n$.

\ref{lemma:sum.embedding.2} Since $R$ is regular, for each $i \in \{1,\ldots,n\}$, we may choose an element $b_{i} \in R$ with $a_{i}b_{i}a_{i} = a_{i}$ and define $f_{i} \defeq a_{i}b_{i}e_{i}$. We observe that, for every $i \in \{ 1,\ldots,n \}$, \begin{displaymath}
	f_{i}f_{i} \, = \, a_{i}b_{i}e_{i}a_{i}b_{i}e_{i} \, = \, a_{i}b_{i}a_{i}b_{i}e_{i} \, = \, a_{i}b_{i}e_{i} \, = \, f_{i}
\end{displaymath} and \begin{displaymath}
	a_{i}R \, = \, a_{i}b_{i}a_{i}R \, = \, a_{i}b_{i}e_{i}a_{i}R \, = \, f_{i}a_{i}R \, \subseteq \, f_{i}R \, = \, a_{i}b_{i}e_{i}R \, \subseteq \, a_{i}R,
\end{displaymath} that is, $a_{i}R = f_{i}R$. Moreover, $f_{1},\dots,f_{n}$ are pairwise orthogonal: if $i,j \in \{ 1,\ldots,n \}$ are distinct, then \begin{displaymath}
	f_{i}f_{j} \, = \, a_{i}b_{i}e_{i}a_{j}b_{j}e_{j} \, = \, a_{i}b_{i}e_{i}e_{j}a_{j}b_{j}e_{j} \, = \, 0 .
\end{displaymath} Also, \begin{align*}
	(a_{1} + \ldots + a_{n})(e_{1}b_{1}e_{1} + \ldots + e_{n}b_{n}e_{n}) \, &= \, a_{1}e_{1}b_{1}e_{1} + \ldots + a_{n}e_{n}b_{n}e_{n} \\
	&= \, a_{1}b_{1}e_{1} + \ldots + a_{n}b_{n}e_{n} \, = \, f_{1} + \ldots + f_{n} .		
\end{align*} We conclude that \begin{align*}
	\rho(a_{1}) + \ldots + \rho(a_{n}) \, &= \, \rho(f_{1}) + \dots + \rho(f_{n}) \, = \, \rho(f_{1} + \ldots + f_{n}) \\
	&= \, \rho((a_{1} + \ldots + a_{n})(e_{1}b_{1}e_{1} + \ldots + e_{n}b_{n}e_{n})) \\
	&\leq \, \rho(a_{1} + \ldots + a_{n}) \, \stackrel{\ref{remark:properties.pseudo.rank.function}\ref{remark:inequation.sum.pseudo.rank}}{\leq} \, \rho(a_{1}) + \ldots + \rho(a_{n}),
\end{align*} thus $\rho(a_{1}+\ldots+a_{n}) = \rho(a_{1})+\ldots+\rho(a_{n})$. \end{proof}

We arrive at the announced construction of subgroups.

\begin{lem}\label{lemma:subgroup.unit.group} Let $R$ be a unital ring and let $e,f \in \E(R)$. Then \begin{displaymath}
	\Gamma_{R}(e) \, \defeq \, \GL(eRe) + 1-e \, = \, \GL(R) \cap (eRe + 1-e)
\end{displaymath} is a subgroup of $\GL(R)$ and \begin{align*}
	&{\GL(eRe)} \, \longrightarrow \, \Gamma_{R}(e),\quad a \, \longmapsto \, a+1-e, \\
	&{\Gamma_{R}(e)} \, \longrightarrow \, \GL(eRe), \quad a \, \longmapsto \, ae
\end{align*} are mutually inverse group isomorphisms. Moreover, the following hold. \begin{enumerate}
	\item\label{lemma:subgroup.unit.group.order} If $e \leq f$, then $\Gamma_{R}(e)\leq\Gamma_{R}(f)$.
	\item\label{lemma:subgroup.unit.group.orthogonal} If $e \perp f$, then $ab=ba$ for all $a\in\Gamma_{R}(e)$ and $b\in\Gamma_{R}(f)$.
	\item\label{lemma:subgroup.unit.group.conjugation} If $a\in \GL(R)$, then $a\Gamma_{R}(e)a^{-1}=\Gamma_R(aea^{-1})$.
\end{enumerate} \end{lem}

\begin{proof} Note that $e\perp (1-e) \in \E(R)$ according to Remark~\ref{remark:quantum.logic}\ref{remark:quantum.logic.1}. Consider the unital ring $S \defeq eRe \times (1-e)R(1-e)$ and observe that \begin{displaymath}
	\GL(S) \, = \, \GL(eRe) \times \GL((1-e)R(1-e)) .
\end{displaymath} In turn, $\psi \colon \GL(eRe) \to \GL(S), \, a \mapsto (a,1-e)$ is a well-defined group embedding. Moreover, it follows from Lemma~\ref{lemma:sum.embedding} that \begin{displaymath}
	\phi \colon \, \GL(S) \, \longrightarrow \, \GL(R), \quad (a,b) \, \longmapsto \, a+b
\end{displaymath} is a well-defined embedding. Thus, $\phi \circ \psi \colon \GL(eRe) \to \GL(R)$ is an embedding with \begin{displaymath}
	\Gamma_{R}(e) \, = \, \im (\phi \circ \psi) \, \subseteq \, \GL(R) \cap (eRe + 1-e) .
\end{displaymath} Now, if $a \in \GL(R) \cap (eRe + 1-e)$, then $ae=eae=ea$ and $a(1-e)=1-e=(1-e)a$, thus $a^{-1}e=ea^{-1}$ and $a^{-1}(1-e) = 1-e = (1-e)a^{-1}$, which implies that \begin{displaymath}
	ea^{-1}eeae \, = \, ea^{-1}ae \, = \, ee \, = \, e \, = \, ee \, = \, eaa^{-1}e \, = \, eaeea^{-1}e
\end{displaymath} and therefore $eae \in \GL(eRe)$, whence $a = eae + 1-e \in \Gamma_{R}(e)$. This shows that indeed $\Gamma_{R}(e) = \GL(R) \cap (eRe + 1-e)$, as claimed. Since ${\GL(eRe)} \to \Gamma_{R}(e),\, a \mapsto a+1-e$ is an isomorphism, its inverse ${\Gamma_{R}(e)} \to \GL(eRe), \, a \mapsto ae$ is an isomorphism, too.

\ref{lemma:subgroup.unit.group.order} If $e\leq f$, then \begin{displaymath}
	eRe+1-e \, = \, eRe+f-e+1-f \, = \, f(eRe+1-e)f+1-f \, \subseteq \, fRf+1-f
\end{displaymath} and therefore \begin{displaymath}
	\Gamma_{R}(e) \, = \, \GL(R) \cap (eRe+1-e) \, \subseteq \, \GL(R) \cap (fRf+1-f) \, = \, \Gamma_{R}(f).
\end{displaymath} 

\ref{lemma:subgroup.unit.group.orthogonal} This follows from Lemma~\ref{lemma:sum.embedding}\ref{lemma:sum.embedding.3}.

\ref{lemma:subgroup.unit.group.conjugation} If $a\in \GL(R)$, then \begin{align*}
	a\Gamma_{R}(e)a^{-1}\! \, &= \, a(\GL(R) \cap (eRe+1-e))a^{-1} \\
	& = \, \GL(R) \cap {\left(aea^{-1}Raea^{-1}+1-aea^{-1}\right)} \, = \, \Gamma_{R}\!\left( aea^{-1} \right)\! .\qedhere
\end{align*} \end{proof}
	
We point out the following result for the proof of Lemma~\ref{lemma:invertible.rank.idempotent} and Lemma~\ref{lemma:GL(R).covered.by.c_nW}.

\begin{lem}\label{lemma:Gamma.annihilator.right-ideal} Let $R$ be a unital ring, let $a \in R$ and $e, f \in \E(R)$ be such that $1-f \leq e$, $f \in \rAnn(1-a)$, and $(1-a)R \subseteq eR$. Then $a\in eRe+1-e$. Moreover, if $a\in \GL(R)$, then $a\in \Gamma_{R}(e)$. \end{lem}

\begin{proof} Since $(1-a)R \subseteq eR$, we see that \begin{equation}\label{eq--20}
	a \, = \, ea+(1-e)a \, = \, ea-e+e(1-a)+a \, = \, ea-e+1-a+a \, = \, ea+1-e.
\end{equation} Moreover, as $1-f\leq e$ (i.e., $1-e\leq f$) and $(1-a)f=0$ (i.e., $af=f$), \begin{equation}\label{eq--21}
	a \, = \, ae+a(1-e) \, = \, ae+af(1-e) \, = \, ae+1-e.
\end{equation} Thus, we conclude that \begin{equation*}
	a \, \stackrel{\eqref{eq--20}}{=} \, ea+1-e \, \stackrel{\eqref{eq--21}}{=} \, e(ae+1-e)+1-e \, = \, eae+1-e.
\end{equation*} Finally, if $a\in \GL(R)$, then $a\in \Gamma_{R}(e)$ by Lemma~\ref{lemma:subgroup.unit.group}. \end{proof}

\begin{lem}\label{lemma:invertible.rank.idempotent} Let $\rho$ be a pseudo-rank function on a regular ring $R$, and let $e \in \E(R)$. For every $a\in \Gamma_{R}(e)$, there is $f \in \E(R)$ with $f \leq e$, $a \in \Gamma_{R}(f)$ and $\rho(f) \leq 2\rho(1-a)$. \end{lem}

\begin{proof} Let $a\in \Gamma_{R}(e)$. Then $\rAnn(e-a) \subseteq eR$: indeed, if $x \in \rAnn(e-a)$, then \begin{displaymath}
	x \, = \, ex + (1-e)x \, \stackrel{a \in \Gamma_{R}(e)}{=} \, ex + (1-e)ax \, \stackrel{(e-a)x=0}{=} \, ex + (1-e)ex \, = \, ex \, \in \, eR .
\end{displaymath} Thanks to Remark~\ref{remark:bijection.annihilator} and~\cite[Lemma~7.5]{SchneiderGAFA}, there exists $e'\in \E(R)$ such that $e' \leq e$ and $e'R=\rAnn(e-a)$. Note that $\E(R) \ni e-e' \leq e$ due to Remark~\ref{remark:quantum.logic}\ref{remark:quantum.logic.1}, in particular $(e-e')R \subseteq eR$. Furthermore, $1-a = e-eae \in eR$ and therefore $(1-a)R \subseteq eR$. Hence, $(e-e')R+(1-a)R \subseteq eR$. By Remark~\ref{remark:bijection.annihilator} and~\cite[Lemma~7.5]{SchneiderGAFA}, we thus find $f\in \E(R)$ with $e-e' \leq f \leq e$ and $fR=(e-e')R+(1-a)R$. We conclude that \begin{align*}
	\rho(e-e') \, &\stackrel{\ref{remark:properties.pseudo.rank.function}\ref{remark:rank.difference.smaller.idempotent}}{=} \, \rho(e)-\rho(e') \, \stackrel{\ref{lemma:pseudo.dimension.function}\ref{lemma:rank.dimension.annihilator}}{=} \, \rho(e)-1+\rho(e-a) \\
	&\stackrel{\ref{remark:properties.pseudo.rank.function}\ref{remark:rank.difference.smaller.idempotent}}{=} \, \rho(e-a)-\rho(1-e) \, = \, \rho(e(1-a)e - (1-e))-\rho(1-e) \\
	&\stackrel{\ref{lemma:sum.embedding}\ref{lemma:sum.embedding.2}}{=} \, \rho(e(1-a)e) + \rho(-(1-e)) - \rho(1-e) \, = \, \rho(e(1-a)e) \, \leq \, \rho(1-a)
\end{align*} and therefore \begin{displaymath}
	\rho(f) \, \leq \, \rho(e-e')+\rho(1-a) \, \leq \, 2\rho(1-a).
\end{displaymath} Finally, we see that $f' \defeq e'+1-e \in \E(R)$ by Remark~\ref{remark:quantum.logic} and \begin{displaymath}
	(1-a)f' \, = \, (1-a)e'+(1-e)-a(1-e) \, = \, (1-e)-(1-e) \, = \, 0, 
\end{displaymath} thus $f' \in \rAnn(1-a)$. Moreover, $1-f' = e-e' \leq f$ and $(1-a)R \subseteq fR$. Therefore, $a\in \Gamma_{R}(f)$ thanks to Lemma~\ref{lemma:Gamma.annihilator.right-ideal}. \end{proof}

We conclude this section with a discussion of some general matters of convergence in complete rank rings, which will turn out useful later on in Sections~\ref{section:decomposition.into.locally.special.elements} and~\ref{section:steinhaus.property}.

\begin{remark} Let $(R,\rho)$ be a rank ring. If $I$ is any set and $(e_{i})_{i \in I}$ is a family of pairwise orthogonal elements of $\E(R)$, then \begin{displaymath}
	\forall n \in \N_{>0} \colon \qquad \left\lvert \left\{ i \in I \left\vert \, \rho(e_{i}) \geq \tfrac{1}{n} \right\} \right\rvert \, \leq \, n , \right.
\end{displaymath} thus $\{ i \in I \mid e_{i} \ne 0 \} = \bigcup_{n=1}^{\infty} \!\left\{ i \in I \left\vert \, \rho(e_{i}) \geq \tfrac{1}{n} \right\} \right.$ is countable. \end{remark}

Recall that, for any set $I$, the set $\Pfin(I)$ of all finite subsets of $I$, equipped with the partial order defined by inclusion, constitutes a directed set.

\begin{lem}\label{lemma:convergence.sequences} Let $(R,\rho)$ be a complete rank ring, let $I$ be a set, and let $(e_{i})_{i \in I}$ be a family of pairwise orthogonal elements of $\E(R)$. Then \begin{displaymath}
	\prod\nolimits_{i \in I} e_{i}Re_{i} \, \longrightarrow \, R, \quad (a_{i})_{i \in I} \, \longmapsto \, \sum\nolimits_{i \in I} a_{i} \defeq \lim\nolimits_{F \to \Pfin(I)} \sum\nolimits_{i \in F} a_{i}
\end{displaymath} is a well-defined $\ZZ(R)$-algebra embedding and, for every $(a_{i})_{i \in I} \in \prod\nolimits_{i \in I} e_{i}Re_{i}$,\footnote{According to Lemma~\ref{lemma:sum.embedding}\ref{lemma:sum.embedding.3}, there is no need to specify an order for the factors in the product.} \begin{displaymath}
	\sum\nolimits_{i \in I} a_{i} + 1 - \sum\nolimits_{i \in I} e_{i} \, = \, \lim\nolimits_{F \to \Pfin(I)} \prod\nolimits_{i \in F} a_{i}+1-e_{i} .
\end{displaymath} In particular, the following hold. \begin{enumerate}
	\item\label{lemma:convergence.idempotent} $\sum\nolimits_{i \in I} e_{i} \in \E(R)$.
	\item\label{lemma:convergence.monotone} If $J \subseteq J' \subseteq I$, then $\sum\nolimits_{i \in J} e_{i} \leq \sum\nolimits_{i \in J'} e_{i}$.
	\item\label{lemma:convergence.orthogonal} If $J,J' \subseteq I$ and $J \cap J' = \emptyset$, then $\left( \sum\nolimits_{i \in J} e_{i}\right)\! \perp \! \left( \sum\nolimits_{i \in J'} e_{i}\right)$.
\end{enumerate} \end{lem}
	
\begin{proof} To prove that the map is well-defined, let $(a_{i})_{i \in I} \in \prod\nolimits_{i \in I} e_{i}Re_{i}$ and consider \begin{displaymath}
	s \, \defeq \, \sup\nolimits_{F \in \Pfin(I)} \rho\!\left(\sum\nolimits_{i \in F} a_{i}\right)\! \, \in \, [0,1] .
\end{displaymath} We will prove that the net $(\sum_{i \in F} a_{i})_{F \in \Pfin(I)}$ is a Cauchy in $(R,d_{\rho})$. For this purpose, let $\epsilon \in \R_{>0}$. Then there exists some $F_{0} \in \Pfin(I)$ such that $s \leq \rho(\sum_{i \in F_{0}} a_{i}) + \epsilon$. Thus, if $F,F' \in \Pfin(I)$ and $F_{0} \subseteq F \cap F'$, then \begin{displaymath}
	\rho\!\left( \sum\nolimits_{i \in F \mathbin{\triangle} F'} a_{i}\right) \! + \rho\!\left( \sum\nolimits_{i \in F_{0}} a_{i}\right) \! \, \stackrel{\ref{lemma:sum.embedding}\ref{lemma:sum.embedding.2}}{=} \, \rho\!\left(\sum\nolimits_{i \in (F \mathbin{\triangle} F') \cup F_{0}} a_{i}\right)\! \, \leq \, s
\end{displaymath} and therefore \begin{align*}
	\rho\!\left( \sum\nolimits_{i \in F} a_{i} - \sum\nolimits_{i \in F'} a_{i}\right) \! \, &= \, \rho\!\left( \sum\nolimits_{i \in F\setminus F'} a_{i} + \sum\nolimits_{i \in F'\setminus F} -a_{i} \right) \\
	&\stackrel{\ref{lemma:sum.embedding}\ref{lemma:sum.embedding.2}}{=} \, \sum\nolimits_{i \in F \mathbin{\triangle} F'} \rho(a_{i}) \, \stackrel{\ref{lemma:sum.embedding}\ref{lemma:sum.embedding.2}}{=} \, \rho\!\left( \sum\nolimits_{i \in F \mathbin{\triangle} F'} a_{i} \right) \\
	& \leq \, s - \rho\!\left( \sum\nolimits_{i \in F_{0}} a_{i}\right)\! \, \leq \, \epsilon .
\end{align*} This shows that $(\sum_{i \in F} a_{i})_{F \in \Pfin(I)}$ is indeed a Cauchy net, hence converges in $(R,d_{\rho})$ due to completeness of the latter. In turn, $\phi \colon \prod\nolimits_{i \in I} e_{i}Re_{i} \to R, \, a \mapsto \sum\nolimits_{i \in I} a_{i}$ is well defined. Since $R$ is a Hausdorff topological ring with respect to the $\rho$-topology by Remark~\ref{remark:properties.pseudo.rank.function}\ref{remark:rank.group.topology}, it follows from Lemma~\ref{lemma:sum.embedding} that $\phi$ is a $\ZZ(R)$-algebra homomorphism, which readily entails the assertions~\ref{lemma:convergence.idempotent}, \ref{lemma:convergence.monotone}, and~\ref{lemma:convergence.orthogonal}. Finally, we observe that $\phi$ is injective: if $a \in \prod\nolimits_{i \in I} e_{i}Re_{i}$ and $\phi(a)=0$, then \begin{displaymath}
	a_{i} \, \stackrel{\ref{remark:properties.pseudo.rank.function}\ref{remark:rank.group.topology}}{=} \, e_{i}\sum\nolimits_{i \in I} a_{i} \, = \, e_{i}\phi(a) \, = \, e_{i}0 \, = \, 0
\end{displaymath} for each $i \in I$, i.e., $a=0$. \end{proof}

\begin{lem}\label{lemma:local.decomposition} Let $(R,\rho)$ be a complete rank ring, let $I$ be a set, let $(e_{i})_{i \in I}$ be a family of pairwise orthogonal elements of $\E(R)$, and let $e \defeq \sum\nolimits_{i \in I} e_{i}$. Then both \begin{displaymath}
	\prod\nolimits_{i \in I} \GL(e_{i}Re_{i}) \, \longrightarrow \, \GL(eRe), \quad (a_{i})_{i \in I} \, \longmapsto \, \sum\nolimits_{i \in I} a_{i}
\end{displaymath} and \begin{displaymath}
	\prod\nolimits_{i \in I} \Gamma_{R}(e_{i}) \, \longrightarrow \, \Gamma_{R}(e), \quad (a_{i})_{i \in I} \, \longmapsto \, \prod\nolimits_{i \in I} a_{i} \defeq \lim\nolimits_{F \to \Pfin(I)} \prod\nolimits_{i \in F} a_{i}
\end{displaymath} are group embeddings. \end{lem}

\begin{proof} This is a consequence of Lemma~\ref{lemma:convergence.sequences} and Lemma~\ref{lemma:subgroup.unit.group}. \end{proof}

\section{Geometry of involutions}\label{section:geometry.of.involutions}

In this section we prove that two involutions of an irreducible, continuous ring are conjugate in the associated unit group if and only if they are at equal rank distance from the identity (Proposition~\ref{proposition:conjugation.involution.rank}). Moreover, we provide two geometric decomposition results for involutions in non-discrete irreducible, continuous rings (Lemma~\ref{lemma:involution.rank.distance.char.neq2} and Lemma~\ref{lemma:involution.rank.distance.char=2}), relevant for the proof of Lemma~\ref{lemma:Gamma_{R}(e).subset.W^192}. The latter also builds on the fact that the unit group of every non-discrete irreducible, continuous ring contains a separable, uncountable, Boolean subgroup (Corollary~\ref{corollary:boolean.subgroup}).

Let us first clarify some notation. The set of \emph{involutions} of a group $G$ is defined as \begin{displaymath}
	\I(G) \, \defeq \, \left\{ g\in G \left\vert \, g^{2} = 1 \right\} . \right.
\end{displaymath} For a unital ring $R$, we consider $\I(R) \defeq \I(\GL(R))$ as well as the set \begin{displaymath}
	\Nil_{2}(R) \, \defeq \, \left\{ a\in R \left\vert \, a^{2} = 0 \right\} \right.
\end{displaymath} of \emph{$2$-nilpotent} elements of $R$. We observe that, for a unital ring $R$, the action of $\GL(R)$ on $R$ by conjugation preserves the sets $\E(R)$, $\I(R)$ and $\Nil_{2}(R)$. Our considerations are based on the following correspondence between involutions on the one hand and idempotents (resp., 2-nilpotent elements) on the other. As is customary, the \emph{characteristic} of a unital ring $R$ will be denoted by $\cha(R)$.

\begin{remark}\label{remark:ehrlich} Let $R$ be a unital ring. \begin{enumerate}
	\item\label{remark:bijection.idempotent.involution.char.neq2} If\footnote{By Remark~\ref{remark:irreducible.center.field}, this holds for any irreducible, regular ring $R$ with $\cha(R) \neq 2$.} $2 \in \GL(R)$, then \begin{align*}
			\qquad &\E(R) \, \longrightarrow \, \I(R), \quad e \, \longmapsto \, 2e-1, \\
			&\I(R) \, \longrightarrow \, \E(R), \quad u \, \longmapsto \, \tfrac{1}{2}(1+u)
		\end{align*} are mutually inverse bijections, $\GL(R)$-equivariant with respect to conjugation. This observation is due to~\cite[Lemma~1]{Ehr56}.
	\item\label{remark:bijection.2-nilpotent.involution.char=2} If $\cha(R)=2$, then \begin{align*}
			\qquad &\Nil_{2}(R) \, \longrightarrow \, \I(R), \quad x \, \longmapsto \, 1+x, \\
			&\I(R) \, \longrightarrow \, \Nil_{2}(R), \quad u \, \longmapsto \, 1+u
		\end{align*} are mutually inverse bijections, $\GL(R)$-equivariant with respect to conjugation. This follows by straightforward calculation.
\end{enumerate} \end{remark}

Our later considerations crucially rely on the following two insights about idempotent (resp., 2-nilpotent) elements and the resulting characterization of conjugacy of involutions (Proposition~\ref{proposition:conjugation.involution.rank}). 

\begin{lem}[\cite{Ehr56}, Lemma~9]\label{lemma:conjugation.idempotent} Let $R$ be an irreducible, continuous ring and $e,f \in \E(R)$. Then \begin{displaymath}
	\rk_{R}(e) = \rk_{R}(f) \ \ \Longleftrightarrow \ \ \exists g \in \GL(R)\colon \ geg^{-1} = f.
\end{displaymath} \end{lem}

\begin{proof} ($\Longleftarrow$) This follows from the $\GL(R)$-invariance of $\rk_{R}$.
		
($\Longrightarrow$) This is proved in~\cite[Lemma~9]{Ehr56}. (The standing assumption of~\cite{Ehr56} excluding characteristic 2 is not used in this particular argument.) \end{proof}

\begin{lem}\label{lemma:conjugation.2-nilpotent} Let $R$ be an irreducible, continuous ring and $a,b \in\Nil_{2}(R)$. Then \begin{displaymath}
	\rk_{R}(a) = \rk_{R}(b) \ \ \Longleftrightarrow \ \ \exists g \in \GL(R) \colon \, gag^{-1} = b.
\end{displaymath} \end{lem}

\begin{proof} ($\Longleftarrow$) This follows from the $\GL(R)$-invariance of $\rk_{R}$.
		
($\Longrightarrow$) Suppose that $\rk_{R}(a) = \rk_{R}(b)$. Due to~\cite[II.XVII, Theorem~17.1(d), p.~224]{VonNeumannBook}, there exist $u,v \in \GL(R)$ such that $b=uav$. By Remark~\ref{remark:bijection.annihilator} and Remark~\ref{remark:regular.idempotent.ideals}, there exists $f \in \E(R)$ with $fR = \rAnn(a)$. Then \begin{displaymath}
	\tilde{f} \, \defeq \, v^{-1}fv \, \in \, \E(R)
\end{displaymath} and \begin{displaymath}
	\rAnn(b) \, = \, \rAnn(uav) \, = \, \rAnn(av) \, = \, v^{-1}\rAnn(a) \, = \, v^{-1}fR \, = \, v^{-1}fvR \, = \, \tilde{f}R.
\end{displaymath} Since $a,b\in\Nil_{2}(R)$, we see that \begin{displaymath}
	aR \subseteq \rAnn(a) = fR, \qquad bR \subseteq \rAnn(b) = \tilde{f}R.
\end{displaymath} Thus, by~\cite[Lemma~7.5]{SchneiderGAFA}, there exist $e,\tilde{e} \in \E(R)$ such that \begin{displaymath}
	eR=aR, \quad e\leq f,\qquad \tilde{e}R=bR, \quad \tilde{e}\leq \tilde{f}.
\end{displaymath} From $uaR=uavR=bR$, we infer that $ueR=uaR=bR=\tilde{e}R$. Since \begin{align*}
	\rk_{R}(f-e) \, &\stackrel{\ref{remark:properties.pseudo.rank.function}\ref{remark:rank.difference.smaller.idempotent}}{=} \, \rk_{R}(f)-\rk_{R}(e) \, = \, \rk_{R}(v^{-1}fv)-\rk_{R}(a) \\
	& = \, \rk_{R}(\tilde{f})-\rk_{R}(\tilde{e}) \, = \, \rk_{R}(\tilde{f}-\tilde{e}),
\end{align*} the work of von Neumann~\cite[II.XVII, Theorem~17.1(d), p.~224]{VonNeumannBook} asserts the existence of $w,\tilde{w}\in\GL(R)$ such that $w(f-e)\tilde{w}=(\tilde{f}-\tilde{e})$. Therefore, $w(f-e)R =(\tilde{f}-\tilde{e})R$. Moreover, we observe that \begin{displaymath}
	v^{-1}(1-f)R \, = \, v^{-1}(1-f)vR \, = \, (1-v^{-1}fv)R \, = \, (1-\tilde{f})R.
\end{displaymath} Now, define \begin{displaymath}
	g\defeq v^{-1}(1-f)+w(f-e)+ue,\qquad h\defeq v(1-\tilde{f})+w^{-1}(\tilde{f}-\tilde{e})+u^{-1}\tilde{e}.
\end{displaymath} Then \begin{align*}
	hg \, &= \, v(1-\tilde{f})v^{-1}(1-f)+v(1-\tilde{f})w(f-e)+v(1-\tilde{f})ue\\
		&\qquad +w^{-1}(\tilde{f}-\tilde{e})v^{-1}(1-f)+w^{-1}(\tilde{f}-\tilde{e})w(f-e)+w^{-1}(\tilde{f}-\tilde{e})ue\\
		&\qquad +u^{-1}\tilde{e}v^{-1}(1-f)+u^{-1}\tilde{e}w(f-e)+u^{-1}\tilde{e}ue\\
		&= \,vv^{-1}(1-f)+v(1-\tilde{f})(\tilde{f}-\tilde{e})w(f-e)+v(1-\tilde{f})\tilde{e}ue\\
		&\qquad +w^{-1}(\tilde{f}-\tilde{e})(1-\tilde{f})v^{-1}(1-f)+w^{-1}w(f-e)+w^{-1}(\tilde{f}-\tilde{e})\tilde{e}ue\\
		&\qquad +u^{-1}\tilde{e}(1-\tilde{f})v^{-1}(1-f)+u^{-1}\tilde{e}(\tilde{f}-\tilde{e})w(f-e)+u^{-1}ue\\
		&= \, (1-f)+(f-e)+e \, = \, 1
\end{align*} and hence $g \in \GL(R)$ by Remark~\ref{remark:properties.rank.function}\ref{remark:directly.finite}. 
Finally, \begin{align*}
	ga \, &= \, v^{-1}(1-f)a+w(f-e)a+uea \, = \, v^{-1}(1-f)ea+w(f-e)ea+ua \, = \, ua, \\
	bg \, &= \, bv^{-1}(1-f)+bw(f-e)+b\!\!\smallunderbrace{ue}_{\in\,bR} \, = \, ua(1-f)+b(\tilde{f}-\tilde{e})w(f-e) \\
		& = \, ua(1-f) \, = \, ua(1-f)+uaf\, = \, ua,
\end{align*} That is, $b=gag^{-1}$ as desired. \end{proof}

\begin{prop}\label{proposition:conjugation.involution.rank} Let $R$ be an irreducible, continuous ring and $a,b\in \I(R)$. Then \begin{displaymath}
	\rk_{R}(1-a) = \rk_{R}(1-b) \ \ \Longleftrightarrow \ \ \exists g \in \GL(R) \colon \, gag^{-1} = b.
\end{displaymath} \end{prop}

\begin{proof} ($\Longleftarrow$) This follows from the $\GL(R)$-invariance of $\rk_{R}$.
		
($\Longrightarrow$) We distinguish two cases, depending on the characteristic of $R$.

First, let us suppose that $\cha(R)\neq 2$. Since $\ZZ(R)$ is a field due to Remark~\ref{remark:irreducible.center.field} and $\cha(\ZZ(R)) = \cha(R) \ne 2$, we see that $2 \in \ZZ(R)\setminus \{ 0 \} \subseteq \GL(R)$. Moreover, by Remark~\ref{remark:ehrlich}\ref{remark:bijection.idempotent.involution.char.neq2}, there exist $e,f \in \E(R)$ such that $a=2e-1$ and $b=2f-1$. In turn, \begin{align*}
	1-\rk_{R}(e) \, &\stackrel{\ref{remark:properties.pseudo.rank.function}\ref{remark:rank.difference.smaller.idempotent}}{=} \, \rk_{R}(1-e) \, \stackrel{2 \in \GL(R)}{=} \, \rk_{R}(2(1-e)) \, = \, \rk_{R}(1-2e+1) \\
	& = \, \rk_{R}(1-a) \, = \, \rk_{R}(1-b) \, = \, \rk_{R}(1-2f+1) \, = \, \rk_{R}(2(1-f)) \\
	& \stackrel{2 \in \GL(R)}{=} \, \rk_{R}(1-f) \, \stackrel{\ref{remark:properties.pseudo.rank.function}\ref{remark:rank.difference.smaller.idempotent}}{=} \, 1-\rk_{R}(f),
\end{align*} i.e., $\rk_{R}(e)=\rk_{R}(f)$. Hence, by Lemma~\ref{lemma:conjugation.idempotent}, there exists $g\in \GL(R)$ with $f=geg^{-1}$, which readily entails that $b=gag^{-1}$ by Remark~\ref{remark:ehrlich}\ref{remark:bijection.idempotent.involution.char.neq2}.

Second, suppose that $\cha(R)=2$. Then Remark~\ref{remark:ehrlich}\ref{remark:bijection.2-nilpotent.involution.char=2} asserts the existence of $x,y\in \Nil_2(R)$ such that $a=1+x$ and $b=1+y$. We observe that \begin{displaymath}
	\rk_{R}(x) \, = \, \rk_{R}(1-a) \, = \, \rk_{R}(1-b) \, = \, \rk_{R}(y).
\end{displaymath} Thanks to Lemma~\ref{lemma:conjugation.2-nilpotent}, thus there exists $g\in \GL(R)$ such that $y=gxg^{-1}$, wherefore $b=gag^{-1}$ according to Remark~\ref{remark:ehrlich}\ref{remark:bijection.2-nilpotent.involution.char=2}. \end{proof}

The following two decomposition results are pivotal to the proof of Lemma~\ref{lemma:Gamma_{R}(e).subset.W^192}.

\begin{lem}\label{lemma:involution.rank.distance.char.neq2} Let $R$ be a non-discrete irreducible, continuous ring with $\cha(R)\neq 2$. Then \begin{displaymath}
	\I(R) \, \subseteq \, \left\{ gh \left\vert \, g,h\in\I(R), \, \rk_{R}(1-g) = \rk_{R}(1-h) = \tfrac{1}{2} \right\}. \right.
\end{displaymath} \end{lem}

\begin{proof} Let $a \in \I(R)$. By Remark~\ref{remark:ehrlich}\ref{remark:bijection.idempotent.involution.char.neq2}, there exists $e \in \E(R)$ such that $a = 2e-1$. Due to Remark~\ref{remark:rank.function.general}\ref{remark:characterization.discrete} and Lemma~\ref{lemma:order}\ref{lemma:order.1}, we find $f,f' \in \E(R)$ with $f \leq e$, $\rk_{R}(f) = \tfrac{1}{2}\rk_{R}(e)$, $f'\leq 1-e$, and $\rk_{R}(f') = \tfrac{1}{2}\rk_{R}(1-e)$. Note that \begin{displaymath}
	g \, \defeq \, 2(e-f+f')-1, \qquad h \, \defeq \, 2(1-f-f')-1
\end{displaymath} belong to $\I(R)$ by Remark~\ref{remark:quantum.logic} and Remark~\ref{remark:ehrlich}\ref{remark:bijection.idempotent.involution.char.neq2}. As $\ZZ(R)$ is a field by Remark~\ref{remark:irreducible.center.field} and $\cha(\ZZ(R)) = \cha(R) \ne 2$, we see that $2 \in \ZZ(R)\setminus \{ 0 \} \subseteq \GL(R)$. Hence, \begin{align*}
	\rk_{R}(1-g) \, &= \, \rk_{R}((1-e)-f'+f) \, = \, \rk_{R}((1-e)-f') + \rk_{R}(f) \\
		&\stackrel{\ref{remark:properties.pseudo.rank.function}\ref{remark:rank.difference.smaller.idempotent}}{=} \, \rk_{R}(1-e) - \rk_{R}(f') + \rk_{R}(f) \, = \, \tfrac{1}{2}\rk_{R}(1-e) + \tfrac{1}{2}\rk_{R}(e) \, = \, \tfrac{1}{2}, \\
	\rk_{R}(1-h) \, &= \, \rk_{R}(f+f') \, = \, \rk_{R}(f) + \rk_{R}(f') \, = \, \tfrac{1}{2}\rk_{R}(e) + \tfrac{1}{2}\rk_{R}(1-e) \, = \, \tfrac{1}{2}.
\end{align*} Finally, as desired, \begin{align*}
	gh \, &= \, (2(e-f+f')-1)(2(1-f-f')-1)\\
		&= \, 4(e-f+f')(1-f-f')-2(e-f+f')-2(1-f-f')+1\\
		&= \, 4e-4f-2e+4f-1 \, = \, 2e-1 \, = \, a. \qedhere
\end{align*} \end{proof}

\begin{lem}\label{lemma:involution.rank.distance.char=2} Let $R$ be a non-discrete irreducible, continuous ring with $\cha(R) = 2$. Then \begin{displaymath}
	\I(R) \, \subseteq \, \left\{ gh \left\vert \, g,h\in\I(R), \, \rk_{R}(1-g) = \rk_{R}(1-h) = \tfrac{1}{4} \right\}. \right.
\end{displaymath} \end{lem}

\begin{proof} Let $a \in \I(R)$. By Remark~\ref{remark:ehrlich}\ref{remark:bijection.2-nilpotent.involution.char=2}, there exists $b \in \Nil_{2}(R)$ such that $a=1+b$. According to Remark~\ref{remark:regular.idempotent.ideals}, we find $e \in \E(R)$ such that $Rb = Re$. By Remark~\ref{remark:bijection.annihilator}, thus $\rAnn(b)=\rAnn(Rb)=\rAnn(Re)=(1-e)R$. Since $b\in\Nil_{2}(R)$, we see that $bR\subseteq \rAnn(b)=(1-e)R$. Consequently, \cite[Lemma~7.5]{SchneiderGAFA} asserts the existence of some $f'\in \E(R)$ such that $bR=f'R$ and $f'\leq 1-e$. Using Remark~\ref{remark:quantum.logic}\ref{remark:quantum.logic.1}, we infer that $f \defeq 1-e-f' \in \E(R)$ and $f\leq 1-e$, i.e., $e \perp f$. Clearly, $bR=f'R=(1-e-f)R$. Note that \begin{align*}
	\rk_{R}(b) \, &= \, \rk_{R}(1-e-f) \, \stackrel{\ref{remark:properties.pseudo.rank.function}\ref{remark:rank.difference.smaller.idempotent}}{=} \, \rk_{R}(1-e)-\rk_{R}(f)\\
	&\stackrel{\ref{remark:properties.pseudo.rank.function}\ref{remark:rank.difference.smaller.idempotent}}{=} \, 1-\rk_{R}(e)-\rk_{R}(f) \, = \, 1-\rk_{R}(b)-\rk_{R}(f),
\end{align*} wherefore \begin{equation}\label{eq5}
	\rk_{R}(f) \, = \, 1-2\rk_{R}(b).
\end{equation} Furthermore, thanks to Remark~\ref{remark:rank.function.general}\ref{remark:characterization.discrete} and Lemma~\ref{lemma:order}\ref{lemma:order.1}, there exists $e_{0} \in \E(R)$ such that $e_{0} \leq e$ and $\rk_{R}(e_{0}) = \tfrac{1}{2}\rk_{R}(e)$. Then $e_{1} \defeq e-e_{0} \in \E(R)$ and $e_{1} \leq e$ due to Remark~\ref{remark:quantum.logic}\ref{remark:quantum.logic.1} and $\rk_{R}(e_{1}) = \rk_{R}(e)-\rk_{R}(e_{0}) = \tfrac{1}{2}\rk_{R}(e)$ by Remark~\ref{remark:properties.pseudo.rank.function}\ref{remark:rank.difference.smaller.idempotent}. In a similar manner, using Remark~\ref{remark:rank.function.general}\ref{remark:characterization.discrete}, Lemma~\ref{lemma:order}\ref{lemma:order.1}, Remark~\ref{remark:quantum.logic}\ref{remark:quantum.logic.1} and Remark~\ref{remark:properties.pseudo.rank.function}\ref{remark:rank.difference.smaller.idempotent}, we find $f_{0},f_{1} \in \E(R)$ such that \begin{align*}
	&f_{0} \leq f, \quad f_{1} \leq f, \quad f_{0} \perp f_{1}, \quad \rk_{R}(f_{0}) = \rk_{R}(f_{1}) = \tfrac{1}{4}\rk_{R}(f).
\end{align*} If $f \ne 0$, then by Remark~\ref{remark:eRe.non-discrete.irreducible.continuous} and Lemma~\ref{lemma:conjugation.idempotent} there is $u \in \GL(fRf)$ with $uf_{0} = f_{1}u$. Otherwise, in case $f = 0$, we observe that $f_{0} = 0 = f_{1}$ and $\GL(fRf) = \GL(\{ 0 \}) = \{ 0 \}$ and we let $u \defeq 0$. Now, define \begin{displaymath}
	x_{i} \, \defeq \, uf_{0} + be_{i} \qquad (i\in\{0,1\}).
\end{displaymath} For any $i,j \in \{0,1\}$, we see that \begin{align*}
	x_{i}x_{j} \, &= \, (uf_{0}+be_{i})(uf_{0}+be_{j}) \, = \, uf_{0}uf_{0}+uf_{0}be_{j}+be_{i}uf_{0}+be_{i}be_{j} \\
	&= \, u\!\smallunderbrace{f_{0}f_{1}}_{=\,0}\!u+uf_{0}\!\smallunderbrace{f(1-e-f)}_{=\,0}\!be_{j}+be_{i}\!\smallunderbrace{ef}_{=\,0}\!f_{1}u+be_{i}\!\smallunderbrace{e(1-e-f)}_{=\,0}\!be_{j} \, = \, 0.
\end{align*} Hence, $x_{0},x_{1} \in \Nil_{2}(R)$ and $x_{1}x_{0} =x_{0}x_{1} = 0$. Since $e \in Rb$ and $(1-e-f)b = b$, we find $c \in R$ such that $cb = e$ and $c(1-e-f) = c$. Moreover, let us consider the inverse\footnote{In particular, if $f = 0$, then $u^{-1} = 0$.} $u^{-1}\in\GL(fRf)\subseteq fRf$. Then, for each $i\in\{0,1\}$, \begin{align}
	\left(u^{-1}+c\right)\!(uf_{0}+be_{i}) \, &= \, u^{-1}uf_{0}+u^{-1}be_{i}+cuf_{0}+cbe_{i}\nonumber\\
	&= \, f_{0}+u^{-1}\!\smallunderbrace{f(1-e-f)}_{=\,0}be_{i}+c\smallunderbrace{(1-e-f)f}_{=\,0}\!f_{1}u+ee_{i}\nonumber\\
	&= \, f_{0}+e_{i}\label{eq4}
\end{align} and therefore \begin{align*}
	\rk_{R}(f_{0})+\rk_{R}(e_{i}) \, &= \, \rk_{R}(f_{0}+e_{i}) \, \stackrel{\eqref{eq4}}{=} \, \rk_{R}\!\left(\left(u^{-1}+c\right)\!(uf_{0}+be_{i})\right)\\
	&\leq \, \rk_{R}(uf_{0}+be_{i}) \, = \, \rk_{R}(x_{i}) \\ 
	&\stackrel{\ref{remark:properties.pseudo.rank.function}\ref{remark:inequation.sum.pseudo.rank}}{\leq} \, \rk_{R}(uf_{0})+\rk_{R}(be_{i}) \, \leq \, \rk_{R}(f_{0})+\rk_{R}(e_{i}) ,
\end{align*} which entails that \begin{displaymath}
	\rk_{R}(x_{i}) = \rk_{R}(f_{0})+\rk_{R}(e_{i}) = \tfrac{1}{4}\rk_{R}(f)+\tfrac{1}{2}\rk_{R}(e)\stackrel{\eqref{eq5}}{=} \tfrac{1}{4}(1-2\rk_{R}(b))+\tfrac{1}{2}\rk_{R}(b) = \tfrac{1}{4}.
\end{displaymath} We conclude that, for each $i\in\{0,1\}$, \begin{displaymath}
	g_{i} \, \defeq \, 1+x_{i} \, \stackrel{\ref{remark:ehrlich}\ref{remark:bijection.2-nilpotent.involution.char=2}}{\in} \, \I(R)
\end{displaymath} and $\rk_{R}(1-g_{i}) = \rk_{R}(x_{i}) = \tfrac{1}{4}$. Finally, as intended, \begin{align*}
	&g_{0}g_{1} \, = \, (1+x_{0})(1+x_{1}) \, = \, 1+x_{0}+x_{1}+x_{0}x_{1} \, = \, 1+x_{0}+x_{1} \\
	&\ \ = \, 1+uf_{0}+be_{0}+uf_{0}+be_{1} \, \stackrel{\cha(R)=2}{=} \, 1+b(e_{0}+e_{1}) \, = \, 1+be \, = \, 1+b \, = \, a. \qedhere
\end{align*} \end{proof}

Turning to this section's final results, let us recall that a group $G$ is said to be \emph{Boolean} if $x^{2} = 1$ for every $x \in G$.

\begin{remark} Let $S$ be a commutative unital ring. Then $(\E(S),\vee,\wedge,\neg,0,1)$, with the operations defined by \begin{displaymath}
	x\vee y \, \defeq \, x+y-xy,\qquad x\wedge y \, \defeq \, xy, \qquad \neg x \, \defeq \, 1-x
\end{displaymath} for all $x,y \in \E(S)$, constitutes a Boolean algebra (see, e.g.,~\cite[Chapter~8, p.~83]{GoodearlBook}). In particular, $(\E(S),\mathbin{\triangle})$ with \begin{displaymath}
	x\mathbin{\triangle} y \, \defeq \, (x\vee y)\wedge \neg(x\wedge y) \, = \, x+y-2xy \qquad (x,y\in\E(S))
\end{displaymath} is a Boolean group. This Boolean group will be considered in Proposition~\ref{proposition:homomorphism.idempotent.eRe.GL(R)} and the proof of Corollary~\ref{corollary:boolean.subgroup}. \end{remark}

\begin{prop}\label{proposition:homomorphism.idempotent.eRe.GL(R)} Let $R$ be a unital ring, $e \in \E(R)$, $a \in \GL(R)$ with $e \perp aea^{-1}$, and $S$ be a commutative unital subring of $eRe$. Then \begin{displaymath}
	\phi \colon \, \E(S) \, \longrightarrow\, \GL(R), \quad x \, \longmapsto \, ax+xa^{-1}+1-x-axa^{-1}
\end{displaymath} is a well-defined group embedding. Furthermore, if $R$ is a regular ring and $\rho$ is a pseudo-rank function on $R$, then \begin{equation*}
	\phi \colon \, (\E(S),d_{\rho}) \, \longrightarrow \, (\GL(R),d_{\rho})
\end{equation*} is 4-Lipschitz, hence continuous. \end{prop}

\begin{proof} Note that $axa^{-1} \leq aea^{-1}$ for all $x \in \E(eRe)$. As $e \perp aea^{-1}$, this entails that \begin{equation}\label{stern}
	\forall x,y \in \E(eRe) \colon \quad x \perp aya^{-1}.
\end{equation} Therefore, if $x,y \in \E(S)$, then \begin{align*}
	\phi(x)\phi(y) \, &= \, \!\left(ax+xa^{-1}+1-x-axa^{-1}\right)\!\left(ay+ya^{-1}+1-y-aya^{-1}\right)\\
			&= \, \smallunderbrace{axay}_{=\,axa\mathrlap{ya^{-1}a \, \stackrel{\eqref{stern}}{=}\,0}}+\:axya^{-1}+ax(1-y)-\smallunderbrace{axaya^{-1}}_{\stackrel{\eqref{stern}}{=}\,0}\\
			&\qquad+xa^{-1}ay+\smallunderbrace{xa^{-1}ya^{-1}}_{=\,a^{-1}axa^{-1}y\mathrlap{a^{-1}\,\stackrel{\eqref{stern}}{=}\,0}}+\:xa^{-1}-\smallunderbrace{xa^{-1}y}_{=\,a^{-1}a\mathrlap{xa^{-1}y\,\stackrel{\eqref{stern}}{=}\,0}}-\:xa^{-1}aya^{-1}\\
			&\qquad+\smallunderbrace{(1-x)ay}_{=\,(1-x)ay\mathrlap{a^{-1}a\,\stackrel{\eqref{stern}}{=}\,ay}}+\:(1-x)ya^{-1}+(1-x)(1-y)-\smallunderbrace{(1-x)aya^{-1}}_{\stackrel{\eqref{stern}}{=} \, aya^{-1}}\\
			&\qquad-axa^{-1}ay-\smallunderbrace{axa^{-1}ya^{-1}}_{\stackrel{\eqref{stern}}{=}\,0}-\smallunderbrace{axa^{-1}(1-y)}_{\stackrel{\eqref{stern}}{=}\,axa^{-1}}+\:axa^{-1}aya^{-1}\\
			&= \, axya^{-1}+ax-axy+xy+xa^{-1}-xya^{-1}+ay\\
			&\qquad+(1-x)ya^{-1}+(1-x)(1-y)-aya^{-1}-axy-axa^{-1}+axya^{-1}\\
			&= \, ax+ay-2axy+xa^{-1}+ya^{-1}-2xya^{-1}\\
			&\qquad+1-x-y+2xy-axa^{-1}-aya^{-1}+2axya^{-1}\\
			&= \, a(x\mathbin{\triangle} y)+(x\mathbin{\triangle} y)a^{-1}+1-(x\mathbin{\triangle} y)-a(x\mathbin{\triangle} y)a^{-1}\\
			&= \, \phi(x\mathbin{\triangle} y).
\end{align*} Since $\phi(0) = 1$, it follows that $\phi \colon \E(S) \to \GL(R)$ is a well-defined group homomorphism. From $e\perp aea^{-1}$, we deduce that $eR \cap aeR = eR \cap aea^{-1}R = \{ 0 \}$. Now, if $x \in \E(S)$ and $\phi(x) = 1$, then \begin{displaymath}
	eR \, \ni \, x \, = \, \phi(x)x \, = \, ax + \smallunderbrace{xa^{-1}x}_{=\,a^{-1}ax\mathrlap{a^{-1}x \, \stackrel{\eqref{stern}}{=} \, 0}} + \, (1-x)x + \smallunderbrace{axa^{-1}x}_{\stackrel{\eqref{stern}}{=} \, 0} \, = \, ax \, \in \, aeR
\end{displaymath}  and thus $x=0$. This shows that $\phi$ is an embedding. Finally, if $R$ is regular and $\rho$ is a pseudo-rank function on $R$, then \begin{align*}
	d_{\rho}(\phi(x),\phi(y)) \, &= \, \rho(\phi(x)-\phi(y))\\
		& = \, \rho\!\left(a(x-y)+(x-y)a^{-1}-(x-y)-a(x-y)a^{-1}\right) \\
		& \stackrel{\ref{remark:properties.pseudo.rank.function}\ref{remark:inequation.sum.pseudo.rank}}{\leq} \, \rho(a(x-y))+\rho\!\left((x-y)a^{-1}\right)\!+\rho(x-y) +\rho\!\left(a(x-y)a^{-1}\right) \\
		& \stackrel{a \in \GL(R)}{=} \, 4\rho(x-y) \, = \, 4d_{\rho}(x,y)
\end{align*} for all $x,y \in \E(S)$, i.e., $\phi \colon (\E(S),d_{\rho}) \to (\GL(R),d_{\rho})$ is indeed 4-Lipschitz. \end{proof}

\begin{cor}\label{corollary:boolean.subgroup} Let $R$ be a non-discrete irreducible, continuous ring. Then $\GL(R)$ contains a separable, uncountable, Boolean subgroup. \end{cor}

\begin{proof} According to the Hausdorff maximal principle, there exists a maximal chain~$E$ in $(\E(R),\leq)$. By~\cite[Corollary~7.20]{SchneiderGAFA}, the map ${{\rk_{R}}\vert_{E}}\colon (E,\leq)\to ([0,1],\leq)$ is an order isomorphism. Now, let $e\defeq({{\rk_{R}}\vert_{E}})^{-1}\!\left(\tfrac{1}{2}\right)$. Since $(E,\leq)$ is a chain, \begin{displaymath}
	E' \, \defeq \, \{ f \in E \mid f \leq e \} \, = \, \left\{ f \in E \left\vert \, \rk_{R}(f) \leq \tfrac{1}{2} \right\} \right.
\end{displaymath} is a commutative submonoid of the multiplicative monoid of $eRe$. Hence, the unital subring $S$ of $eRe$ generated by $E'$ is commutative, too. As ${{\rk_{R}}\vert_{E}}\colon (E,d_{R}) \to ([0,1],d)$ is isometric with respect to the standard metric $d$ on $[0,1]$ by Remark~\ref{remark:properties.pseudo.rank.function}\ref{remark:rank.order.isomorphism}, we see that $(E,d_{R})$ is separable, therefore $(E',d_{R})$ is separable. Due to Remark~\ref{remark:eRe.non-discrete.irreducible.continuous} and Remark~\ref{remark:properties.pseudo.rank.function}\ref{remark:rank.group.topology}, we know that $eRe$ is a topological ring with regard to the topology generated by $d_{R}$, wherefore $(S,d_{R})$ is separable, thus $(\E(S),d_{R})$ is separable. From $E' \subseteq \E(S)$ and $\vert E'\vert =\left\lvert\left[0,\tfrac{1}{2}\right]\right\rvert$, we infer that $\E(S)$ is uncountable. Since \begin{displaymath}
	\rk_{R}(1-e) \, \stackrel{\ref{remark:properties.pseudo.rank.function}\ref{remark:rank.difference.smaller.idempotent}}{=} \, 1-\rk_{R}(e) \, = \, 1-\tfrac{1}{2} \, = \, \tfrac{1}{2} \, = \, \rk_{R}(e),
\end{displaymath} Lemma~\ref{lemma:conjugation.idempotent} asserts the existence of $a \in \GL(R)$ with $aea^{-1} = 1-e$. By Proposition~\ref{proposition:homomorphism.idempotent.eRe.GL(R)}, there exists a continuous injective homomorphism from $\E(S)$ to $\GL(R)$, whose image necessarily constitutes a separable, uncountable, Boolean subgroup of $\GL(R)$. \end{proof}

\section{Dynamical independence of ideals}\label{section:dynamical.independence}

The focus of this section is on the following notion of dynamical independence of principal right ideals of a regular ring (Definition~\ref{definition:halperin.independence}), which turns out to be a key ingredient for the proof of our characterization of algebraic elements (Theorem~\ref{theorem:matrixrepresentation.case.algebraic}). 

\begin{definition}\label{definition:halperin.independence} Let $R$ be a regular ring, let $a\in R$, and let $m \in \N$. Then a pair $(I,J) \in \lat(R) \times \lat(R)$ will be called \emph{$(a,m)$-independent} if $(I,aI,\ldots,a^{m-1}I,J)\perp$. \end{definition}

\begin{lem}\label{lemma:independence.inductive} Let $R$ be an irreducible, continuous ring, let $a\in R$, let $m \in \N$, let $J,J_{1} \in \lat(R)$, and let $P \defeq \{ I\in \lat(R) \mid \text{$(I,J)$ $(a,m)$-independent}, \, I \subseteq J_{1} \}$. Then the partially ordered set $(P,{\subseteq})$ is inductive. In particular, every element of $P$ is contained in a maximal element of $(P,{\subseteq})$. \end{lem}

\begin{proof} First we show that $P$ is closed in $(\lat(R),d_{\lat(R)})$. Note that $\delta_{\lat(R)} = \delta_{\rk_{R}}$ thanks to Theorem~\ref{theorem:unique.rank.function}, Lemma~\ref{lemma:pseudo.dimension.function} and Theorem~\ref{theorem:dimension.function.lattice}. Applying~\cite[V.7, p.~76, Lemma]{BirkhoffBook} (or~\cite[I.6, Satz 6.2(IV), p.~46]{MaedaBook}) and Lemma~\ref{lemma:pseudo.dimension.function}\ref{lemma:multplication.Lipschitz}, we see that \begin{displaymath}
	\Phi \colon \, (\lat(R),d_{\lat(R)}) \, \longrightarrow \, (\lat(R),d_{\lat(R)}), \quad I \, \longmapsto \, I+J_{1}
\end{displaymath} is $1$-Lipschitz, \begin{displaymath}	
	\Psi \colon \, (\lat(R),d_{\lat(R)}) \, \longrightarrow \, (\lat(R),d_{\lat(R)}), \quad I \, \longmapsto \, J \cap \left( \sum\nolimits_{i=0}^{m-1} a^{i}I\right)
\end{displaymath} is $m$-Lipschitz, and for each $i\in\{0,\ldots,m-1\}$ the map \begin{displaymath}
	\Psi_{i} \colon \, (\lat(R),d_{\lat(R)}) \, \longrightarrow \, (\lat(R),d_{\lat(R)}), \quad I \, \longmapsto \, a^{i}I \cap \left( J + \sum\nolimits_{j=0,\,j\ne i}^{m-1} a^{j}I\right)
\end{displaymath} is $m$-Lipschitz. The continuity of those maps implies that
\begin{align*}
	P \, &= \, \{ I\in \lat(R) \mid (I,aI,\ldots,a^{m-1}I,J)\perp, \, I \subseteq J_{1} \} \\
		&\stackrel{\ref{remark:independence.equivalence.intersection}}{=} \, {\Phi^{-1}(\{J_{1}\})} \cap {\Psi^{-1}(\{\{0\}\})} \cap {\bigcap\nolimits_{i=0}^{m-1}\Psi_i^{-1}(\{\{0\}\})}
\end{align*} is closed in $(\lat(R),d_{\lat(R)})$, as claimed. 
	
We now deduce that $(P,{\subseteq})$ is inductive. Since $\{ 0\} \in P$, we see that $P$ is non-void. Consider a non-empty chain $\mathcal{C}$  in  $(P,\subseteq)$. As $P$ is closed in $(\lat(R),d_{{\lat(R)}})$, we conclude that $\overline{\mathcal{C}}\subseteq P$, where the closure is taken with respect to the topology induced by $d_{{\lat(R)}}$. For every $\epsilon\in\R_{>0}$, Proposition~\ref{proposition:dimension.function.continuous} asserts the existence of some $C\in\mathcal{C}$ such that \begin{displaymath}
	\delta_{\lat(R)}\!\left(\bigvee \mathcal{C}\right)\!-\delta_{\lat(R)}(C) \, \leq \, \epsilon ,
\end{displaymath} whence \begin{displaymath}
	d_{{\lat(R)}}\!\left(C,\bigvee\mathcal{C}\right)\! \, \stackrel{C\subseteq\bigvee\mathcal{C}}{=} \, \delta_{\lat(R)}\!\left(\bigvee \mathcal{C}\right)\! -\delta_{\lat(R)}(C) \, \leq \, \epsilon .
\end{displaymath} Thus, $\bigvee \mathcal{C}\in \overline{\mathcal{C}} \subseteq P$. In turn, $\bigvee \mathcal{C}$ constitutes the desired upper bound for $\mathcal{C}$ in $(P,{\subseteq})$. This shows that $(P,\subseteq)$ is inductive. The final assertion follows by Zorn's lemma. \end{proof}
	
\begin{lem}\label{lemma:independence.sum} Let $R$ be a regular ring, let $a\in R$, let $m \in \N$, let $I,I',J \in \lat(R)$, and $J' \defeq \sum\nolimits_{i=0}^{m-1}a^{i}I + J$. If $(I,J)$ is $(a,m)$-independent and $(I',J')$ is $(a,m)$-independent, then $(I+I',J)$ is $(a,m)$-independent. \end{lem}	

\begin{proof} Suppose that $(I,J)$ is $(a,m)$-independent and $(I',J')$ is $(a,m)$-independent. Let $x_{0},\ldots,x_{m-1} \in I$, $y_{0},\ldots,y_{m-1} \in I'$ and $z \in J$ be such that \begin{displaymath}
	\sum\nolimits_{i=0}^{m-1} a^{i}x_{i} + \sum\nolimits_{i=0}^{m-1} a^{i}y_{i} + z \, = \, 0 .
\end{displaymath} Since $(I',J')$ is $(a,m)$-independent and $\sum\nolimits_{i=0}^{m-1} a^{i}x_{i} + z \in J'$, it follows that $a^{i}y_{i} = 0$ for each $i \in \{ 0,\ldots,m-1\}$ and $\sum\nolimits_{i=0}^{m-1} a^{i}x_{i} + z = 0$. As $(I,J)$ is $(a,m)$-independent, the latter necessitates that $a^{i}x_{i} = 0$ for each $i \in \{ 0,\ldots,m-1\}$ and $z = 0$. According to Remark~\ref{remark:independence.ideals.idempotents}, this shows that $(I+I',\ldots,a^{m-1}(I+I'),J)$ is independent, i.e., $(I+I',J)$ is $(a,m)$-independent. \end{proof}

The subsequent Lemma~\ref{lemma:independence.halperin}, this section's main insight, refines and extends an argument from~\cite{Halperin62}. For its proof, the following characterization of the center of a non-discrete irreducible, continuous ring will be needed.

\begin{prop}[\cite{Halperin62}, Corollary and Remark to Lemma~2.2, p.~4]\label{proposition:center.halperin} If $R$ is a non-discrete irreducible, continuous ring, then \begin{displaymath}
	\ZZ(R) \, = \, \{ a \in R \mid \forall e \in \E(R) \colon \, eae = ae \} .
\end{displaymath} \end{prop}

For the statement of the next result, let us clarify some notation and terminology. To this end, let $K$ be a field and let $R$ be a unital $K$-algebra. We denote by $K[X]$ the polynomial ring over $K$ and by $\deg \colon K[X]\setminus\{0\} \to \N$ the usual degree function. For any $a \in R$, we consider the induced evaluation map \begin{displaymath}
	K[X] \, \longrightarrow \, R, \quad p=\sum\nolimits_{i=0}^{m}c_{i}X_{i} \, \longmapsto \, p(a) \defeq p_{R}(a) \defeq \sum\nolimits_{i=0}^{m}c_{i}a^{i} ,
\end{displaymath} which is a unital $K$-algebra homomorphism.

\begin{lem}[cf.~\cite{Halperin62}, Lemma~5.1]\label{lemma:independence.halperin} Let $R$ be a non-discrete irreducible, continuous ring, let $K \defeq \ZZ(R)$, let $a \in R$, let $J \in \lat(R)$ and let $J_{1} \in \lat(R)$ be such that $J \subseteq J_{1}$, $aJ \subseteq J$, $aJ_{1} \subseteq J_{1}$. Suppose that $m\in\N$ is such that \begin{displaymath}
	\forall p \in K[X] \setminus \{0\} \colon \quad \deg(p)<m \ \Longrightarrow \ \rk_{R}(p(a))=1 .
\end{displaymath} Then the following hold. \begin{enumerate}
	\item\label{lemma:independence.halperin.1} If $J \ne J_{1}$, then there exists $I \in \lat(R)$ such that $(I,J)$ is $(a,m)$-independent and $\{ 0 \} \ne I \subseteq J_{1}$.
	\item\label{lemma:independence.halperin.2} Suppose that $a^{m}x \in \sum\nolimits_{i=0}^{m-1} Ka^{i}x + J$ for every $x \in J_{1}$. If $I \in \lat(R)$ is maximal such that $(I,J)$ is $(a,m)$-independent and $I \subseteq J_{1}$, then $J_{1} = \bigoplus\nolimits_{i=0}^{m-1}a^{i} I \oplus J$.
\end{enumerate} \end{lem}

\begin{proof} \ref{lemma:independence.halperin.1} The argument, which follows the lines of~\cite[Proof of Lemma~5.1]{Halperin62}, proceeds by induction on $m \in \N$. If $m = 0$, then the statement is trivially satisfied for $I \defeq J_{1}$, as the $1$-tuple $(I,\ldots,a^{m-1}I,J) = (J)$ is clearly independent in $\lat(R)$ (cf.~Remark~\ref{remark:independence.equivalence.intersection}). For $m = 1$, we see that, thanks to Lemma~\ref{lemma:complement}, there exists $J' \in \lat(R)$ such that $J \oplus J' = J_{1}$, where $J' \neq \{0\}$ as $J \ne J_{1}$. For the induction step, let us assume that the desired implication is true for some $m \in \N_{>0}$. Suppose now that \begin{displaymath}
	\forall p\in K[X]\setminus \{0\}\colon\quad \deg(p)<m+1 \ \Longrightarrow \ \rk_{R}(p(a))=1.
\end{displaymath} As $X\in K[X]\setminus\{0\}$ and $\deg(X)=1<m+1$, we see that $\rk_{R}(a)=1$, so $a\in\GL(R)$ by Remark~\ref{remark:properties.rank.function}\ref{remark:invertible.rank}. Our induction hypothesis asserts the existence of some $I_{0} \in \lat(R)$ such that $\{ 0 \} \ne I_{0}\subseteq J_{1}$ and $(I_{0},\ldots,a^{m-1}I_{0},J)\perp$. We will deduce that \begin{equation}\label{eq3}
	\exists I \in \lat(R) \colon \quad I\subseteq I_{0}, \ \, a^{m}I \nsubseteq I+\ldots+a^{m-1}I+J.
\end{equation} For a proof of~\eqref{eq3} by contraposition, suppose that \begin{displaymath}
	\forall I\in \lat(R), \, I\subseteq I_{0} \colon \quad a^{m}I \subseteq I+\ldots+a^{m-1}I+J.
\end{displaymath} From this we conclude that \begin{equation}\label{eqhalp}
	\forall e\in\E(R)\cap I_{0}\ \exists x\in J\ \exists u_{0},\ldots,u_{m-1}\in eRe\colon\quad a^{m}e = x+\sum\nolimits_{i=0}^{m-1}a^{i}u_{i}.
\end{equation} Indeed, if $ e\in\E(R)\cap I_{0}$, then $a^{m}e \in eI_{0}+\ldots+a^{m-1}eI_{0}+J$, thus \begin{displaymath}
	a^{m}e \, \in \, eI_{0}e+\ldots+a^{m-1}eI_{0}e+Je \, \subseteq \, eRe+\ldots+a^{m-1}eRe+J .
\end{displaymath} Now, let us choose any $e_{0} \in \E(R)$ with $e_{0}R = I_{0}$. According to~\eqref{eqhalp}, there exist $x \in J$ and $u_{0},\ldots,u_{m-1} \in e_{0}Re_{0}$ such that $a^{m}e_{0} = x+\sum\nolimits_{i=0}^{m-1}a^{i}u_{i}$. We will show that $u_{0},\ldots,u_{m-1} \in \ZZ(e_{0}Re_{0})$. As $e_{0}Re_{0}$ is a non-discrete irreducible, continuous ring by Remark~\ref{remark:eRe.non-discrete.irreducible.continuous}, we may do so using Proposition~\ref{proposition:center.halperin}. To this end, let $e\in \E(e_{0}Re_{0})$. Then \begin{displaymath}
	a^{m}e \, = \, a^{m}e_{0}e \, = \, xe+\sum\nolimits_{i=0}^{m-1}a^{i}u_{i}e .
\end{displaymath} Moreover, since $\E(e_{0}Re_{0}) \subseteq {\E(R)} \cap {I_{0}}$, assertion~\eqref{eqhalp} provides the existence of $y \in J$ and $v_{0},\ldots,v_{m-1} \in eRe$ such that $a^{m}e = y+\sum\nolimits_{i=0}^{m-1}a^{i}v_{i}$. Hence, \begin{displaymath}
	0 \, = \, a^{m}e-a^{m}e \, = \, \smallunderbrace{(y-xe)}_{\in\, J} + \sum\nolimits_{i=0}^{m-1}\smallunderbrace{a^{i}(v_{i}-u_{i}e)}_{\in\, a^{i}I_{0}}.
\end{displaymath} As $(I_{0},\ldots,a^{m-1}I_{0},J)\perp$ and $a\in\GL(R)$, thus $v_{i}=u_{i}e$ for each $i\in\{0,\ldots,m-1\}$. It follows that, for every $i\in\{0,\ldots,m-1\}$, \begin{displaymath}
	eu_{i}e \, = \, ev_{i} \, = \, v_{i} \, = \, u_{i}e.
\end{displaymath} Due to Proposition~\ref{proposition:center.halperin}, this shows that $u_{0},\ldots,u_{m-1}\in \ZZ (e_{0}Re_{0})$. So, by Remark~\ref{remark:eRe.non-discrete.irreducible.continuous}, for each $i \in \{0,\ldots,m-1\}$ we find $z_{i} \in K$ such that $u_{i} = z_{i}e_{0}$. In turn, \begin{displaymath}
	a^{m}e_{0} \, = \, x+\sum\nolimits_{i=0}^{m-1}a^{i}z_{i}e_{0}
\end{displaymath} and therefore $x = \bigl(a^{m}-\sum\nolimits_{i=0}^{m-1}z_{i}a^{i}\bigr)e_{0}$. Concerning \begin{displaymath}
	p \, \defeq \, X^{m}-\sum\nolimits_{i=0}^{m-1}z_{i}X^{i} \, \stackrel{m > 0}{\in} \, K[X]\setminus\{0\},
\end{displaymath} we thus see that $p(a)e_{0} = x$. Furthermore, $\deg(p) < m+1$ and hence $\rk_{R}(p(a)) = 1$, so $p(a) \in \GL(R)$ by Remark~\ref{remark:properties.rank.function}\ref{remark:invertible.rank}. From $aJ \subseteq J$ we infer that $p(a)J \subseteq J$, so that $p(a)J = J$ by \cite[Lemma~9.4(B)]{SchneiderGAFA}, and therefore $p(a)^{-1}J = J$. We conclude that \begin{displaymath}
	e_{0} \, = \, p(a)^{-1}x \, \in \, p(a)^{-1}J \, = \, J
\end{displaymath} and hence \begin{displaymath}
	e_{0} \, \in \, I_{0} \cap J \, \stackrel{(I_0,J)\perp}{=} \, \{0\},
\end{displaymath} thus $e_{0} = 0$ and so $I_{0} = e_{0}R = \{0\}$, which gives the intended contradiction. This proves~\eqref{eq3}. Now, by~\eqref{eq3}, we find $I \in \lat(R)$ with $I \subseteq I_{0}$ and $a^{m}I \nsubseteq I + \ldots + a^{m-1}I + J$. By Remark~\ref{remark:bijection.annihilator} and Lemma~\ref{lemma:complement}, there exists $I' \in \lat(R)$ such that \begin{displaymath}
	((a^{m}I) \cap (I+\ldots+a^{m-1}I+J))\oplus I' \, = \, a^{m}I.
\end{displaymath} Necessarily, $I' \neq \{0\}$. Consider $I'' \defeq a^{-m}I' \in \lat(R) \setminus \{ 0 \}$ and observe that \begin{displaymath}
	I'' \, \subseteq \, a^{-m}a^{m}I \, = \, I \, \subseteq \, I_{0}.
\end{displaymath} 
Hence, $(I'',\ldots,a^{m-1}I'', J)\perp$. Moreover, as $I' \subseteq a^{m}I$, \begin{align*}
	a^{m}I''\cap (I''+\ldots+a^{m-1}I''+J) \, &\subseteq \, I' \cap\,(I+\ldots+a^{m-1}I+J)\\
			&= \, I'\cap((a^{m}I)\cap(I+\ldots+a^{m-1}I+J)) \, = \, \{0\}.
\end{align*} Thus, $(I'',\ldots,a^{m-1}I'',a^mI'',J)\perp$ by Remark~\ref{remark:independence.equivalence.intersection}. This completes the induction.
	
\ref{lemma:independence.halperin.2} By Lemma~\ref{lemma:independence.inductive}, we find $I \in \lat(R)$ maximal such that $(I,J)$ is $(a,m)$-independent and $I \subseteq J_{1}$. Our assumption entails that \begin{displaymath}
	a^{m}I \, \subseteq \, \sum\nolimits_{i=0}^{m-1} Ka_{i}I + J \, = \, \sum\nolimits_{i=0}^{m-1} a_{i}I + J \, \eqdef \, J' .
\end{displaymath} Since $aJ \subseteq J$, we conclude that $aJ' \subseteq J'$. Also, $J' \subseteq J_{1}$. We claim that $J' = J_{1}$. Suppose for contradiction that $J' \ne J_{1}$. Then, by~\ref{lemma:independence.halperin.1}, there exists $I'\in\lat(R)$ such that $(I',J')$ is $(a,m)$-independent and $\{ 0 \} \ne I' \subseteq J_{1}$. Thus, Lemma~\ref{lemma:independence.sum} implies that $I+I'$ is $(a,m)$-independent. As $I \subsetneq I+I' \subseteq J_{1}$, this contradicts maximality of $I$. Therefore, $J' = J_{1}$ as claimed. Since $I$ is $(a,m)$-independent, thus $J_{1} = J' = \bigoplus\nolimits_{i=0}^{m-1}a^{i} I \oplus J$. \end{proof}
	
Finally, let us record a sufficient condition for the hypothesis of Lemma~\ref{lemma:independence.halperin}. Recall that the center of an irreducible, regular ring constitutes a field by Remark~\ref{remark:irreducible.center.field}.

\begin{lem}\label{lemma:sufficient.condition.halperin} Let $R$ be an irreducible, continuous ring, let $K\defeq\ZZ(R)$, let $p\in K[X]$ be irreducible with $m\defeq \deg(p)$, and let $a\in R$. If there exists some $n\in \N_{>0}$ such that $p^{n}(a) = 0$, then \begin{displaymath}
	\forall q \in K[X]\setminus\{0\} \colon \quad \deg(q)<m \ \Longrightarrow \ \rk_{R}(q(a)) = 1 .
\end{displaymath} \end{lem}

\begin{proof} Let $q \in K[X]\setminus\{0\}$ with $\deg(q)<m$. By Remark~\ref{remark:bijection.annihilator} and Remark~\ref{remark:regular.idempotent.ideals}, there is $e \in \E(R)$ such that $\rAnn(q(a)) = eR$. Consider the ideal $I \defeq \{ f \in K[X] \mid f(a)e = 0 \}$ in $K[X]$. Suppose that $I \neq K[X]$. Since $K$ is a field by Remark~\ref{remark:irreducible.center.field} and thus $K[X]$ is a principal ideal domain, we now find $f \in K[X]$ such that $I=K[X]f$. From $0 \neq q \in I = K[X]f$, we infer that $\deg(f) \leq \deg(q) < m$ and $f \ne 0$. Moreover, $p^{n} \in I = K[X]f$ for some $n \in \N_{>0}$. Since $K[X]$ is a unique factorization domain, $f$ is not a unit in $K[X]$ (as $I\neq K[X]$), and $p$ is irreducible, we conclude that $f\in K[X]p$ and therefore $m = \deg(p) \leq \deg(f) < m$, which is absurd. Consequently, $I=K[X]$ and so $1\in I$. This necessitates that $e=0$ and hence $\rAnn(q(a))=\{0\}$, so that \begin{equation*}
	\rk_{R}(q(a)) \, \stackrel{\ref{lemma:pseudo.dimension.function}\ref{lemma:rank.dimension.annihilator}}{=} \, 1-\delta_{\rk_{R}}(\rAnn(q(a))) \, = \, 1. \qedhere
\end{equation*} \end{proof}

\section{Algebraic elements}\label{section:algebraic.elements}

Considering a field $K$ and a unital $K$-algebra $R$, an element $a\in R$ is said to be \emph{algebraic over $K$} if there exists $p\in K[X]\setminus\{0\}$ such that $p(a)=0$. The purpose of this section is to prove that the algebraic elements in a non-discrete irreducible, continuous ring are precisely those which are contained in some subalgebra isomorphic to a finite product of matrix algebras (Theorem~\ref{theorem:matrixrepresentation.case.algebraic}). Let us recall that, if $R$ is an irreducible, regular ring, then its center $\ZZ(R)$ constitutes a field by Remark~\ref{remark:irreducible.center.field} and the ring $R$ will be viewed as a unital $\ZZ(R)$-algebra.

In order to clarify some terminology, let $K$ be a field. Then a $K$-algebra $R$ is called \emph{matricial}~\cite[Chapter~15, p.~217, Definition]{GoodearlBook} if there exist $m\in\N_{>0}$ and $n_{1},\ldots,n_{m} \in \N_{>0}$ such that $R \cong_{K} \prod\nolimits_{i=1}^{m} \M_{n_{i}}(K)$, where (as in the following) the symbol $\cong_{K}$ indicates $K$-algebra isomorphism. As follows from a classical theorem (see, e.g.,~\cite[IX.1, Corollary~1.5, p.~361]{GrilletBook}), a $K$-algebra is both matricial and simple if and only if it is isomorphic to $\M_{n}(K)$ for some $n \in \N_{>0}$.

\begin{remark}\label{remark:matricial} \begin{enumerate}
	\item\label{remark:matricial.decomposition} Let $K$ be a field. A $K$-algebra $R$ is matricial if and only if there exist $m\in\N_{>0}$, $f_{1},\ldots,f_{m} \in \E(R)\setminus \{ 0 \}$ pairwise orthogonal with $1 = \sum_{i=1}^{m} f_{i}$, and simple, matricial unital $K$-subalgebras $R_{1}\leq f_{1}Rf_{1}, \, \ldots, \, R_{m}\leq f_{m}Rf_{m}$ such that $S = R_{1} + \ldots + R_{m}$.
	\item\label{remark:matricial.sum} Let $R$ be an irreducible, regular ring. For any $m\in\N_{>0}$, pairwise orthogonal elements $f_{1},\ldots,f_{m} \in \E(R)\setminus \{ 0 \}$ with $1 = \sum_{i=1}^{m} f_{i}$, and matricial unital $\ZZ(R)$-subalgebras $R_{1}\leq f_{1}Rf_{1}, \, \ldots, \, R_{m}\leq f_{m}Rf_{m}$, the set $R_{1} + \ldots + R_{m}$ is a matricial unital $\ZZ(R)$-subalgebra of $R$. This is a consequence of Lemma~\ref{lemma:sum.embedding}.
\end{enumerate} \end{remark}

Based on the terminology above, we give the following definition.

\begin{definition}\label{definition:matricial} Let $R$ be an irreducible, regular ring. An element of $R$ will be called \emph{matricial} if it is contained in some matricial unital $\ZZ(R)$-subalgebra of $R$. An element of $R$ will be called \emph{simply matricial} if it is contained in some simple, matricial unital $\ZZ(R)$-subalgebra of $R$. \end{definition}

We will show that, in a non-discrete irreducible, continuous ring, the set of algebraic elements coincides with the set of matricial elements. This result, Theorem~\ref{theorem:matrixrepresentation.case.algebraic}, which provides a central ingredient in the proof of Theorem~\ref{theorem:matricial.dense}, will be verified via four intermediate steps, detailed in Lemmata~\ref{lemma:matrixrepresentation.case.p=0}--\ref{lemma:matrixrepresentation.case.p^n=0}. The following facts are needed.

\begin{remark}[\cite{SchneiderGAFA}, Lemma~8.4(2)]\label{remark:root.K[X]X+K.invertible} Let $K$ be a field, let $R$ be a unital $K$-algebra, and let $a\in R$. If $p(a) = 0$ for some $p\in K[X]\cdot X+(K\setminus\{0\})$, then $a\in\GL(R)$. \end{remark}

\begin{lem}\label{lemma:properties.polynomials} Let $K$ be a field, $R$ be a unital $K$-algebra, $a\in R$ and $p\in K[X]$. \begin{enumerate}
	\item\label{lemma:annihilator.invariant} $a\rAnn(p(a))\subseteq \rAnn(p(a))$.
	\item\label{lemma:evaluation.eRe} Let $e\in\E(R)$. If $eae = ae$, then $ep(a)e = p(a)e = p_{eRe}(ae)$.
	\item\label{lemma:bezout} Let $q\in K[X]$ be such that $p$ and $q$ are coprime. Then \begin{displaymath}
				\qquad \rAnn((pq)(a)) \, = \, \rAnn(p(a))\oplus\rAnn(q(a)).
			\end{displaymath}
\end{enumerate} \end{lem}

\begin{proof} \ref{lemma:annihilator.invariant} If $x\in \rAnn(p(a))$, then $p(a)ax=ap(a)x=0$, thus $ax\in \rAnn(p(a))$.
	
\ref{lemma:evaluation.eRe} This is a consequence of~\cite[Proposition~9.3(2)]{SchneiderGAFA} and~\cite[Lemma~10.5(1)]{SchneiderGAFA}.
	
\ref{lemma:bezout} Since $p$ and $q$ are coprime, there exist $s,t \in K[X]$ with $1=sp+tq$. So, \begin{equation}\label{eq--9}
	\forall x \in R \colon \quad x = s(a)p(a)x + t(a)q(a)x.
\end{equation} This directly implies that $\rAnn(p(a)) \cap \rAnn(q(a))= \{ 0 \}$. Moreover, both $\rAnn(p(a))$ and $\rAnn(q(a))$ are contained in \begin{displaymath}
	J \, \defeq \, \rAnn((pq)(a)) \, = \, \rAnn((qp)(a)) ,
\end{displaymath} thus $\rAnn(p(a)) + \rAnn(q(a)) \subseteq J$. Conversely, if $x \in J$, then $s(a)p(a)x \in \rAnn(q(a))$ and $t(a)q(a)x \in \rAnn(p(a))$, thus \begin{displaymath}
	x \, \stackrel{\eqref{eq--9}}{=} \, s(a)p(a)x + t(a)q(a)x \, \in \, \rAnn(q(a)) + \rAnn(p(a)) .
\end{displaymath} Hence, $J = \rAnn(p(a))\oplus\rAnn(q(a))$. \end{proof}

\begin{lem}\label{lemma:matrixrepresentation.case.p=0} Let $R$ be a non-discrete irreducible, continuous ring, let $K \defeq \ZZ(R)$, let $p \in K[X]$ be irreducible, consider $m \defeq \deg(p) \in \N_{>0}$, and let $c_{0},\ldots,c_{m} \in K$ be such that $p = \sum\nolimits_{i=0}^{m} c_{i}X^{i}$. If $a\in R$ and $p(a)=0$, then $a$ is simply matricial in $R$. \end{lem}
	
\begin{proof} Note that $m > 0$ due to irreducibility of $p$. Let $a\in R$ with $p(a)=0$. If $m=1$, then $p = c_{1}X + c_{0}$ with $c_{1} \ne 0$, whence $a = -c_{0}c_{1}^{-1} \in K$, which entails the claim. Therefore, we may and will henceforth assume that $m > 1$, thus $p \in K[X]\cdot X + (K\setminus \{ 0 \})$ by irreducibility of $p$, and so $a \in \GL(R)$ by Remark~\ref{remark:root.K[X]X+K.invertible}. By Lemma~\ref{lemma:independence.inductive}, Lemma~\ref{lemma:sufficient.condition.halperin} and Lemma~\ref{lemma:independence.halperin}\ref{lemma:independence.halperin.2}, there exists $I\in\lat(R)$ such that $R = I\oplus aI \oplus \ldots \oplus a^{m-1}I$. By Remark~\ref{remark:independence.ideals.idempotents}, we find pairwise orthogonal $e_1,\ldots,e_{m}\in\E(R)$ with $e_1+\ldots+e_{m}=1$ and $e_{i}R=a^{i-1}I$ for each $i\in\{1,\ldots,m\}$. Now, let us define $s \in R^{m\times m}$ by setting \begin{displaymath}
	s_{ij} \, \defeq \, a^{i-j}e_{j} \qquad (i,j\in\{1,\ldots,m\}) .
\end{displaymath} For all $i,j\in\{1,\ldots,m\}$, we see that $s_{ij}e_{j}R = a^{i-j}e_{j}R = a^{i-j}a^{j-1}I = a^{i-1}I = e_{i}R$, hence $s_{ij} = s_{ij}e_{j} = e_{i}s_{ij}e_{j}$, i.e., $s_{ij}\in e_{i}Re_{j}$. In particular, $s_{ij}s_{k\ell} = s_{ij}e_{j}e_{k}s_{k\ell} = 0$ whenever $i,j,k,\ell\in\{1,\ldots,m\}$ and $j\ne k$. Moreover, for all $i,j,k \in\{1,\ldots,m\}$, \begin{displaymath}
	s_{ij}s_{jk} \, = \, a^{i-j}e_{j}s_{jk} \, = \, a^{i-j}s_{jk} \, = \, a^{i-j}a^{j-k}e_{k} \, = \, a^{i-k}e_{k} \, = \, s_{ik} .
\end{displaymath} Finally, since $s_{ii} = a^{0}e_{i} = e_{i}$ for every $i \in \{ 1,\ldots,m\}$, we readily conclude that $s_{11}+\ldots +s_{mm}=e_{1}+\ldots+e_{m}=1$. This shows that $s$ is a family of matrix units for $R$. Furthermore, \begin{align*}
	ae_{m}\, &= \, a^{m}a^{1-m}e_{m} \, = \, a^{m}s_{1m}\, \stackrel{p(a)=0}{=} \, \!\left(-c_{m-1}c_m^{-1}a^{m-1}-\ldots-c_{1}c_{m}^{-1}a-c_{0}c_{m}^{-1}\right)\! s_{1m}\\
		&=\, -c_{m-1}c_m^{-1}a^{m-1}s_{1m}-\ldots-c_1c_m^{-1}as_{1m}-c_0c_m^{-1}s_{1m}\\
		&=\, -c_{m-1}c_m^{-1}s_{mm}-\ldots-c_{1}c_m^{-1}s_{2m}-c_{0}c_m^{-1}s_{1m}.
\end{align*} Consequently, \begin{align*}
	a \, &= \, a(e_{1}+\ldots+e_{m}) \, = \, ae_{1} + \ldots + ae_{m-1} + ae_{m} \\
	&= \, s_{21} + \ldots + s_{m,\, m-1} + -c_{m-1}c_m^{-1}s_{mm}-\ldots-c_1c_m^{-1}s_{2m}-c_0c_m^{-1}s_{1m} .
\end{align*} Thus, $a$ is simply matricial in $R$ by Remark~\ref{remark:matrixunits.ring.homomorphism}. \end{proof}

\begin{lem}\label{lemma:matrixrepresentation.case.nilpotent} Let $R$ be an irreducible, regular ring, let $a\in R$, $n \in \N_{>0}$ and $I \in \lat(R)$ such that $R = \bigoplus\nolimits_{i=0}^{n-1} a^{i}I$ and $a^{n-1}I = \rAnn(a)$. Then $a$ is simply matricial in $R$. \end{lem}

\begin{proof} Let $K \defeq \ZZ(R)$. By Remark~\ref{remark:independence.ideals.idempotents}, we find pairwise orthogonal $e_{1},\ldots,e_{n}\in\E(R)$ with $e_{1}+\ldots+e_{n}=1$ and $e_{i}R=a^{i-1}I$ for each $i\in\{1,\ldots,n\}$. Since \begin{displaymath}
	\rAnn(a) \, = \, a^{n-1}I \, = \, e_{n}R \, = \, (1-(e_{1}+\ldots +e_{n-1}))R ,
\end{displaymath} Lemma~\ref{lemma:partial.inverse} asserts the existence of some $b \in R$ such that $ba = e_{1}+\ldots +e_{n-1}$. By induction, we see that \begin{equation}\label{induction.partial.inverse}
	\forall j \in \{ 1,\ldots,n \} \colon \quad b^{j-1}a^{j-1}e_{1} \, = \, e_{1} . 
\end{equation} Indeed, $b^{0}a^{0}e_{1} = e_{1}$, and if $j \in \{ 2,\ldots,n\}$ is such that $b^{j-2}a^{j-2}e_{1} \, = \, e_{1}$, then \begin{align*}
	b^{j-1}a^{j-1}e_{1} \, &= \, b^{j-2}baa^{j-2}e_{1} \, = \, b^{j-2}bae_{j-1}a^{j-2}e_{1} \\
	& = \, b^{j-2}(e_{1}+\ldots +e_{n-1})e_{j-1}a^{j-2}e_{1} \\
	& = \, b^{j-2}e_{j-1}a^{j-2}e_{1} \, = \, b^{j-2}a^{j-2}e_{1} \, = \, e_{1} .
\end{align*} In turn, $b^{j-1}e_{j}R = b^{j-1}a^{j-1}I = b^{j-1}a^{j-1}e_{1}R \stackrel{\eqref{induction.partial.inverse}}{=} e_{1}R$ for every $j \in \{ 1,\ldots,n \}$, thus \begin{equation}\label{induction.partial.inverse.2}
	\forall j \in \{ 1,\ldots,n \} \colon \quad b^{j-1}e_{j} \, = \, e_{1}b^{j-1}e_{j} . 
\end{equation} Now, let us define $s \in R^{n \times n}$ by setting \begin{displaymath}
	s_{ij} \, \defeq \, a^{i-1}b^{j-1}e_{j} \qquad (i,j\in\{1,\ldots,n\}) .
\end{displaymath} For all $i,j\in\{1,\ldots,n\}$, we observe that \begin{displaymath}
	s_{ij} \, = \, a^{i-1}b^{j-1}e_{j} \, \stackrel{\eqref{induction.partial.inverse.2}}{=} \, a^{i-1}e_{1}b^{j-1}e_{j} \, = \, e_{i}a^{i-1}e_{1}b^{j-1}e_{j} \, \in \, e_{i}Re_{j} .
\end{displaymath} In particular, $s_{ij}s_{k\ell} = s_{ij}e_{j}e_{k}s_{k\ell} = 0$ whenever $i,j,k,\ell\in\{1,\ldots,n\}$ and $j\ne k$. Moreover, for all $i,j,k \in\{1,\ldots,n\}$, \begin{align*}
	s_{ij}s_{jk} \, &= \, a^{i-1}b^{j-1}e_{j}s_{jk} \, = \, a^{i-1}b^{j-1}s_{jk} \, = \, a^{i-1}b^{j-1}a^{j-1}b^{k-1}e_{k} \\
	& \stackrel{\eqref{induction.partial.inverse.2}}{=} \, a^{i-1}b^{j-1}a^{j-1}e_{1}b^{k-1}e_{k} \, \stackrel{\eqref{induction.partial.inverse}}{=} \, a^{i-1}e_{1}b^{k-1}e_{k} \, \stackrel{\eqref{induction.partial.inverse.2}}{=} \, a^{i-1}b^{k-1}e_{k} \, = \, s_{ik} .
\end{align*} Finally, for each $i \in \{ 1,\ldots,n \}$, we find $x \in R$ such that $e_{i} = a^{i-1}e_{1}x$, wherefore \begin{displaymath}
	s_{ii} \, = \, a^{i-1}b^{i-1}e_{i} \, = \, a^{i-1}b^{i-1}a^{i-1}e_{1}x \, \stackrel{\eqref{induction.partial.inverse}}{=} \, a^{i-1}e_{1}x \, = \, e_{i} .
\end{displaymath} We conclude that $s_{11}+\dots +s_{nn}=e_{1}+\ldots+e_{n}=1$. Hence, $s$ is a family of matrix units for $R$. Moreover, $ae_{n} = 0$ as $e_{n} \in a^{n-1}I = \rAnn(a)$, and so \begin{displaymath}
	a \, = \, a(e_{1}+\ldots+e_{n}) \, = \, ae_{1} + \ldots + ae_{n-1}  \, = \, s_{21} + \ldots + s_{n,\, n-1} .
\end{displaymath} Thus, $a$ is simply matricial in $R$ by Remark~\ref{remark:matrixunits.ring.homomorphism}. \end{proof}

\begin{lem}\label{lemma:matrixrepresentation.case.tower} Let $R$ be an irreducible, continuous ring, let $K \defeq \ZZ(R)$, let $p\in K[X]$ be irreducible with $m \defeq \deg(p)$, let $a\in R$, $n \in \N_{>0}$ and $I \in \lat(R)$ be such that \begin{displaymath}
	R = \bigoplus\nolimits_{j=0}^{n-1} \bigoplus\nolimits_{i=0}^{m-1} a^{i}p(a)^{j}I , \qquad \rAnn (p(a)) = \bigoplus\nolimits_{i=0}^{m-1} a^{i}p(a)^{n-1}I .
\end{displaymath} Then $a$ is simply matricial in $R$. \end{lem}

\begin{proof} First of all, let us note that $p(a)^{n-1}I \subseteq \bigoplus\nolimits_{i=0}^{m-1} a^{i}p(a)^{n-1}I = \rAnn(p(a))$ and thus $p(a)^nI=p(a)p(a)^{n-1}I=\{0\}$, wherefore \begin{displaymath}
	p(a)^{n}R \, = \, \sum\nolimits_{j=0}^{n-1}\sum\nolimits_{i=0}^{m-1} p(a)^{n}a^{i}p(a)^{j}I \, = \, \sum\nolimits_{j=0}^{n-1}\sum\nolimits_{i=0}^{m-1}a^{i}p(a)^{j}p(a)^{n}I \, = \, \{ 0 \} .
\end{displaymath} That is, $p^{n}(a) = p(a)^{n} = 0$. Now, let $c_{0},\ldots,c_{m} \in K$ be such that $p = \sum\nolimits_{i=0}^{m} c_{i}X^{i}$. By irreducibility of $p$, either $p = c_{1}X$ with $c_{1} \ne 0$, or $p\in K[X]\cdot X+(K\setminus\{0\})$. In the first case, $R = \bigoplus\nolimits_{j=0}^{n-1} a^{j}I$ and $\rAnn(a) = a^{n-1}I$, thus $a$ is simply matricial in $R$ by Lemma~\ref{lemma:matrixrepresentation.case.nilpotent}. Therefore, we henceforth assume that $p\in K[X]\cdot X+(K\setminus\{0\})$. Then $p^{n}\in K[X]\cdot X+(K\setminus\{0\})$ and so $a\in \GL(R)$ by Remark~\ref{remark:root.K[X]X+K.invertible}. Consider the bijection \begin{displaymath}
	\{1,\ldots,m\}\times \{1,\ldots,n\} \, \longrightarrow\, \{1,\ldots,mn\}, \quad (i,j) \,\longmapsto\, m(j-1)+i .
\end{displaymath} According to Remark~\ref{remark:independence.ideals.idempotents}, there exist pairwise orthogonal $e_{1},\ldots,e_{mn}\in \E(R)$ such that $1 = \sum\nolimits_{i=1}^{m}\sum\nolimits_{j=1}^{n} e_{m(j-1)+i}$ and \begin{equation}\label{eq-7}
	\forall i\in\{1,\ldots,m\} \ \forall j\in\{1,\ldots,n\}\colon \quad e_{m(j-1)+i}R \, = \, a^{i-1}p(a)^{j-1}I.
\end{equation} Consider $f \defeq \sum\nolimits_{i=1}^{m} \sum\nolimits_{j=1}^{n-1} e_{m(j-1)+i} = 1 - \sum\nolimits_{i=1}^{m} e_{m(n-1)+i} \in \E(R)$ and note that \begin{align*}
	(1-f)R \, &= \, \!\left(\sum\nolimits_{i=1}^{m} e_{m(n-1)+i}\right)\!R \\
	& = \, \bigoplus\nolimits_{i=1}^{m} e_{m(n-1)+i}R \, \stackrel{\eqref{eq-7}}{=} \, \bigoplus\nolimits_{i=0}^{m-1} a^{i}p(a)^{n-1}I \, = \, \rAnn (p(a)) .
\end{align*} Hence, by Lemma~\ref{lemma:partial.inverse}, there exists $b\in R$ with $bp(a) = f$. By induction, we see that \begin{equation}\label{eq6}
	\forall \ell\in\{1,\ldots,n\} \colon \quad b^{\ell-1}p^{\ell-1}(a)e_{1} \, = \, e_{1}.
\end{equation} Indeed, $b^{0}p^{0}(a)e_{1} = e_{1}$, and if $j \in \{ 2,\ldots,n \}$ satisfies $b^{j-2}p(a)^{j-2}e_{1} = e_{1}$, then \begin{displaymath}
	b^{j-1}p(a)^{j-1}e_{1} \, = \, b^{j-2}f\smallunderbrace{p(a)^{j-2}e_{1}}_{\in fR} \, = \, b^{j-2}p(a)^{j-2}e_{1} \, = \, e_{1}.
\end{displaymath} For any $i \in \{ 1,\ldots,m \}$ and $j \in \{ 1,\ldots,n\}$, from \begin{displaymath}
	b^{j-1}a^{-(i-1)}e_{m(j-1)+i} \, \stackrel{\eqref{eq-7}}{\in}\, b^{j-1}p(a)^{j-1}I \, \stackrel{\eqref{eq-7}}{=} \, b^{j-1}p(a)^{j-1}e_{1}R \, \stackrel{\eqref{eq6}}{=} \, e_{1}R
\end{displaymath} we infer that \begin{equation}\label{eq-6}
	b^{j-1}a^{-(i-1)}e_{m(j-1)+i} \, = \, e_{1}b^{j-1}a^{-(i-1)}e_{m(j-1)+i} .
\end{equation} Now, for any $i,i'\in\{1,\ldots,m\}$ and $j,j'\in\{1,\ldots,n\}$, we define
\begin{displaymath}
	s_{m(j-1)+i,\,m(j'-1)+i'} \, \defeq \, a^{i-1}p(a)^{j-1}b^{j'-1}a^{-(i'-1)}e_{m(j'-1)+i'}
\end{displaymath} and note that, by~\eqref{eq-7} and~\eqref{eq-6}, \begin{equation}\label{eq-range}
	s_{m(j-1)+i,\,m(j'-1)+i'} \, \in \, e_{m(j-1)+i}Re_{m(j'-1)+i'} .
\end{equation} Evidently, for any choice of $i,i',i'',i'''\in\{1,\ldots,m\}$ and $j,j',j'',j'''\in\{1,\ldots,n\}$ with $(i',j')\neq(i'',j'')$, from $e_{m(j'-1)+i'}\perp e_{m(j''-1)+i''}$, it follows that \begin{displaymath}
	s_{m(j-1)+i,\,m(j'-1)+i'} \cdot s_{m(j''-1)+i'',\,m(j'''-1)+i'''} \, \stackrel{\eqref{eq-range}}{=} \, 0 .
\end{displaymath} Moreover, for all $i,i',i''\in\{1,\ldots,m\}$ and $j,j',j''\in\{1,\ldots,n\}$, \begin{align*}
	&s_{m(j-1)+i,\,m(j'-1)+i'}\cdot s_{m(j'-1)+i',\, m(j''-1)+i''} \\
	&\quad = \, a^{i-1}p(a)^{j-1}b^{j'-1}a^{-(i'-1)}e_{m(j'-1)+i'}a^{i'-1}p(a)^{j'-1}b^{j''-1}a^{-(i''-1)}e_{m(j''-1)+i''} \\
	&\quad \stackrel{\eqref{eq-6}}{=} \, a^{i-1}p(a)^{j-1}b^{j'-1}a^{-(i'-1)}e_{m(j'-1)+i'}a^{i'-1}p(a)^{j'-1}e_{1}b^{j''-1}a^{-(i''-1)}e_{m(j''-1)+i''} \\
	&\quad \stackrel{\eqref{eq-7}}{=} \, a^{i-1}p(a)^{j-1}b^{j'-1}a^{-(i'-1)}a^{i'-1}p(a)^{j'-1}e_{1}b^{j''-1}a^{-(i''-1)}e_{m(j''-1)+i''} \\
	&\quad = \, a^{i-1}p(a)^{j-1}b^{j'-1}p(a)^{j'-1}e_{1}b^{j''-1}a^{-(i''-1)}e_{m(j''-1)+i''} \\
	&\quad \stackrel{\eqref{eq6}}{=} \, a^{i-1}p(a)^{j-1}e_{1}b^{j''-1}a^{-(i''-1)}e_{m(j''-1)+i''} \\
	&\quad \stackrel{\eqref{eq-6}}{=} \, a^{i-1}p(a)^{j-1}b^{j''-1}a^{-(i''-1)}e_{m(j''-1)+i''} \, = \, s_{m(j-1)+i,\,m(j''-1)+i''}.
\end{align*} Also, for each pair $(i,j)\in\{1,\ldots,m\} \times \{1,\ldots,n\}$, thanks to~\eqref{eq-7} we find $x\in R$ such that $e_{m(j-1)+i} = a^{i-1}p(a)^{j-1}e_{1}x$, which entails that \begin{align}\label{eq9}
	s_{m(j-1)+i,\,m(j-1)+i} \, &= \, a^{i-1}p(a)^{j-1}b^{j-1}a^{-(i-1)}e_{m(j-1)+i} \nonumber \\
	& = \, a^{i-1}p(a)^{j-1}b^{j-1}a^{-(i-1)} a^{i-1}p(a)^{j-1}e_{1}x \nonumber\\
	&= a^{i-1}p(a)^{j-1}b^{j-1}p(a)^{j-1}e_{1}x \, \stackrel{\eqref{eq6}}{=} \, a^{i-1}p(a)^{j-1}e_{1}x \nonumber \\
	& = \, e_{m(j-1)+i}
\end{align} In turn, \begin{displaymath}
	\sum\nolimits_{i=1}^{m}\sum\nolimits_{j=1}^{n} s_{m(j-1)+i,\,m(j-1)+i} \, \stackrel{\eqref{eq9}}{=} \, \sum\nolimits_{i=1}^{m}\sum\nolimits_{j=1}^{n} e_{m(j-1)+i} \, = \, 1 ,
\end{displaymath} and so $\left(s_{m(j-1)+i,\,m(j'-1)+i'}\right)_{i,i'\in \{1,\ldots,m\},\, j,j'\in \{1,\ldots,n\}}$ constitutes a family of matrix units for $R$. Now, if $j\in\{1,\ldots,n\}$ and $i,k\in\{1,\ldots,m\}$ are such that $i+k \leq m$, then
\begin{align}\label{eq10}
	a^{k}s_{m(j-1)+i,\,m(j-1)+i} \, &= \, a^{k+i-1}p(a)^{j-1}b^{j-1}a^{-(i-1)}e_{m(j-1)+i} \nonumber \\ 
	&= \, s_{m(j-1)+k+i,\,m(j-1)+i}.
\end{align} In particular, for all $j \in \{ 1,\ldots,n\}$ and $i \in \{ 1,\ldots,m-1\}$, \begin{align}\label{eq66}
	a e_{m(j-1)+i} \, \stackrel{\eqref{eq9}}{=} \, as_{m(j-1)+i,\,m(j-1)+i} \, \stackrel{\eqref{eq10}}{=} \, s_{m(j-1)+i+1,\,m(j-1)+i}.
\end{align} Furthermore, for every $j\in\{1,\ldots,n-1\}$, \begin{align}\label{eq11}
	p(a)s_{m(j-1)+1,\,m(j-1)+m} \, &= \, p(a) p(a)^{j-1}b^{j-1}a^{-(m-1)}e_{m(j-1)+m}\nonumber\\
	&= \, s_{mj+1,\,m(j-1)+m}
\end{align} and hence \begin{align}\label{67}
	& ae_{m(j-1)+m} \, \stackrel{\eqref{eq9}}{=} \, as_{m(j-1)+m,\,m(j-1)+m} \, \stackrel{\eqref{eq10}}{=} \, a^{m}s_{m(j-1)+1,\,m(j-1)+m}\nonumber\\
	&\qquad = \, \! \left(c_{m}^{-1}p(a)-\sum\nolimits_{i=0}^{m-1}c_{m}^{-1}c_{i}a^{i}\right)\!s_{m(j-1)+1,\,m(j-1)+m}\nonumber\\
	&\qquad \stackrel{\eqref{eq10}+\eqref{eq11}}{=} \, c_{m}^{-1}s_{mj+1,\,m(j-1)+m}-\sum\nolimits_{i=0}^{m-1}c_{m}^{-1}c_{i}s_{m(j-1)+i+1,\,m(j-1)+m}.
\end{align} Finally, since $s_{m(n-1)+1, \, m(n-1)+m} \stackrel{\eqref{eq-range}}{\in} e_{m(n-1)+1}R \stackrel{\eqref{eq-7}}{=} p(a)^{n-1}I \subseteq \rAnn(p(a))$, \begin{align}\label{68}
	ae_{m(n-1)+m} \, &\stackrel{\eqref{eq9}}{=} \, as_{m(n-1)+m, \, m(n-1)+m} \, \stackrel{\eqref{eq10}}{=} \, a^{m}s_{m(n-1)+1,\,m(n-1)+m} \nonumber\\
		&= \, \left(c_{m}^{-1}p(a)-\sum\nolimits_{i=0}^{m-1}c_{m}^{-1}c_{i}a^{i}\right) \! s_{m(n-1)+1, \, m(n-1)+m} \nonumber \\
		&\stackrel{\eqref{eq10}}{=} \, \sum\nolimits_{i=0}^{m-1}-c_{m}^{-1}c_{i}s_{m(n-1)+i+1,\,m(n-1)+m}.
\end{align} Combining~\eqref{eq66}, \eqref{67} and~\eqref{68}, we conclude that \begin{align*}
	a \, &= \, a\sum\nolimits_{i=1}^{m}\sum\nolimits_{j=1}^{n} e_{m(j-1)+i} \, = \, \sum\nolimits_{i=1}^{m}\sum\nolimits_{j=1}^{n} ae_{m(j-1)+i} \\
		&= \, \sum\nolimits_{i=1}^{m-1}\sum\nolimits_{j=1}^{n}s_{m(j-1)+i+1,\,m(j-1)+i}\\
		&\qquad +\sum\nolimits_{j=1}^{n-1}\left(c_m^{-1}s_{mj+1,\,m(j-1)+m}-\sum\nolimits_{i=0}^{m-1}c_m^{-1}c_{i}s_{m(j-1)+k+1,\,m(j-1)+m}\right)\\
		&\qquad +\sum\nolimits_{i=0}^{m-1}-c_{m}^{-1}c_{i}s_{m(n-1)+k+1,\,m(n-1)+m}.
\end{align*} Consequently, $a$ is simply matricial in $R$ by Remark~\ref{remark:matrixunits.ring.homomorphism}. \end{proof}

\begin{lem}\label{lemma:matrixrepresentation.case.p^n=0} Let $n \in \N_{>0}$. If $R$ is a non-discrete irreducible, continuous ring, $p\in \ZZ(R)[X]$ is irreducible, and $a\in R$ satisfies $p^{n}(a) = 0$, then $a$ is matricial in $R$. \end{lem}

\begin{proof} We prove the claim by induction on $n \in \N_{>0}$. The induction base, at $n=1$, is provided by Lemma~\ref{lemma:matrixrepresentation.case.p=0}. Now, let $n \in \N_{>0}$ and suppose that the claim has been established for every strictly smaller positive integer. Let $R$ be a non-discrete irreducible, continuous ring, let $K\defeq \ZZ(R)$, let $p\in K[X]$ be irreducible, and let $a\in R$ be such that $p^{n}(a) = 0$. Note that $m \defeq \deg(p) \in \N_{>0}$ by irreducibility of $p$. Let $c_{0},\ldots,c_{m}\in K$ so that $p=\sum\nolimits_{i=0}^mc_iX^i$. By our induction hypothesis, we may and will assume that $n\in\N_{>0}$ is minimal such that $p^n(a)=0$. Moreover, let us note that \begin{equation}\label{equation0}
	\forall x \in R \colon \quad a^{m}x = \sum\nolimits_{i=0}^{m-1} -c_{i}c_{m}^{-1}a^{i}x + c_{m}^{-1}p(a)x .
\end{equation} We now proceed in three preparatory steps.
		
Our first step is to find $I_{1},\ldots,I_{n}\in \lat(R)$ such that \begin{equation}\label{first.step.1}
	\forall j \in \{ 1,\ldots,n \} \colon \quad \rAnn(p(a)^{j}) = I_{j} \oplus aI_{j}\oplus \ldots \oplus a^{m-1}I_{j}\oplus \rAnn(p(a)^{j-1})
\end{equation} and \begin{equation}\label{first.step.2}
	\forall j\in\{2,\ldots,n\} \colon \quad p(a)I_{j} \subseteq I_{j-1} .
\end{equation} The argument proceeds by downward induction starting at $j=n$. By Lemma~\ref{lemma:independence.inductive}, there exists $I_{n} \in \lat(R)$ maximal such that $(I_{n},\rAnn(p(a)^{n-1}))$ is $(a,m)$-independent. For every $x \in R = \rAnn(p^{n}(a)) = \rAnn(p(a)^{n-1}p(a))$, \begin{displaymath}
	a^{m}x \, \stackrel{\eqref{equation0}}{=} \, \sum\nolimits_{i=0}^{m-1} -c_{i}c_{m}^{-1}a^{i}x + c_{m}^{-1}p(a)x \, \in \, \sum\nolimits_{i=0}^{m-1} Ka^{i}x + \rAnn(p(a)^{n-1})
\end{displaymath} Also, \begin{displaymath}
	a\rAnn(p(a)^{n-1}) \, = \, a\rAnn(p^{n-1}(a)) \, \stackrel{\ref{lemma:properties.polynomials}\ref{lemma:annihilator.invariant}}{\subseteq} \, \rAnn(p^{n-1}(a)) \, = \, \rAnn(p(a)^{n-1}) .
\end{displaymath} Using Lemma~\ref{lemma:sufficient.condition.halperin} and Lemma~\ref{lemma:independence.halperin}\ref{lemma:independence.halperin.2}, we deduce that \begin{displaymath}
	\rAnn(p(a)^{n}) \, = \, R \, = \, I_{n} \oplus \ldots \oplus a^{m-1}I_{n} \oplus \rAnn(p(a)^{n-1}).
\end{displaymath} For the induction step, let $j\in\{2,\ldots,n\}$ and suppose that $I_{n},\ldots,I_{j}\in\lat(R)$ have been chosen with the desired properties. In particular, \begin{displaymath}
	\rAnn(p(a)^{j}) \, = \, I_{j} \oplus \ldots \oplus a^{m-1}I_{j} \oplus \rAnn(p(a)^{j-1}).
\end{displaymath} This immediately entails that $p(a)I_{j} \subseteq p(a)\rAnn(p(a)^{j}) \subseteq \rAnn(p(a)^{j-1})$. Moreover, since $\rAnn(p(a)) \subseteq \rAnn(p(a)^{j-1})$, the additive group homomorphism \begin{displaymath}
	\sum\nolimits_{i=0}^{m-1} a^{i}I_{j} \, \longrightarrow \, R, \quad x \, \longmapsto \, p(a)x
\end{displaymath} is injective: if $x\in \sum\nolimits_{i=0}^{m-1} a^{i}I_{j}$ and $p(a)x=0$, then \begin{displaymath}
	x \, \in \, \!\left(\sum\nolimits_{i=0}^{m-1} a^{i}I_{j}\right) \cap \rAnn(p(a)) \, = \, \{ 0 \} ,
\end{displaymath} i.e., $x = 0$. It follows that \begin{displaymath}
	(p(a)I_{j}, p(a)aI_{j},\ldots,p(a)a^{m-1}I_{j})\perp.
\end{displaymath} Since $(I_{j}, \ldots,a^{m-1}I_{j},\rAnn(p(a)^{j-1})\perp$, we also see that \begin{displaymath}
	\left(p(a)I_{j}\oplus \ldots\oplus p(a)a^{m-1}I_{j}\right) \cap \rAnn(p(a)^{j-2}) \, = \, \{0\}.
\end{displaymath} Consequently, by Remark~\ref{remark:independence.equivalence.intersection}, \begin{displaymath}
	(p(a)I_{j},\smallunderbrace{p(a)aI_{j}}_{=\,ap(a)I_{j}},\ldots,\smallunderbrace{p(a)a^{m-1}I_{j}}_{=\,a^{m-1}p(a)I_{j}},\rAnn(p(a)^{j-2}))\perp,
\end{displaymath} i.e., $\left(p(a)I_{j},\rAnn(p(a)^{j-2})\right)$ is $(a,m)$-independent. Thanks to Lemma~\ref{lemma:independence.inductive}, there exists $I_{j-1} \in \lat(R)$ maximal such that $\left(I_{j-1},\rAnn(p(a)^{j-2})\right)$ is $(a,m)$-independent and $p(a)I_{j} \subseteq I_{j-1} \subseteq \rAnn(p(a)^{j-1})$. For every $x \in \rAnn(p(a)^{j}) = \rAnn(p(a)^{j-1}p(a))$, \begin{displaymath}
	a^{m}x \, \stackrel{\eqref{equation0}}{=} \, \sum\nolimits_{i=0}^{m-1} -c_{i}c_{m}^{-1}a^{i}x + c_{m}^{-1}p(a)x \, \in \, \sum\nolimits_{i=0}^{m-1} Ka^{i}x + \rAnn(p(a)^{j-1})
\end{displaymath} Also, \begin{displaymath}
	a\rAnn(p(a)^{j-1}) \, = \, a\rAnn(p^{j-1}(a)) \, \stackrel{\ref{lemma:properties.polynomials}\ref{lemma:annihilator.invariant}}{\subseteq} \, \rAnn(p^{j-1}(a)) \, = \, \rAnn(p(a)^{j-1}) .
\end{displaymath} Applying Lemma~\ref{lemma:sufficient.condition.halperin} and Lemma~\ref{lemma:independence.halperin}\ref{lemma:independence.halperin.2}, we infer that \begin{displaymath}
	\rAnn(p(a)^{j-1}) \, = \, I_{j-1} \oplus \ldots \oplus a^{m-1}I_{j-1} \oplus \rAnn(p(a)^{j-2}) .
\end{displaymath} This finishes the induction.

Our second step is to find $J_{1},\ldots,J_{n}\in\lat(R)$ such that \begin{equation}\label{second.step.1}
	\forall j \in \{ 1,\ldots,n \} \colon \quad I_{j} = p(a)^{n-j}I_{n} \oplus J_{j}
\end{equation} and \begin{equation}\label{second.step.2}
	\forall j \in \{ 2,\ldots,n \} \colon \quad p(a)J_{j} \subseteq J_{j-1}
\end{equation} Again, the argument proceeds by downward induction starting at $j=n$, for which we define $J_{n} \defeq \{0\}$. For the induction step, let $j\in\{2,\ldots,n\}$ and suppose that $J_{n},\ldots,J_{j}\in\lat(R)$ have been chosen with the desired properties. Then \begin{displaymath}
	p(a)^{n-(j-1)}I_{n} \, \subseteq \, p(a)^{n-j}I_{n-1} \, \subseteq \, \ldots \, \subseteq \, p(a)I_{j} \, \subseteq \, I_{j-1}
\end{displaymath} and $I_{j} = p(a)^{n-j}I_{n} \oplus J_{j}$. The additive group homomorphism $I_{j} \to R, \, x \mapsto p(a)x$ is injective: indeed, if $x\in I_{j}$ and $p(a)x=0$, then \begin{displaymath}
	x \, \in \, I_{j} \cap \rAnn(p(a)) \, \subseteq \, I_{j} \cap \rAnn(p(a)^{j-1}) \, \stackrel{\eqref{first.step.1}}{=} \, \{ 0 \} ,
\end{displaymath} i.e., $x = 0$. Since $I_{j} = p(a)^{n-j}I_{n} \oplus J_{j}$, we conclude that $p(a)^{n-(j-1)}I_{n} \cap p(a)J_{j} = \{0\}$. Also, \begin{displaymath}
	p(a)J_{j} \, \stackrel{\eqref{second.step.1}}{\subseteq} \, p(a)I_{j} \, \stackrel{\eqref{first.step.2}}{\subseteq} \, I_{j-1} .
\end{displaymath} So, by Remark~\ref{remark:bijection.annihilator} and Lemma~\ref{lemma:complement}, there exists $J_{j-1}\in\lat(R)$ with $p(a)J_{j} \subseteq J_{j-1}$ and \begin{displaymath}
	I_{j-1} \, = \, p(a)^{n-(j-1)}I_{n} \oplus J_{j-1} ,
\end{displaymath} which finishes the induction.

As our third step, we show that \begin{equation}\label{third.step}
	\forall i \in \{ 0,\ldots,m-1\} \ \forall j \in \{ 1,\ldots,n\} \colon \quad {a^{i}p(a)^{n-j}I_{n}} \cap {a^{i}J_{j}} = \{ 0 \} .
\end{equation} For this purpose, we observe that irreducibility of $p$ implies that either $p = c_{1}X$ with $c_{1} \ne 0$, or $p\in K[X]\cdot X+(K\setminus\{0\})$. In the first case, $m=1$ and so~\eqref{third.step} follows directly from~\eqref{second.step.1}. In the second case, we see that $p^{n}\in K[X]\cdot X+(K\setminus\{0\})$, thus $a\in \GL(R)$ by Remark~\ref{remark:root.K[X]X+K.invertible} and so $a^{i} \in \GL(R)$ for each $i \in \{ 0,\ldots,m-1 \}$, wherefore~\eqref{third.step} is an immediate consequence of~\eqref{second.step.1}, too.
		
Due to the preparations above and the fact that $\rAnn(p(a)^{0}) = \rAnn(1) = \{ 0 \}$, \begin{align*}
	R \, = \, \rAnn(p(a)^{n}) \, &\stackrel{\eqref{first.step.1}}{=} \, \bigoplus\nolimits_{j=1}^{n} \bigoplus\nolimits_{i=0}^{m-1}a^{i}I_{j}\\
		&\stackrel{\eqref{second.step.1}}{=} \, \bigoplus\nolimits_{j=1}^{n} \bigoplus\nolimits_{i=0}^{m-1} a^{i}\!\left(p(a)^{n-j}I_{n}\oplus J_{j}\right)\\
		&\stackrel{\eqref{third.step}}{=} \, \bigoplus\nolimits_{j=1}^{n} \bigoplus\nolimits_{i=0}^{m-1} a^{i}p(a)^{n-j}I_{n}\oplus a^{i}J_{j} \\
		&= \, \! \left(\bigoplus\nolimits_{j=0}^{n-1} \bigoplus\nolimits_{i=0}^{m-1}a^{i}p(a)^{j}I_{n} \right) \! \oplus \! \left(\bigoplus\nolimits_{j=1}^{n} \bigoplus\nolimits_{i=0}^{m-1} a^{i}J_{j} \right) \! .
\end{align*} That is, $R = I \oplus J$ where \begin{align*}
	I \, \defeq \, \bigoplus\nolimits_{j=0}^{n-1} \bigoplus\nolimits_{i=0}^{m-1}a^{i}p(a)^{j}I_{n} , \qquad J \, \defeq \, \bigoplus\nolimits_{j=1}^{n} \bigoplus\nolimits_{i=0}^{m-1} a^{i}J_{j} .
\end{align*} Thanks to Remark~\ref{remark:independence.ideals.idempotents}, there exists $e\in\E(R)$ such that $I=eR$ and $J=(1-e)R$. Since $p(a)^{n} = p^{n}(a) = 0$ and thus $p(a)^{n}I_{n} = \{ 0 \}$, \begin{align*}
	&aI = \sum\nolimits_{j=0}^{n-1} \sum\nolimits_{i=0}^{m-1} aa^{i}p(a)^{j}I_{n} = \left( \sum\nolimits_{j=0}^{n-1} \sum\nolimits_{i=1}^{m-1} a^{i}p(a)^{j}I_{n} \right)\! + \! \left( \sum\nolimits_{j=0}^{n-1} a^{m}p(a)^{j}I_{n} \right) \\
	&\stackrel{\eqref{equation0}}{\subseteq} \left( \sum\nolimits_{j=0}^{n-1} \sum\nolimits_{i=0}^{m-1} a^{i}p(a)^{j}I_{n} \right)\! + \! \left( \sum\nolimits_{j=1}^{n} p(a)^{j}I_{n} \right) \subseteq \sum\nolimits_{j=0}^{n-1} \sum\nolimits_{i=0}^{m-1} a^{i}p(a)^{j}I_{n} = I .
\end{align*} Furthermore, as $p(a)J_{1} \stackrel{\eqref{second.step.1}}{\subseteq} p(a)I_{1} \stackrel{\eqref{first.step.1}}{=} \{ 0 \}$, \begin{align*}
	aJ \, &= \, \sum\nolimits_{j=1}^{n} \sum\nolimits_{i=0}^{m-1} aa^{i}J_{j} \, = \, \! \left( \sum\nolimits_{j=1}^{n} \sum\nolimits_{i=1}^{m-1} a^{i}J_{j} \right)\! + \! \left( \sum\nolimits_{j=1}^{n} a^{m}J_{j} \right) \\
	&\stackrel{\eqref{equation0}}{\subseteq} \, \! \left( \sum\nolimits_{j=1}^{n} \sum\nolimits_{i=0}^{m-1} a^{i}J_{j} \right)\! + \! \left( \sum\nolimits_{j=1}^{n} p(a)J_{j} \right) \\
	& \stackrel{\eqref{second.step.2}}{\subseteq} \, \! \left( \sum\nolimits_{j=1}^{n} \sum\nolimits_{i=0}^{m-1} a^{i}J_{j} \right)\! + \! \left( \sum\nolimits_{j=1}^{n-1} J_{j} \right)\! + p(a)J_{1} \, = \, \sum\nolimits_{j=1}^{n} \sum\nolimits_{i=0}^{m-1} a^{i}J_{j} \, = \, J .
\end{align*} Therefore, $aeR = aI\subseteq I = eR$ and $a(1-e)R = aJ\subseteq J = (1-e)R$. We conclude that $ae=eae$ and $a(1-e)=(1-e)a(1-e)$. Consequently, \begin{equation}\label{eq-4}
	a \, = \, ae + a(1-e) \, = \, eae + (1-e)a(1-e) \, \in \, eRe + (1-e)R(1-e).
\end{equation} This means that $ea=ae$ and thus entails that \begin{equation}\label{equation:polynomial.commutes}
	ep(a) \, = \, p(a)e \, \stackrel{\ref{lemma:properties.polynomials}\ref{lemma:evaluation.eRe}}{=} \, p_{eRe}(ae) \, = \, p_{eRe}(eae) .
\end{equation} From~\eqref{first.step.1} and minimality of $n$, we infer that $I_{n} \ne \{ 0 \}$ and thus $I \ne \{0\}$, whence $e \ne 0$. Therefore, $eRe$ is a non-discrete irreducible, continuous ring by Remark~\ref{remark:eRe.non-discrete.irreducible.continuous}. We observe that $I_{n}e \in \lat(eRe)$ by Lemma~\ref{lemma:right.multiplication.regular}\ref{lemma:right.multiplication.4} and, moreover, \begin{align*}
	eRe \, = \, Ie \, &\stackrel{\ref{lemma:right.multiplication.regular}\ref{lemma:right.multiplication.3}}{=} \, \bigoplus\nolimits_{j=0}^{n-1} \bigoplus\nolimits_{i=0}^{m-1}a^{i}p(a)^{j}I_{n}e \, = \, \bigoplus\nolimits_{j=0}^{n-1} \bigoplus\nolimits_{i=0}^{m-1}a^{i}p(a)^{j}eI_{n}e \\ 
	&\stackrel{\eqref{eq-4}+\ref{lemma:properties.polynomials}\ref{lemma:evaluation.eRe}}{=} \, \bigoplus\nolimits_{j=0}^{n-1} \bigoplus\nolimits_{i=0}^{m-1}(eae)^{i}p_{eRe}(eae)^{j}I_{n}e .
\end{align*} Furthermore, \begin{align*}
	\rAnn(p(a)) \, &\stackrel{\eqref{first.step.1}}{=} \, \bigoplus\nolimits_{i=0}^{m-1}a^{i}I_{1}\, \stackrel{\eqref{second.step.1}}{=} \, \bigoplus\nolimits_{i=0}^{m-1} a^{i}\!\left(p(a)^{n-1}I_{n}\oplus J_{1}\right)\\
		&\stackrel{\eqref{third.step}}{=} \, \bigoplus\nolimits_{i=0}^{m-1} a^{i}p(a)^{n-1}I_{n}\oplus a^{i}J_{1} \\
		&= \, \! \left(\bigoplus\nolimits_{i=0}^{m-1}a^{i}p(a)^{n-1}I_{n} \right) \! \oplus \! \left( \bigoplus\nolimits_{i=0}^{m-1} a^{i}J_{1} \right)
\end{align*} and thus $e\rAnn(p(a)) = \bigoplus\nolimits_{i=0}^{m-1}a^{i}p(a)^{n-1}I_{n}$, wherefore \begin{align*}
	&\rAnn(p_{eRe}(eae)) \cap eRe \, \stackrel{\eqref{equation:polynomial.commutes}}{=} \, \rAnn(p(a)e) \cap eRe \, \stackrel{\eqref{equation:polynomial.commutes}+\ref{lemma:right.multiplication}\ref{lemma:annihilator.eRe}}{=} \, e\rAnn(p(a))e \\
	& \qquad = \, \! \left(\bigoplus\nolimits_{i=0}^{m-1}a^{i}p(a)^{n-1}I_{n}\right)\! e \, \stackrel{\ref{lemma:right.multiplication.regular}\ref{lemma:right.multiplication.3}}{=} \,  \bigoplus\nolimits_{i=0}^{m-1}a^{i}p(a)^{n-1}I_{n}e \\
	& \qquad = \, \bigoplus\nolimits_{i=0}^{m-1}a^{i}p(a)^{n-1}eI_{n}e \, \stackrel{\eqref{eq-4}+\ref{lemma:properties.polynomials}\ref{lemma:evaluation.eRe}}{=} \, \bigoplus\nolimits_{i=0}^{m-1}(eae)^{i}p_{eRe}(eae)^{n-1}I_{n}e .
\end{align*} Hence, $eae$ is matricial in $eRe$ by Lemma~\ref{lemma:matrixrepresentation.case.tower}. Thus, if $1-e = 0$, then $a$ is matricial in $R$, as desired. Suppose now that $1-e \ne 0$, whence $(1-e)R(1-e)$ is a non-discrete irreducible, continuous ring by Remark~\ref{remark:eRe.non-discrete.irreducible.continuous}. Furthermore, since $J_{n} = \{0\}$ by~\eqref{second.step.1}, \begin{align*}
	J \, = \, \bigoplus\nolimits_{j=1}^{n-1} \bigoplus\nolimits_{i=0}^{m-1} a^{i}J_{j} \, &\stackrel{\eqref{second.step.1}+\eqref{first.step.1}}{\subseteq} \, \sum\nolimits_{j=1}^{n-1} \sum\nolimits_{i=0}^{m-1} a^{i}\rAnn(p(a)^{j}) \\
	& \stackrel{\ref{lemma:properties.polynomials}\ref{lemma:annihilator.invariant}}{\subseteq} \, \sum\nolimits_{j=1}^{n-1} \rAnn(p(a)^{j}) \, \subseteq \, \rAnn(p(a)^{n-1})
\end{align*} In particular, $1-e \in J \subseteq \rAnn(p^{n-1}(a))$ and hence \begin{displaymath}
	\left(p^{n-1}\right)_{(1-e)R(1-e)}\!((1-e)a(1-e)) \, \stackrel{\eqref{eq-4}+\ref{lemma:properties.polynomials}\ref{lemma:evaluation.eRe}}{=} \, p(a)^{n-1}(1-e)\, = \, 0.
\end{displaymath} Our induction hypothesis now asserts that $(1-e)a(1-e)$ is matricial in $(1-e)R(1-e)$. According to~\eqref{eq-4} and Remark~\ref{remark:matricial}\ref{remark:matricial.sum}, it follows that $a$ is matricial in $R$, which completes the induction. \end{proof}

Everything is prepared to deduce the desired characterization of algebraic elements of non-discrete irreducible, continuous rings.

\begin{thm}\label{theorem:matrixrepresentation.case.algebraic} Let $R$ be a non-discrete irreducible, continuous ring, let $K\defeq \ZZ(R)$. An element of $R$ is algebraic over $K$ if and only if it is matricial in $R$. \end{thm}

\begin{proof} ($\Longrightarrow$) Let $a \in R$ be algebraic over $K$. Then there exists $f\in K[X]\setminus \{0\}$ such that $f(a)=0$. Since $K$ is a field by  Remark~\ref{remark:irreducible.center.field} and hence $K[X]$ is a principal ideal domain, there exist $m\in\N_{>0}$, $n_{1},\ldots,n_{m}\in \N_{>0}$ and $p_{1},\ldots,p_{m}\in K[X]\setminus \{0\}$ irreducible and pairwise distinct such that $f=p_{1}^{n_{1}}\cdots p_{m}^{n_{m}}$. It follows that $p_{1}^{n_{1}},\ldots, p_{m}^{n_{m}}$ are pairwise coprime. Upon permuting the indices, we may and will assume without loss of generality that \begin{displaymath}
	\{ i \in \{1,\ldots,m\} \mid \rAnn (p_{i}^{n_{i}}(a)) \ne \{ 0 \} \} \, = \, \{ 1,\ldots,\ell \} 
\end{displaymath} for some $\ell \in \{0,\ldots,m\}$. An $m$-fold application of Lemma~\ref{lemma:properties.polynomials}\ref{lemma:bezout} shows that \begin{align*}
	R \, &= \, \rAnn(f(a)) \, = \, \rAnn(p_{1}^{n_{1}}(a))\oplus\rAnn(p_{2}^{n_{2}}(a)\cdot\ldots\cdot p_{m}^{n_{m}}(a)) \, = \, \ldots\\
		&= \, \rAnn(p_{1}^{n_{1}}(a))\oplus \ldots \oplus \rAnn(p_{m}^{n_{m}}(a)) \, = \, \rAnn(p_{1}^{n_{1}}(a))\oplus \ldots\oplus \rAnn(p_{\ell}^{n_{\ell}}(a)) .
\end{align*} In particular, $\ell \ne 0$ as $R$ is non-zero. According to Remark~\ref{remark:independence.ideals.idempotents}, there exist pairwise orthogonal $e_{1},\ldots,e_{\ell}\in\E(R)\setminus \{0\}$ such that $e_{1}+\ldots+e_{\ell}=1$ and $e_{i}R=\rAnn(p_{i}^{n_{i}}(a))$ for each $i\in\{1,\ldots,\ell\}$. Moreover, for every $i\in\{1,\ldots,\ell\}$, \begin{displaymath}
	ae_{i}R \, = \, a\rAnn(p_{i}^{n_{i}}(a)) \, \stackrel{\ref{lemma:properties.polynomials}\ref{lemma:annihilator.invariant}}{\subseteq} \, \rAnn(p_{i}^{n_{i}}(a)) \, = \, e_{i}R,
\end{displaymath} thus $ae_{i} = e_{i}ae_{i}$. Consequently, as $e_{1},\ldots,e_{\ell}$ are pairwise orthogonal, \begin{displaymath}
	a \, = \, (e_{1}+\ldots+e_{\ell})a(e_{1}+\ldots+e_{\ell}) \, = \, e_{1}ae_{1}+\ldots+e_{\ell}ae_{\ell} \in e_{1}Re_{1} + \ldots + e_{\ell}Re_{\ell}.
\end{displaymath} For each $i\in\{1,\ldots,\ell\}$, we see that \begin{displaymath}
	0 \, = \, p^{n_{i}}_{i}(a)e_{i} \, \stackrel{\ref{lemma:properties.polynomials}\ref{lemma:evaluation.eRe}}{=} \, (p_{i}^{n_{i}})_{e_{i}Re_{i}}(e_{i}ae_{i}),
\end{displaymath} whence $e_{i}ae_{i}$ is matricial in $e_{i}Re_{i}$ by Remark~\ref{remark:eRe.non-discrete.irreducible.continuous} and Lemma~\ref{lemma:matrixrepresentation.case.p^n=0}. In turn, $a$ is matricial in $R$ due to Remark~\ref{remark:matricial}\ref{remark:matricial.sum}.
			
($\Longleftarrow$) Any matricial unital $K$-algebra is finite-dimensional over $K$, and any element of a finite-dimensional $K$-algebra is $K$-algebraic. This implies the claim. \end{proof}

\section{Approximation by matrix algebras}\label{section:approximation.by.matrix.algebras}

The purpose of this section is to prove two density results for non-discrete irreducible, continuous rings, namely Theorem~\ref{theorem:matricial.dense} and Corollary~\ref{corollary:simply.special.dense}, which may be viewed as refinements of work of von Neumann~\cite{VonNeumann37} and Halperin~\cite{Halperin62}. We continue to regard an irreducible, continuous ring $R$ as an algebra over its center $\ZZ(R)$, which is a field by Remark~\ref{remark:irreducible.center.field}. The following terminology, inspired by Definition~\ref{definition:matricial}, will be convenient.

\begin{definition}\label{definition:simply.special} Let $R$ be an irreducible, regular ring, and let $K \defeq \ZZ(R)$. An element $a\in R$ will be called \begin{enumerate}[label=---\,]
	\item \emph{special} if there exist $m\in\N_{>0}$, $n_{1},\ldots,n_{m}\in\N_{>0}$ and a unital $K$-algebra embedding $\phi\colon\prod\nolimits_{i=1}^m\M_{n_{i}}(K)\to R$ such that $a\in\phi(\prod\nolimits_{i=1}^{m} \SL_{n_{i}}(K))$,
	\item \emph{simply special} if there exist a positive integer $n\in\N_{>0}$ and a unital $K$-algebra embedding $\phi\colon \M_{n}(K)\to R$ such that $a\in \phi(\SL_{n}(K))$.
\end{enumerate} \end{definition}

Note that any special element of an irreducible, regular ring is invertible. Moreover, we record the following observation for the proof of Theorem~\ref{theorem:decomposition}.

\begin{remark}\label{remark:matricial.conjugation.invariant} Let $R$ be an irreducible, regular ring. The set of matricial (resp., simply matricial, special, simply special) elements of $R$ is invariant under the action of $\GL(R)$ on $R$ by conjugation. \end{remark}

The subsequent observations concerning matricial subalgebras will turn out useful in the proofs of Theorem~\ref{theorem:matricial.dense}, Proposition~\ref{proposition:simply.special.dense}, and Lemma~\ref{lemma:special.decomposition}\ref{lemma:special.decomposition.commutator}.
	
\begin{lem}\label{lemma:matricial.algebra.blow.up} Let $R$ be a non-discrete irreducible, continuous ring, let $K \defeq \ZZ(R)$ and $m,n\in \N_{>0}$. Every unital $K$-subalgebra of $R$ isomorphic to $\M_{n}(K)$ is contained in some unital $K$-subalgebra of $R$ isomorphic to $\M_{mn}(K)$. \end{lem}
		
\begin{proof} Let $S$ be a unital $K$-algebra of $R$ isomorphic to $\M_n(K)$. By Remark~\ref{remark:matrixunits.ring.homomorphism}, there exists a family of matrix units $s \in R^{n\times n}$ for $R$ such that $S=\sum\nolimits_{i,j=1}^{n}Ks_{ij}$. Since $R$ is non-discrete, using Remark~\ref{remark:rank.function.general}\ref{remark:characterization.discrete}, Lemma~\ref{lemma:order}\ref{lemma:order.1} and Remark~\ref{remark:quantum.logic}\ref{remark:quantum.logic.1}, we find pairwise orthogonal elements $e_{1},\ldots,e_{m}\in\E(R)$ such that $s_{11}=\sum\nolimits_{k=1}^{m}e_{k}$ and \begin{displaymath}
	\forall k \in \{1,\ldots,m\} \colon \qquad \rk_{R}(e_{k}) = \tfrac{\rk_{R}(s_{11})}{m}.
\end{displaymath} The former entails that $e_{k} \in s_{11}Rs_{11}$ for each $k \in \{1,\ldots,m\}$. Since $s_{11}Rs_{11}$ is an irreducible, continuous ring by Remark~\ref{remark:eRe.non-discrete.irreducible.continuous} and \begin{displaymath}
	\forall k \in \{1,\ldots,m\} \colon \qquad \rk_{s_{11}Rs_{11}}(e_{k}) \stackrel{\ref{remark:eRe.non-discrete.irreducible.continuous}}{=} \tfrac{1}{\rk_{R}(s_{11})}\rk_{R}(e_{k}) = \tfrac{1}{m},
\end{displaymath} thanks to Lemma~\ref{lemma:matrixunits.idempotents} there exists a family of matrix units $t \in (s_{11}Rs_{11})^{m \times m}$ for $s_{11}Rs_{11}$ such that $t_{kk} = e_{k}$ for each $k \in \{1,\ldots,m\}$. Note that \begin{displaymath}
	\{1,\ldots,n\}\times \{1,\ldots,m\} \, \longrightarrow \, \{1,\ldots,mn\}, \quad (i,k) \, \longmapsto \, m(i-1)+k
\end{displaymath} is a bijection. Let us define $r\in R^{mn\times mn}$ by setting \begin{displaymath}
	r_{m(i-1)+k,\,m(j-1)+\ell} \, \defeq \, s_{i1}t_{k\ell}s_{1j} 
\end{displaymath} for all $i,j\in\{1,\ldots,n\}$ and $k,\ell\in\{1,\ldots,m\}$. We see that \begin{align*}
	r_{m(i-1)+k,\,m(j-1)+\ell} \cdot r&_{m(j-1)+\ell,\,m(i'-1)+k'} \, = \, s_{i1}t_{k\ell}s_{1j}s_{j1}t_{\ell k'}s_{1i'} \, = \, s_{i1}t_{k\ell}s_{11}t_{\ell k'}s_{1i'}\\
	&=s_{i1}t_{k\ell}t_{\ell k'}s_{1i'} \, = \, s_{i1}t_{kk'}s_{1i'} \, = \, r_{m(i-1)+k,m(i'-1)+k'}
\end{align*} for all $i,i',j\in\{1,\ldots,n\}$ and $k,k',\ell\in\{1,\ldots,m\}$. Also, if $i,i',j,j'\in\{1,\ldots,n\}$ and $k,k',\ell,\ell'\in\{1,\ldots,m\}$ are such that $m(j-1)+\ell\neq m(i'-1)+k'$, then either $j\neq i'$ and thus \begin{align*}
	r_{m(i-1)+k,\,m(j-1)+\ell} \cdot r_{m(i'-1)+k',\,m(j'-1)+\ell'} \, = \, s_{i1}t_{k\ell}\smallunderbrace{s_{1j}s_{i'1}}_{=\,0}t_{k'\ell'}s_{1j'} \, = \, 0,
\end{align*} or $j=i'$ and therefore $\ell\neq k'$, whence \begin{align*}
	r_{m(i-1)+k,\, m(j-1)+\ell} \cdot r_{m(i'-1)+k',\, m(j'-1)+\ell'} \, = \, s_{i1}\smallunderbrace{t_{k\ell}t_{k'\ell'}}_{=\,0}s_{1j'} \, = \, 0.
\end{align*} Moreover, for all $i,j\in\{1,\ldots,n\}$, \begin{align}\label{eqstar1}
	s_{ij} \, = \, s_{i1}s_{11}s_{1j} \, = \, s_{i1}\!\left(\sum\nolimits_{k=1}^{m}t_{kk}\right)\!s_{1j} \, &= \, \sum\nolimits_{k=1}^{m}s_{i1}t_{kk}s_{1j}\nonumber\\
				&= \, \sum\nolimits_{k=1}^{m}r_{m(i-1)+k,\,m(j-1)+k}.
\end{align} In particular, we conclude that \begin{displaymath}
	\sum\nolimits_{i=1}^{n}\sum\nolimits_{k=1}^{m} r_{m(i-1)+k,\,m(i-1)+k} \stackrel{\eqref{eqstar1}}{=} \, \sum\nolimits_{i=1}^{n}s_{ii} \, = \, 1.
\end{displaymath} This shows that $r\in R^{mn\times mn}$ is a family of matrix units for $R$. In turn, by Remark~\ref{remark:matrixunits.ring.homomorphism}, \begin{displaymath}
	T \, \defeq \, \sum\nolimits_{i,j=1}^{n}\sum\nolimits_{k,\ell=1}^{m}K r_{m(i-1)+k,\,m(j-1)+\ell}
\end{displaymath} is a unital $K$-subalgebra of $R$ with $T \cong_{K} \M_{mn}(K)$. Finally, \begin{align*}
	S \, &= \, \sum\nolimits_{i,j=1}^{n}Ks_{ij} \, \stackrel{\eqref{eqstar1}}{=} \, \sum\nolimits_{i,j=1}^{n}K\!\left(\sum\nolimits_{k=1}^{m} r_{m(i-1)+k,m(j-1)+k}\right)\\
	&\subseteq \, \sum\nolimits_{i,j=1}^{n}\sum\nolimits_{k=1}^{m}K r_{m(i-1)+k,\,m(j-1)+k} \, \subseteq \, T. \qedhere
\end{align*} \end{proof}	

\begin{lem}\label{lemma:sum.subalgebras.eRe.matricial} Let $m \in \N_{>0}$. Let $R$ be an irreducible, continuous ring, let $K\defeq \ZZ(R)$, $e_{1},\ldots,e_{m} \in \E(R)$ pairwise orthogonal with $1=\sum\nolimits_{i=1}^{m}e_{i}$, and $t,r_{1},\ldots,r_{m} \in \N_{>0}$~with \begin{displaymath}
	\forall i\in\{1,\ldots,m\}\colon\quad \rk_{R}(e_{i}) = \tfrac{r_{i}}{t}.
\end{displaymath} For each $i \in \{ 1,\ldots,m \}$, let $S_{i}$ be a unital $K$-subalgebra of $e_{i}Re_{i}$ with $S_{i} \cong_{K} \M_{r_{i}}(K)$. Then there exists a unital $K$-subalgebra $S\leq R$ with $\sum\nolimits_{i=1}^{m} S_{i} \leq S \cong_{K} \M_{t}(K)$. \end{lem}

\begin{proof} We first prove the claim for $m = 2$. For each $\ell \in \{1,2\}$, let us consider a unital $K$-subalgebra $S_{\ell} \leq e_{\ell}Re_{\ell}$ with $S_{\ell} \cong_{K} \M_{r_{\ell}}(K)$, so that by Remark~\ref{remark:matrixunits.ring.homomorphism} there exists a family of matrix units $s^{(\ell)} \in (e_{\ell}Re_{\ell})^{r_{\ell} \times r_{\ell}}$ for $e_{\ell}Re_{\ell}$ with $S_{\ell} = \sum\nolimits_{i,j=1}^{r_{\ell}} Ks_{i,j}^{(\ell)}$. For every $\ell \in \{ 1,2 \}$, we observe that \begin{displaymath}
	\rk_{R}\!\left(s_{1,1}^{(\ell)}\right)\! \, \stackrel{\ref{remark:eRe.non-discrete.irreducible.continuous}+\ref{remark:properties.pseudo.rank.function}\ref{remark:rank.matrixunits}}{=} \, \tfrac{1}{r_{\ell}}\rk_{R}(e_{\ell}) \, = \, \tfrac{1}{r_{\ell}}\tfrac{r_{\ell}}{t} \, = \, \tfrac{1}{t} .
\end{displaymath} Hence, by Lemma~$\ref{lemma:conjugation.idempotent}$, there exists $g\in \GL(R)$ such that $gs_{1,1}^{(1)}=s_{1,1}^{(2)}g$. Note that \begin{displaymath}
	1 \, = \, \rk_{R}(e_{1}) + \rk_{R}(e_{2}) \, = \, \tfrac{r_{1}}{t} + \tfrac{r_{2}}{t} \, = \, \tfrac{r_{1}+r_{2}}{t},
\end{displaymath} thus $t = r_{1} + r_{2}$.  Now, let us define $\tilde{s}\in R^{t\times t}$ by setting \begin{displaymath}
	\tilde{s}_{i,j} \, \defeq \,
		\begin{cases}
			\, s_{i,j}^{(1)} & \text{if } i,j \in \{1,\ldots,r_{1}\} , \\
			\, s_{i-r_1,j-r_1}^{(2)} & \text{if }  i,j\in\{r_{1}+1,\ldots,t\}, \\
			\, s_{i,1}^{(1)}g^{-1}s_{1,j-r_1}^{(2)}& \text{if } i\in\{1,\ldots,r_{1}\}, \, j\in\{r_{1}+1,\ldots,t\}, \\
			\, s_{i-r_1,1}^{(2)}gs_{1,j}^{(1)}& \text{if } i\in\{r_{1}+1,\ldots,t\}, \, j\in\{1,\ldots,r_1\} .
		\end{cases}
\end{displaymath} Straightforward calculations show that $\tilde{s}$ is a family of matrix units in $R$. Therefore, by Remark~\ref{remark:matrixunits.ring.homomorphism}, $S \defeq \sum\nolimits_{i,j=1}^{t} K\tilde{s}_{i,j}$ is a unital $K$-algebra of $R$ with $S \cong_{K} \M_{t}(K)$. Finally, \begin{displaymath}
	S_{1} + S_{2} \, = \, \sum\nolimits_{i,j=1}^{r_{1}} Ks_{i,j}^{(1)} + \sum\nolimits_{i,j=1}^{r_{2}} Ks_{i,j}^{(2)} \, \subseteq \, \sum\nolimits_{i,j=1}^{t} K\tilde{s}_{i,j} \, = \, S.
\end{displaymath} This completes the proof for $m=2$.

The proof of the general statement proceeds by induction on $m \in \N_{>0}$. For $m=1$, the statement is trivial. For the induction step, suppose that the claim is true for some $m \in \N_{>0}$. Let $e_{1},\ldots,e_{m+1} \in\E(R)$ be pairwise orthogonal with $1=\sum\nolimits_{i=1}^{m+1}e_{i}$ and $t,r_{1},\ldots,r_{m+1} \in \N_{>0}$ be such that $\rk_{R}(e_{i}) = \tfrac{r_{i}}{t}$ for each $i \in \{1,\ldots,m+1\}$. Furthermore, for every $i \in \{ 1,\ldots,m+1 \}$, let $S_{i}$ be a unital $K$-subalgebra of $e_{i}Re_{i}$ with $S_{i} \cong_{K} \M_{r_{i}}(K)$. Considering \begin{displaymath}
	e \, \defeq \, \sum\nolimits_{i=1}^{m} e_{i} \, \stackrel{\ref{remark:quantum.logic}\ref{remark:quantum.logic.2}}{\in} \, \E(R)
\end{displaymath} and $t' \defeq \sum_{i=1}^{m} r_{i} = t\sum_{i=1}^{m} \rk_{R}(e_{i}) = t\rk_{R}(e)$, we see that \begin{displaymath}
	\rk_{eRe}(e_{i}) \, \stackrel{\ref{remark:eRe.non-discrete.irreducible.continuous}}{=} \, \tfrac{1}{\rk_{R}(e)}\rk_{R}(e_{i}) \, = \, \tfrac{t}{t'}\tfrac{r_{i}}{t} \, = \, \tfrac{r_{i}}{t'}
\end{displaymath} for each $i \in \{ 1,\ldots,m \}$. Thus, by our induction hypothesis, there exists a unital $K$-subalgebra $S'$ of $eRe$ such that $\sum_{i=1}^{m} S_{i} \leq S' \cong_{K} \M_{t'}(K)$. Moreover, as \begin{displaymath}
	e \perp e_{m+1}, \qquad \rk_{R}(e) = \tfrac{t'}{t} , \qquad \rk_{R}(e_{m+1}) = \tfrac{r_{m+1}}{t} ,
\end{displaymath} the assertion for the special case proven above yields a unital $K$-subalgebra $S$ of $eRe$ such that $\M_{t}(K) \cong_{K} S \geq S'+S_{m+1} \geq \sum\nolimits_{i=1}^{m+1} S_{i}$, as desired. \end{proof}

We arrive at our main density theorem, which builds on~\cite{VonNeumann37,Halperin62}: the relevant result~\cite[Theorem~8.1]{Halperin62} had been announced in~\cite[\S7, (21)]{VonNeumann37} without proof.

\begin{thm}\label{theorem:matricial.dense} Let $R$ be a non-discrete irreducible, continuous ring. Then the set of simply matricial elements of $R$ is dense in $(R,d_{R})$. \end{thm}

\begin{proof} Define $K\defeq \ZZ(R)$. Let $a\in R$ and $\epsilon\in \R_{>0}$. Due to~\cite[Theorem~8.1]{Halperin62}, there is a $K$-algebraic element $b\in R$ with $d_R(a,b)\leq \tfrac{\epsilon}{2}$. By Theorem~\ref{theorem:matrixrepresentation.case.algebraic} and Remark~\ref{remark:matricial}\ref{remark:matricial.decomposition}, we find $m \in\N_{>0}$, $n_{1},\ldots,n_{m} \in \N_{>0}$, $f_{1},\ldots,f_{m} \in \E(R)\setminus \{ 0 \}$ pairwise orthogonal with $1 = \sum_{i=1}^{m} f_{i}$, and unital $K$-subalgebras $R_{1}\leq f_{1}Rf_{1}, \, \ldots, \, R_{m}\leq f_{m}Rf_{m}$ such that $b \in R_{1} + \ldots + R_{m}$ and $R_{i} \cong_{K} \M_{n_{i}}(K)$ for each $i \in \{ 1,\ldots,m\}$. It follows that $\sum\nolimits_{i=1}^{m}\rk_{R}(f_{i}) = 1$ and $\rk_{R}(f_{i}) > 0$ for each $i \in \{1,\ldots,m\}$. Since $\Q$ is dense in $\R$, there exist $q_{1},\ldots,q_{m} \in \Q_{>0}$ such that $1-\tfrac{\epsilon}{2}\leq \sum\nolimits_{i=1}^{m}q_{i} < 1$ and \begin{displaymath}
	\forall i \in \{1,\ldots,m\} \colon \qquad q_{i} \, \leq \, \rk_{R}(f_{i}).
\end{displaymath} For each $i \in \{1,\ldots,m\}$, recall that $f_{i}Rf_{i}$ is a non-discrete irreducible, continuous ring with $\rk_{f_{i}Rf_{i}} = \tfrac{1}{\rk_R(f_{i})}{{\rk_{R}}\vert_{f_{i}Rf_{i}}}$ and $\ZZ(f_{i}Rf_{i}) = Kf_{i} \subseteq R_{i}$ by Remark~\ref{remark:eRe.non-discrete.irreducible.continuous} and observe that \begin{displaymath}
	\dim_{\ZZ(f_{i}Rf_{i})}(R_{i}) \, = \, \dim_{K}(R_{i}) \, = \, \dim_{K}(\M_{n_{i}}(K)) \, = \, n_{i}^{2} \, < \, \infty,
\end{displaymath} wherefore~\cite[Corollary~7.20(1)]{SchneiderGAFA} and~\cite[Corollary~9.12(1)]{SchneiderGAFA} assert the existence of some $e_{i} \in \E(f_{i}Rf_{i})$ such that $\rk_{f_{i}Rf_{i}}(e_i) = \tfrac{q_{i}}{\rk_R(f_{i})}$, i.e., $\rk_{R}(e_{i}) = q_{i}$, and \begin{equation}\label{star}
	\forall x \in R_{i} \colon \qquad e_{i}xe_{i} \, = \, xe_{i} .
\end{equation} Thus, by~\cite[Lemma~10.5(1)]{SchneiderGAFA}, for each $i\in\{1,\ldots,m \}$, the map \begin{displaymath}
	R_{i} \, \longrightarrow \, e_{i}R_{i}e_{i}, \quad x \, \longmapsto \, xe_{i}
\end{displaymath} constitutes a surjective unital non-zero $K$-algebra homomorphism, hence an isomorphism by simplicity of $R_{i} \cong_{K} \M_{n_{i}}(K)$ (see, e.g.,~\cite[IX.1, Corollary~1.5, p.~361]{GrilletBook}), so that $S_{i} \defeq e_{i}R_{i}e_{i} \cong_{K} R_{i} \cong_{K} \M_{n_{i}}(K)$. Since $e_{1} \in \E(f_{1}Rf_{1}), \, \ldots , \, e_{m} \in \E(f_{m}Rf_{m})$, we see that $e_{1},\ldots,e_{m}$ are pairwise orthogonal, too. Concerning \begin{displaymath}
	e_{m+1} \, \defeq \, 1-\sum\nolimits_{i=1}^{m} e_{i} \, \stackrel{\ref{remark:quantum.logic}}{\in} \, \E(R),
\end{displaymath} we note that \begin{displaymath}
	q_{m+1} \, \defeq \, \rk_{R}(e_{m+1}) \, \stackrel{\ref{remark:properties.pseudo.rank.function}\ref{remark:rank.difference.smaller.idempotent}}{=} \, 1-\sum\nolimits_{i=1}^{m}\rk_{R}(e_{i}) \, = \, 1-\sum\nolimits_{i=1}^{m}q_{i} \, \in \, \!\left(0,\tfrac{\epsilon}{2}\right]\!\cap\Q.
\end{displaymath} As $\rk_{R}(e_{m+1})>0$ and therefore $e_{m+1}\neq 0$, the unital $K$-subalgebra \begin{displaymath}
	S_{m+1} \, \defeq \, Ke_{m+1} \, \leq \, e_{m+1}Re_{m+1}
\end{displaymath} is isomorphic to $K \cong_{K} \M_{n_{m+1}}(K)$, where $n_{m+1}\defeq 1$. Now, consider \begin{displaymath}
	c \, \defeq \, b(1-e_{m+1}) \, = \, \sum\nolimits_{i=1}^{m}be_{i} \, = \, \sum\nolimits_{i=1}^{m}\smallunderbrace{bf_{i}}_{\in R_{i}}\!e_{i} \, \stackrel{\eqref{star}}{=} \, \sum\nolimits_{i=1}^{m} e_{i}bf_{i}e_{i} \, \in \, \sum\nolimits_{i=1}^{m+1} S_{i}
\end{displaymath} and observe that \begin{displaymath}
	d_{R}(b,c) \, = \, \rk_{R}(b-c) \, = \, \rk_{R}(be_{m+1}) \, \leq \, \rk_{R}(e_{m+1}) \, \leq \, \tfrac{\epsilon}{2},
\end{displaymath} thus \begin{displaymath}
	d_{R}(a,c) \, \leq \, d_{R}(a,b) + d_{R}(b,c) \, \leq \, \epsilon.
\end{displaymath} We will show that $c$ is simply matricial, which will finish the proof. To this end, choose $t,r_{1},\ldots,r_{m+1}\in\N_{>0}$ such that \begin{displaymath}
	\forall i\in\{1,\ldots,m+1\}\colon\qquad q_{i} = \tfrac{r_{i}}{t},\quad n_{i}\vert r_{i}.
\end{displaymath} For each $i \in \{1,\ldots,m+1\}$, since $e_{i}Re_{i}$ is a non-discrete irreducible, continuous ring by Remark~\ref{remark:eRe.non-discrete.irreducible.continuous}, \begin{displaymath}
	\ZZ(e_{i}Re_{i}) \, \stackrel{\ref{remark:eRe.non-discrete.irreducible.continuous}}{=} \, Ke_{i} \, \leq \, S_{i} \, \leq \, e_{i}Re_{i},
\end{displaymath} and $S_{i} \cong_{K} \M_{n_{i}}(K)$, we infer from Lemma~\ref{lemma:matricial.algebra.blow.up} the existence of a unital $K$-subalgebra $T_{i} \leq e_{i}Re_{i}$ with $S_{i} \subseteq T_{i} \cong_{K} \M_{r_{i}}(K)$. Moreover, as $e_{1},\ldots,e_{m+1}$ are pairwise orthogonal with $1=\sum\nolimits_{i=1}^{m+1}e_{i}$ and \begin{displaymath}
	\forall i\in\{1,\ldots,m+1\}\colon\qquad\rk_{R}(e_{i}) = q_{i} = \tfrac{r_{i}}{t},
\end{displaymath} by Lemma~$\ref{lemma:sum.subalgebras.eRe.matricial}$ there is a unital $K$-subalgebra $T \leq R$ with $\sum_{i=1}^{m+1}T_{i} \subseteq T \cong_{K} \M_{t}(K)$. We conclude that \begin{displaymath}
	c \, \in \, \sum\nolimits_{i=1}^{m+1} S_{i} \, \subseteq \, \sum\nolimits_{i=1}^{m+1} T_{i} \, \subseteq \, T \, \cong_{K} \, \M_{t}(K) ,
\end{displaymath} so $c$ is simply matricial. \end{proof}
	
We deduce our second density result, which concerns unit groups of non-discrete irreducible, continuous rings.

\begin{prop}\label{proposition:simply.special.dense} Let $R$ be a non-discrete irreducible, continuous ring, let $K \defeq \ZZ(R)$, let $a \in \GL(R)$ and $\epsilon \in \R_{>0}$. Then there exists $n \in \N_{>0}$ such that, for every $m \in \N_{>0}$ with $n \vert m$, there exist a unital $K$-algebra embedding  $\phi \colon \M_{m}(K) \to R$ and an element $A \in \SL_{m}(K)$ such that $d_{R}(a,\phi(A)) < \epsilon$. \end{prop}

\begin{proof} Let $a \in \GL(R)$ and $\epsilon \in \R_{>0}$. By Theorem~\ref{theorem:matricial.dense}, there exist $n_{0} \in \N_{>0}$ and a unital $K$-subalgebra $S \leq R$ isomorphic to $\M_{n_{0}}(K)$ with $\inf_{b \in S} d_{R}(a,b) < \tfrac{\epsilon}{3}$. Define \begin{displaymath}
	n \, \defeq \, n_{0} \cdot \left\lceil\tfrac{3}{\epsilon}\right\rceil \, \in \, \N_{>0}.
\end{displaymath} Now, let $m \in \N_{>0}$ with $n \vert m$. Thanks to Lemma~\ref{lemma:matricial.algebra.blow.up}, there exists a unital $K$-algebra embedding $\phi \colon \M_{m}(K) \to R$ such that $S \leq \phi (\M_{m}(K))$. In particular, we find $B \in \M_{m}(K)$ such that $d_{R}(a,\phi(B)) < \tfrac{\epsilon}{3}$. Since \begin{displaymath}
	\rk_{\M_{m}(K)}(B) \, \stackrel{\ref{remark:rank.function.general}\ref{remark:uniqueness.rank.embedding}}{=} \, \rk_{R}(\phi(B)) \, \stackrel{\ref{remark:properties.pseudo.rank.function}\ref{remark:rank.continuous}}{>} \, \rk_{R}(a)-\tfrac{\epsilon}{3} \, = \, 1-\tfrac{\epsilon}{3},
\end{displaymath} an elementary argument using linear algebra shows the existence of some $C\in \GL_{m}(K)$ such that $d_{\M_{m}(K)}(B,C) < \tfrac{\epsilon}{3}$. Furthermore, upon multiplying one column (or one row) of $C$ by $\det(C)^{-1}$, we obtain $A\in \SL_{m}(K)$ such that \begin{displaymath}
	d_{\M_m(K)}(C,A) \, \leq \, \tfrac{1}{m} \, \leq \, \tfrac{1}{n} \, \leq \, \tfrac{\epsilon}{3}.
\end{displaymath} Therefore, \begin{align*}
	d_{R}(a,\phi(A)) \, &\leq \, d_{R}(a,\phi(B)) + d_{R}(\phi(B),\phi(C)) + d_{R}(\phi(C),\phi(A)) \\
			&\stackrel{\ref{remark:rank.function.general}\ref{remark:uniqueness.rank.embedding}}{=} \, d_{R}(a,\phi(B)) + d_{\M_m(K)}(B,C) + d_{\M_m(K)}(C,A) \, < \, \epsilon . \qedhere
\end{align*} \end{proof}

\begin{cor}\label{corollary:simply.special.dense} Let $R$ be a non-discrete irreducible, continuous ring. The set of simply special elements of $R$ is dense in $(\GL(R),d_{R})$. \end{cor}

\begin{proof} This is an immediate consequence of Proposition~\ref{proposition:simply.special.dense}. \end{proof}

\section{Decomposition into locally special elements}\label{section:decomposition.into.locally.special.elements}

This section is devoted to proving our main decomposition result for unit groups of non-discrete irreducible, continuous rings (Theorem~\ref{theorem:decomposition}) and deducing some structural consequences (Theorem~\ref{theorem:width}). The following notion, which rests on Remark~\ref{remark:eRe.non-discrete.irreducible.continuous}, Lemma~\ref{lemma:convergence.sequences} and Lemma~\ref{lemma:local.decomposition}, is fundamental to these results.

\begin{definition}\label{definition:locally.special} Let $R$ be an irreducible, continuous ring. An element $a\in R$ will be called \emph{locally special} if there exist $(e_{n})_{n\in\N}\in \E(R)^{\N}$ pairwise orthogonal and $(a_{n})_{n \in \N} \in \prod_{n \in \N} e_{n}Re_{n}$ such that \begin{enumerate}
	\item[---\,] for each $n \in \N$, the element $a_n$ is simply special in $e_{n}Re_{n}$,
	\item[---\,] and $a = \prod\nolimits_{n\in \N} a_{n}+1-e_{n}$.
\end{enumerate} \end{definition}

\begin{remark}\label{remark:locally.special} Let $R$ be an irreducible, continuous ring. \begin{enumerate}
	\item\label{remark:locally.special.1} Every special element of $R$ is locally special in $R$. This can be seen using Remark~\ref{remark:eRe.non-discrete.irreducible.continuous} (or, alternatively, Remark~\ref{remark:matrixunits.ring.homomorphism}).
	\item\label{remark:locally.special.2} Every locally special element of $R$ belongs to $\GL(R)$, as follows from Lemma~\ref{lemma:convergence.sequences} and Lemma~\ref{lemma:subgroup.unit.group}. 
\end{enumerate} \end{remark}

For the proof of Theorem~\ref{theorem:decomposition}, we make a few preparations.

\begin{lem}\label{lemma:simply.special.involution} Let $R$ be an irreducible, continuous ring and $e_{1}, e_{2} \in \E(R)$ with $e_{1} \perp e_{2}$ and $\rk_R(e_{1}) = \rk_R(e_{2}) = \tfrac{1}{3}$. Then there exists a simply special involution $u \in \I(R)$ such that $ue_{1}u = e_{2}$. \end{lem}

\begin{proof} Let $K \defeq \ZZ(R)$ and $e_{3} \defeq 1-e_{1}-e_{2}$. Then $e_{3} \in \E(R)$, the elements $e_{1},e_{2},e_{3}$ are pairwise orthogonal, and $1 = e_{1} + e_{2} + e_{3}$. So, $\rk_{R}(e_{3})=\tfrac{1}{3}$. Thanks to Lemma~\ref{lemma:matrixunits.idempotents}, there exists a family of matrix units $s\in R^{3\times 3}$ for $R$ such that $e_{i} = s_{ii}$ for each $i\in\{1,2,3\}$. Let us define $B,C,U\in \M_{3}(K)$ by \begin{displaymath}
	B \defeq \begin{pmatrix}
				1 & 0 & 0 \\
				0 & 0 & 0 \\
				0 & 0 & 0 \\
			\end{pmatrix}\!, \quad C\defeq \begin{pmatrix}
				0 & 0 & 0 \\
				0 & 1 & 0 \\
				0 & 0 & 0 \\
			\end{pmatrix}\!, \quad U\defeq \begin{pmatrix}
				0 & 1 & 0 \\
				1 & 0 & 0 \\
				0 & 0 & -1 \\
			\end{pmatrix}\!.
\end{displaymath} Then $U\in\SL_{3}(K)$, $U=U^{-1}$ and $UBU=C$. The unital $K$-algebra homomorphism \begin{displaymath}
	\phi \colon\, \M_{3}(K)\,\longrightarrow\, R, \quad (a_{ij})_{i,j\in\{1,2,3\}}\,\longmapsto\,\sum\nolimits_{i,j=1}^{3} a_{ij}s_{ij}
\end{displaymath} satisfies $\phi(B)=s_{11}=e_{1}$ and $\phi(C)=s_{22}=e_{2}$. Consequently, $u\defeq\phi(U) \in \I(R)$ is simply special and \begin{displaymath}
	ue_{1}u \, = \, \phi(U)\phi(B)\phi(U) \, = \, \phi(UBU) \, = \, \phi(C) \, = \, e_{2} . \qedhere
\end{displaymath} \end{proof}

\begin{lem}\label{lemma:simply.special.involution.2} Let $R$ be an irreducible, continuous ring and $e,f \in \E(R)\setminus \{ 0 \}$ with $e \perp f $ and $\rk_{R}(e)=2\rk_{R}(f)$. Then there exists $v \in \I(\Gamma_{R}(e+f))$ with $v(e+f)$ simply special in $(e+f)R(e+f)$ such that $vfv \leq e$. \end{lem}

\begin{proof} By Lemma~\ref{lemma:order}\ref{lemma:order.1}, there exists $e' \in \E(R)$ with $e' \leq e$ and $\rk_{R}(e') = \rk_{R}(f)$. Using Remark~\ref{remark:quantum.logic}\ref{remark:quantum.logic.2}, we see that $e+f \in \E(R)$, and both $e' \leq e \leq e+f$ and $f \leq e+f$. Our Remark~\ref{remark:eRe.non-discrete.irreducible.continuous} asserts that $S \defeq (e+f)R(e+f)$ is an irreducible, continuous ring. We calculate that \begin{align*}
	\rk_{R}(e+f) \, &= \, \rk_{R}(e) + \rk_{R}(f) \, = \, \tfrac{3}{2}\rk_{R}(e), \\
	\rk_{S}(e') \, &\stackrel{\ref{remark:eRe.non-discrete.irreducible.continuous}}{=} \, \tfrac{1}{\rk_{R}(e+f)}\rk_{R}(e') \, = \, \tfrac{2\rk_{R}(f)}{3\rk_{R}(e)} \, = \, \tfrac{1}{3}, \\
	\rk_{S}(f) \, &\stackrel{\ref{remark:eRe.non-discrete.irreducible.continuous}}{=} \, \tfrac{1}{\rk_{R}(e+f)}\rk_{R}(f) \, = \, \tfrac{2\rk_{R}(f)}{3\rk_{R}(e)} \, = \, \tfrac{1}{3} .
\end{align*} Moreover, from $e \perp f$ and $e' \leq e$, we deduce that $e' \perp f$. Therefore, by Lemma~\ref{lemma:simply.special.involution}, there exists $u \in \I(S)$ simply special in $S$ such that $e' = ufu$. Inspecting \begin{displaymath}
	v \, \defeq \, u + 1-(e+f) \, \stackrel{\ref{lemma:subgroup.unit.group}}{\in} \, \I(\Gamma_{R}(e+f)) ,
\end{displaymath} we deduce that \begin{align*}
	vfv \, &= \, (u + 1 - (e+f))f(u + 1 - (e+f)) \\
	& = \, ufu + uf - uf + fu + f - f - fu - f + f \, = \, ufu \, = \, e' \, \leq \, e . \qedhere
\end{align*} \end{proof}

\begin{lem}\label{lemma:partial.approximation} Let $R$ be a non-discrete irreducible, continuous ring, let $e \in \E(R)\setminus \{ 0 \}$, $t \in (0,\rk_{R}(e)]$, and $a \in \Gamma_{R}(e)$. Then there exist $b \in \Gamma_{R}(e)$ and $f \in \E(R)$ such that \begin{enumerate}
	\item[---\,] $be$ is simply special in $eRe$,
	\item[---\,] $f \leq e$ and $\rk_{R}(f) = t$,
	\item[---\,] $b^{-1}a \in \Gamma_{R}(f)$.
\end{enumerate} \end{lem}

\begin{proof} Note that $ae \in \GL(eRe)$ by Lemma~\ref{lemma:subgroup.unit.group}. By Corollary~\ref{corollary:simply.special.dense} and Remark~\ref{remark:eRe.non-discrete.irreducible.continuous}, there exists $c \in \GL(eRe)$ simply special in $eRe$ and such that \begin{displaymath}
	\rk_{eRe}(c-ae) \, \leq \, \tfrac{t}{2\rk_{R}(e)}.
\end{displaymath} Defining $b \defeq c+1-e \in \Gamma_{R}(e)$, we see that \begin{displaymath}
	\rk_{R}(b-a) \, = \, \rk_{R}(c-ae) \, \stackrel{\ref{remark:eRe.non-discrete.irreducible.continuous}}{=} \, \rk_{R}(e)\rk_{eRe}(c-ae) \, \leq \, \tfrac{t}{2} .
\end{displaymath} By Lemma~\ref{lemma:invertible.rank.idempotent}, there exists $\bar{e} \in \E(R)$ such that $\bar{e} \leq e$, $b^{-1}a \in \Gamma_{R}({\bar{e}})$, and \begin{displaymath}
	\rk_{R}(\bar{e}) \, \leq \, 2\rk_{R}\!\left(1-b^{-1}a\right)\! \, = \, 2\rk_{R}(b-a) \, \leq \, t.
\end{displaymath} Moreover, by Remark~\ref{remark:rank.function.general}\ref{remark:characterization.discrete} and Lemma~\ref{lemma:order}\ref{lemma:order.1}, there exists $f \in \E(R)$ such that $\bar{e} \leq f \leq e$ and $\rk_{R}(f) = t$. Finally, \begin{displaymath}
	b^{-1}a \, \in \, \Gamma_{R}(\bar{e}) \, \stackrel{\ref{lemma:subgroup.unit.group}\ref{lemma:subgroup.unit.group.order}}{\subseteq} \, \Gamma_{R}(f) . \qedhere
\end{displaymath} \end{proof}

We are ready to embark on the proof of the announced decomposition theorem. Our argument proving Theorem~\ref{theorem:decomposition} is inspired by work of Fathi~\cite[Lemme~9]{Fathi} on corresponding properties of the group~$\Aut([0,1],\lambda)$.

\begin{thm}\label{theorem:decomposition} Let $R$ be a non-discrete irreducible, continuous ring. Then every element $a\in\GL(R)$ admits a decomposition \begin{displaymath}
	a \, = \, bu_{1}v_{1}v_{2}u_{2}v_{3}v_{4}
\end{displaymath} where \begin{enumerate}
	\item[---\,] $b \in \GL(R)$ is simply special,
	\item[---\,] $u_{1},u_{2} \in \GL(R)$ are locally special,
	\item[---\,] $v_{1},v_{2},v_{3},v_{4} \in \GL(R)$ are locally special involutions.
\end{enumerate} \end{thm}
	
\begin{proof} Let $a\in \GL(R)$. Define $a_{1} \defeq a$. By Lemma~\ref{lemma:partial.approximation}, there exist $f_{1} \in \E(R)$ with $\rk_{R}(f_{1}) = \tfrac{1}{8}$ and $b_{1} \in \GL(R)$ simply special with $b_{1}^{-1}a_{1} \in \Gamma_{R}(f_{1})$. Due to Remark~\ref{remark:rank.function.general}\ref{remark:characterization.discrete} and Lemma~\ref{lemma:order}\ref{lemma:order.1}, there exists $e \in \E(R)$ with $f_{1} \leq e$ and $\rk_{R}(e) = \tfrac{1}{2}$. By Lemma~\ref{lemma:order}\ref{lemma:order.2}, we find $(e_{m})_{m\in\N_{>0}}\in\E(R)^{\N_{>0}}$ pairwise orthogonal with $e_{1} = e$ and $\rk_{R}(e_{n}) = 2^{-n}$ for each $n \in \N_{>0}$. Now, by Lemma~\ref{lemma:simply.special.involution.2}, there exists $v_{1} \in \I(\Gamma_{R}(e_{2}+f_{1}))$ with $v_{1}(e_{2}+f_{1})$ simply special in $(e_{2}+f_{1})R(e_{2}+f_{1})$ such that $v_{1}f_{1}v_{1} \leq e_{2}$.

Building on the above, recursively we construct sequences \begin{displaymath}
	(f_{i})_{i \in \N_{>1}} \in \E(R)^{\N_{>1}}, \ \ (a_{i})_{i\in\N_{>1}}, \, (b_{i})_{i\in\N_{>1}} \in \GL(R)^{\N_{>1}}, \ \ (v_{i})_{i\in\N_{>1}} \in \I(R)^{\N_{>1}}
\end{displaymath} such that, for every $i \in \N_{>0}$, \begin{align}
	&f_{i}\leq e_{i}, \label{eq:nesting} \\
	&a_{i+1},b_{i+1}\in\Gamma_{R}(e_{i+1}) , \ \ b_{i}^{-1}a_{i}\in \Gamma_{R}(f_{i}) , \label{eq:support} \\
	&b_{i+1}e_{i+1} \textnormal{ simply special in } e_{i+1}Re_{i+1} , \label{eq:simply.special} \\
	&v_{i} \in \Gamma_{R}(e_{i+1}+f_{i}) , \label{eq:support.involution} \\
	&v_{i}(e_{i+1}+f_{i}) \textnormal{ simply special in } (e_{i+1}+f_{i})R(e_{i+1}+f_{i}) , \label{eq:simply.special.involution} \\
	&v_{i}f_{i}v_{i} \leq e_{i+1} , \label{eq:similar.idempotents} \\
	&a_{i+1} = v_{i}b_{i}^{-1}a_{i}v_{i} . \label{eq:connection} 
\end{align} For the recursive step, let $i\in \N_{>0}$. Consider \begin{displaymath}
	a_{i+1} \, \defeq \, v_{i}b_{i}^{-1}a_{i}v_{i} \, \stackrel{\eqref{eq:support}}{\in} \, v_{i}\Gamma_{R}(f_{i})v_{i} \, \stackrel{\ref{lemma:subgroup.unit.group}\ref{lemma:subgroup.unit.group.conjugation}}{=} \, \Gamma_{R}(v_{i}f_{i}v_{i}) \, \stackrel{\eqref{eq:similar.idempotents}+\ref{lemma:subgroup.unit.group}\ref{lemma:subgroup.unit.group.order}}{\subseteq} \, \Gamma_{R}(e_{i+1}) .
\end{displaymath} By Lemma~\ref{lemma:partial.approximation}, there exist $b_{i+1} \in \Gamma_{R}(e_{i+1})$ with $b_{i+1}e_{i+1}$ simply special in $e_{i+1}Re_{i+1}$ and $f_{i+1} \in \E(R)$ with $f_{i+1} \leq e_{i+1}$ and $\rk_{R}(f_{i+1}) = \tfrac{1}{2}\rk_{R}(e_{i+2})$ such that \begin{displaymath}
	b_{i+1}^{-1}a_{i+1} \, \in \, \Gamma_{R}(f_{i+1}) .
\end{displaymath} Thanks to Lemma~\ref{lemma:simply.special.involution.2}, there exists $v_{i+1} \in \I(\Gamma_{R}(e_{i+2}+f_{i+1}))$ with $v_{i+1}(e_{i+2}+f_{i+1})$ simply special in $(e_{i+2}+f_{i+1})R(e_{i+2}+f_{i+1})$ such that $v_{i+1}f_{i+1}v_{i+1} \leq e_{i+2}$.
		
The remainder of the proof makes use of the fact that the rank ring $(R,{\rk_{R}})$ is complete by Theorem~\ref{theorem:unique.rank.function}. From~\eqref{eq:nesting} and pairwise orthogonality of the sequence~$(e_{n})_{n \in \N_{>0}}$, we infer that both $(e_{2i}+f_{2i-1})_{i\in\N_{>1}}$ and $(e_{2i+1}+f_{2i})_{i\in\N_{>1}}$ are pairwise orthogonal. Using Lemma~\ref{lemma:local.decomposition}, we observe that \begin{displaymath}
	b \, \defeq \, \prod\nolimits_{i=1}^{\infty}b_{2i+1}, \qquad b' \, \defeq \, \prod\nolimits_{i=1}^{\infty}b_{2i}
\end{displaymath} are locally special in $R$, and \begin{displaymath}
	v \, \defeq \, \prod\nolimits_{i=1}^{\infty}v_{2i-1}, \qquad v' \, \defeq \, \prod\nolimits_{i=1}^{\infty}v_{2i}
\end{displaymath} are locally special involutions in $R$. Now, for each $i \in \N_{>0}$, since $v_{i} \in \I(\Gamma_{R}(e_{i+1}+f_{i}))$ by~\eqref{eq:support.involution} and \begin{displaymath}
	b_{i}^{-1}a_{i} \, \stackrel{\eqref{eq:support}}{\in} \, \Gamma_{R}(f_{i}) \, \stackrel{\ref{remark:quantum.logic}\ref{remark:quantum.logic.2}+\ref{lemma:subgroup.unit.group}\ref{lemma:subgroup.unit.group.order}}{\subseteq} \, \Gamma_{R}(e_{i+1}+f_{i}) ,
\end{displaymath} we see that \begin{displaymath}
	u_{i} \, \defeq \, b_{i}^{-1}a_{i}v_{i}\!\left( b_{i}^{-1}a_{i}\right)^{-1}\! \, \in \, \I(\Gamma_{R}(e_{i+1}+f_{i})) 
\end{displaymath} and, moreover, \begin{displaymath}
	u_{i}(e_{i+1}+f_{i}) \, \stackrel{\ref{lemma:subgroup.unit.group}}{=} \, b_{i}^{-1}a_{i}(e_{i+1}+f_{i})v_{i}(e_{i+1}+f_{i})\!\left( b_{i}^{-1}a_{i}(e_{i+1}+f_{i})\right)^{-1}
\end{displaymath} is simply special in $(e_{i+1}+f_{i})R(e_{i+1}+f_{i})$ by~\eqref{eq:simply.special.involution}, Remark~\ref{remark:eRe.non-discrete.irreducible.continuous} and Remark~\ref{remark:matricial.conjugation.invariant}. In turn, by Lemma~\ref{lemma:local.decomposition}, \begin{displaymath}
	u \, \defeq \, \prod\nolimits_{i=1}^{\infty}u_{2i-1}, \qquad u' \, \defeq \, \prod\nolimits_{i=1}^{\infty}u_{2i}
\end{displaymath} are locally special involutions in $R$. Moreover, we observe that, for every $i \in \N_{>0}$, \begin{equation}\label{eq:better.connection}
	b_{i}u_{i}v_{i} \, = \, b_{i}b_{i}^{-1}a_{i}v_{i}\!\left(b_{i}^{-1}a_{i}\right)^{-1}\!v_{i} \, = \, a_{i}v_{i}\!\left(b_{i}^{-1}a_{i}\right)^{-1}\!v_{i} \, \stackrel{\eqref{eq:connection}}{=} \, a_{i}a_{i+1}^{-1} .
\end{equation} Combining this with the pairwise orthogonality of $(e_{n})_{n \in \N_{>0}}$ and the resulting pairwise orthogonality of both $(e_{2i}+e_{2i-1})_{i\in\N_{>1}}$ and $(e_{2i+1}+e_{2i})_{i\in\N_{>1}}$, we conclude that \begin{align}
	b_{1}buv \, &= \, b_{1}\!\left(\prod\nolimits_{i=2}^{\infty} b_{2i-1}\right)\! \left(\prod\nolimits_{i=1}^{\infty}u_{2i-1} \right)\!\left(\prod\nolimits_{i=1}^{\infty}v_{2i-1}\right) \nonumber \\
		&= \, b_{1}\!\left(\prod\nolimits_{i=2}^{\infty} b_{2i-1}\right)\! u_{1}\!\left(\prod\nolimits_{i=2}^{\infty}u_{2i-1} \right)\! v_{1}\!\left(\prod\nolimits_{i=2}^{\infty}v_{2i-1}\right) \nonumber \\
		&\stackrel{\ref{lemma:convergence.sequences}\ref{lemma:convergence.orthogonal}+\ref{lemma:local.decomposition}+\ref{lemma:subgroup.unit.group}\ref{lemma:subgroup.unit.group.orthogonal}}{=} \, b_{1}u_{1}v_{1}\!\left(\prod\nolimits_{i=2}^{\infty} b_{2i-1}\right)\!\left(\prod\nolimits_{i=2}^{\infty}u_{2i-1} \right)\! \left(\prod\nolimits_{i=2}^{\infty}v_{2i-1}\right) \nonumber \\
		&\stackrel{\ref{lemma:local.decomposition}}{=} \, b_{1}u_{1}v_{1}\!\left(\prod\nolimits_{i=2}^{\infty} b_{2i-1}u_{2i-1}v_{2i-1}\right)\! \, \stackrel{\eqref{eq:better.connection}}{=} \, a_{1}a_{2}^{-1}\!\left(\prod\nolimits_{i=2}^{\infty} a_{2i-1}a_{2i}^{-1}\right)\nonumber \\
		&\stackrel{\ref{lemma:local.decomposition}}{=} \, a_{1}a_{2}^{-1}\!\left(\prod\nolimits_{i=2}^{\infty}a_{2i-1}\right)\!\left(\prod\nolimits_{i=2}^{\infty}a_{2i}^{-1}\right) \nonumber \\
		& \stackrel{\ref{lemma:convergence.sequences}\ref{lemma:convergence.orthogonal}+\ref{lemma:local.decomposition}+\ref{lemma:subgroup.unit.group}\ref{lemma:subgroup.unit.group.orthogonal}}{=} \, a_{1}\!\left(\prod\nolimits_{i=2}^{\infty}a_{2i-1}\right)\! a_{2}^{-1}\!\left(\prod\nolimits_{i=2}^{\infty}a_{2i}^{-1}\right) \nonumber \\
		&= \, a_{1}\!\left(\prod\nolimits_{i=1}^{\infty}a_{2i+1}\right)\!\left(\prod\nolimits_{i=1}^{\infty}a_{2i}^{-1}\right)\!, \label{eq:partial.decomposition.1} \\
	b'u'v' \, &= \, \left(\prod\nolimits_{i=1}^{\infty}b_{2i}\right)\!\left(\prod\nolimits_{i=1}^{\infty}u_{2i}\right) \! \left(\prod\nolimits_{i=1}^{\infty}v_{2i}\right) \! \, \stackrel{\ref{lemma:local.decomposition}}{=} \, \prod\nolimits_{i=1}^{\infty}b_{2i}u_{2i}v_{2i} \nonumber\\
		& \stackrel{\eqref{eq:better.connection}}{=} \, \prod\nolimits_{i=1}^{\infty}a_{2i}a_{2i+1}^{-1} \, \stackrel{\ref{lemma:local.decomposition}}{=} \, \!\left(\prod\nolimits_{i=1}^{\infty}a_{2i}\right)\!\left(\prod\nolimits_{i=1}^{\infty}a_{2i+1}^{-1}\right) \! , \label{eq:partial.decomposition.2}
\end{align} so that finally \begin{align*}
	a \, &= \, a_{1} \, = \, a_{1}\!\left(\prod\nolimits_{i=1}^{\infty}a_{2i+1}a_{2i+1}^{-1}\right)\!\left(\prod\nolimits_{i=1}^{\infty}a_{2i}^{-1}a_{2i}\right) \nonumber \\
		&\stackrel{\ref{lemma:local.decomposition}}{=} \, a_{1}\!\left(\prod\nolimits_{i=1}^{\infty}a_{2i+1}\right)\!\left(\prod\nolimits_{i=1}^{\infty}a_{2i+1}^{-1}\right)\!\left(\prod\nolimits_{i=1}^{\infty}a_{2i}^{-1}\right)\!\left(\prod\nolimits_{i=1}^{\infty}a_{2i}\right) \\
		&\stackrel{\ref{lemma:convergence.sequences}\ref{lemma:convergence.orthogonal}+\ref{lemma:local.decomposition}+\ref{lemma:subgroup.unit.group}\ref{lemma:subgroup.unit.group.orthogonal}}{=} \, a_{1}\!\left(\prod\nolimits_{i=1}^{\infty}a_{2i+1}\right)\!\left(\prod\nolimits_{i=1}^{\infty}a_{2i}^{-1}\right)\!\left(\prod\nolimits_{i=1}^{\infty}a_{2i}\right)\!\left(\prod\nolimits_{i=1}^{\infty}a_{2i+1}^{-1}\right) \\
		&\stackrel{\eqref{eq:partial.decomposition.1}+\eqref{eq:partial.decomposition.2}}{=} \, b_{1}buvb'u'v'. \qedhere
\end{align*} \end{proof}

The result above has further implications for the associated unit groups, which are detailed in Theorem~\ref{theorem:width}. As usual, a \emph{commutator} in a group $G$ is any element of the form $[x,y] \defeq xyx^{-1}y^{-1}$ where $x,y \in G$. A group $G$ is said to be \emph{perfect} if its commutator subgroup, i.e., the subgroup of $G$ generated by the set $\{ [x,y] \mid x,y \in G \}$, coincides with $G$.
 	
\begin{remark}\label{remark:product.involutions.commutators} Let $\ell \in \N$, let $G$ be a group, let $(G_{i})_{i\in I}$ be a family of groups, and let $\phi \colon \prod\nolimits_{i\in I} G_{i} \to G$ be a homomorphism. Then \begin{displaymath}
	\phi\!\left( \prod\nolimits_{i \in I} \I(G_{i})^{\ell} \right)\! \, = \, \phi\!\left( {\I\!\left(\prod\nolimits_{i \in I} G_{i}\right)}^{\ell}\right) \, \subseteq \, \I(G)^{\ell}
\end{displaymath} and, for all $m \in \N$ and $w \in \free(m)$, \begin{displaymath}
	\phi\!\left( \prod\nolimits_{i \in I} w(G_{i})^{\ell} \right)\! \, = \, \phi\!\left(\! {w\!\left(\prod\nolimits_{i \in I} G_{i}\right)}^{\ell}\right) \, \subseteq \, w(G)^{\ell} .
\end{displaymath} \end{remark}

Once again, let us recall that the center $\ZZ(R)$ of an irreducible, regular ring $R$ constitutes a field by Remark~\ref{remark:irreducible.center.field}.

\begin{lem}\label{lemma:special.decomposition} Let $R$ be a non-discrete irreducible, continuous ring. \begin{enumerate}
	\item\label{lemma:special.decomposition.commutator} Every simply special element of $R$ is a commutator in $\GL(R)$.
	\item\label{lemma:special.decomposition.word} Suppose that $\ZZ(R)$ is algebraically closed. If $m \in \N$ and $w \in \free (m)\setminus \{ \epsilon \}$, then every simply special element of $R$ belongs to $w(\GL(R))^{2}$.
\end{enumerate} \end{lem}

\begin{proof} Let $K \defeq \ZZ(R)$ and $a \in R$ be simply special in $R$. Then there exist $n \in \N_{>0}$ and a unital $K$-algebra embedding $\phi \colon \M_{n}(K) \to R$ such that $a \in \phi(\SL_{n}(K))$.
	
\ref{lemma:special.decomposition.commutator} If $(n,\vert K\vert )\neq (2,2)$, then every element of $\SL_{n}(K)$ is a commutator in $\GL_{n}(K)$ by Thompson's results~\cite{Thompson61,Thompson62,ThompsonPortugaliae}, which implies that $a$ is a commutator in $\GL(R)$, as desired. Suppose now that $(n,\vert K\vert ) = (2,2)$. According to Lemma~\ref{lemma:matricial.algebra.blow.up}, there exists a unital $K$-algebra $S\leq R$ such that $\phi(\M_{2}(K)) \leq S \cong \M_{4}(K)$. In particular, we find a unital $K$-algebra embedding $\psi\colon \M_{4}(K)\to R$ such that $a \in \psi(\M_{4}(K))$. Consider any $B\in \M_{4}(K)$ such that $a = \psi(B)$. Then \begin{displaymath}
	\rk_{\M_{4}(K)}(B) \, \stackrel{\ref{remark:rank.function.general}\ref{remark:uniqueness.rank.embedding}}{=} \, \rk_{R}(\psi(B)) \, = \, \rk_{R}(a) \, \stackrel{\ref{remark:properties.rank.function}\ref{remark:invertible.rank}}{=} \, 1
\end{displaymath} and therefore \begin{displaymath}
	B \, \stackrel{\ref{remark:properties.rank.function}\ref{remark:invertible.rank}}{\in} \, \GL_{4}(K) \, \stackrel{\vert K \vert = 2}{=} \, \SL_{4}(K) .
\end{displaymath} Thus, $B$ is a commutator in $\GL_{n}(K)$ by~\cite{Thompson61}, whence $a$ is a commutator in $\GL(R)$. This completes the argument.

\ref{lemma:special.decomposition.word} Let $m \in \N$ and $w \in \free (m)\setminus \{ \epsilon \}$. The work of Borel~\cite[\S1, Theorem~1]{Borel} (see also Larsen's proof~\cite[Lemma~3]{Larsen}) asserts that $w(\SL_{n}(K))$ contains a dense open subset of $\SL_{n}(K)$ (with respect to the Zariski topology), whence $\SL_{n}(K) = w(\SL_{n}(K))^{2}$ thanks to~\cite[I, Proposition~1.3(a), p.~47]{BorelBook}, as argued in~\cite[\S1, p.~157, Remark~3]{Borel}. Consequently, $a \in \phi(\SL_{n}(K)) = \phi(w(\SL_{n}(K))^{2}) \subseteq w(\GL(R))^{2}$. \end{proof}
	
\begin{lem}\label{lemma:locally.special.decomposition} Let $R$ be a non-discrete irreducible, continuous ring. \begin{enumerate}
	\item\label{lemma:locally.special.decomposition.involutions} Every locally special element of $R$ is a product of $4$ involutions in $\GL(R)$.
	\item\label{lemma:locally.special.decomposition.commutators} Every locally special element of $R$ is a commutator in $\GL(R)$.
	\item\label{lemma:locally.special.decomposition.word} Suppose that $\ZZ(R)$ is algebraically closed. If $m \in \N$ and $w \in \free (m)\setminus \{ \epsilon \}$, then every locally special element of $R$ belongs to $w(\GL(R))^{2}$.
\end{enumerate} \end{lem}

\begin{proof} Let $b \in \GL(R)$ be locally special. Then we find $(e_{n})_{n\in\N} \in (\E(R)\setminus \{ 0 \})^{\N}$ pairwise orthogonal and $(b_{n})_{n\in\N} \in \prod\nolimits_{n \in \N} \GL(e_{n}Re_{n})$ such that the map \begin{align*}
	\phi\colon\, \prod\nolimits_{n \in \N} \GL(e_{n}Re_{n}) \, \longrightarrow \, \GL(R), \quad (a_{n})_{n\in\N} \, \longmapsto \, \prod\nolimits_{n \in \N} a_{n}+1-e_{n}	
\end{align*} satisfies $\phi((b_{n})_{n\in\N}) = b$ and, for each $n\in\N$, the element $b_{n}$ is simply special in $e_{n}Re_{n}$. We proceed by separate arguments for~\ref{lemma:locally.special.decomposition.involutions}, \ref{lemma:locally.special.decomposition.commutators}, and~\ref{lemma:locally.special.decomposition.word}.

\ref{lemma:locally.special.decomposition.involutions} Let $K \defeq \ZZ(R)$. For every $n \in \N$, since $\ZZ(e_{n}Re_{n}) \cong K$ by Remark~\ref{remark:eRe.non-discrete.irreducible.continuous}, there exist $m_{n} \in \N_{>0}$ and a unital $K$-algebra homomorphism $\psi_{n} \colon \M_{m_{n}}(K) \to e_{n}Re_{n}$ such that $b_{n} \in \psi_{n}(\SL_{m_{n}}(K))$. For each $n\in\N$, due to~\cite{GustafsonHalmosRadjavi76}, every element of $\SL_{m_{n}}(K)$ is a product of $4$ involutions in $\GL_{m_{n}}(K)$, hence $b_{n}$ is a product of $4$ involutions in~$\GL(e_{n}Re_{n})$. Therefore, by Theorem~\ref{theorem:unique.rank.function}, Lemma~\ref{lemma:local.decomposition} and Remark~\ref{remark:product.involutions.commutators}, the element $b$ is a product of $4$ involutions in $\GL(R)$. 

\ref{lemma:locally.special.decomposition.commutators} For every $n \in \N$, Remark~\ref{remark:eRe.non-discrete.irreducible.continuous} and Lemma~\ref{lemma:special.decomposition}\ref{lemma:special.decomposition.commutator} together assert that $b_{n}$ is a commutator in $\GL(e_{n}Re_{n})$. As Lemma~\ref{lemma:local.decomposition} asserts that $\phi$ is a group homomorphism, Remark~\ref{remark:product.involutions.commutators} implies that $b$ is a commutator in $\GL(R)$.

\ref{lemma:locally.special.decomposition.word} Let $m \in \N$ and $w \in \free (m)\setminus \{ \epsilon \}$. Then \begin{displaymath}
	b \, = \, \phi((b_{n})_{n\in\N}) \, \stackrel{\ref{remark:eRe.non-discrete.irreducible.continuous}+\ref{lemma:special.decomposition}\ref{lemma:special.decomposition.word}}{\in} \, \phi\!\left( \prod\nolimits_{n \in \N} w(\GL(e_{n}Re_{n}))^{2} \right)\! \, \stackrel{\ref{lemma:local.decomposition}+\ref{remark:product.involutions.commutators}}{\subseteq} \, w(\GL(R))^{2} .\qedhere
\end{displaymath} \end{proof}	

Everything is prepared to deduce the announced width bounds.

\begin{thm}\label{theorem:width} Let $R$ be a non-discrete irreducible, continuous ring. \begin{enumerate}
	\item\label{theorem:width.involutions} Every element of $\GL(R)$ is a product of $16$ involutions.
	\item\label{theorem:width.commutators} Every element of $\GL(R)$ is a product of $7$ commutators. In particular, $\GL(R)$ is perfect.
	\item\label{theorem:width.word} Suppose that $\ZZ(R)$ is algebraically closed. For all $m \in \N$ and $w \in \free (m)\setminus \{ \epsilon \}$,
				\begin{displaymath}
					\qquad \GL(R) \, = \, w(\GL(R))^{14} .
				\end{displaymath} In particular, $\GL(R)$ is verbally simple.
\end{enumerate} \end{thm}

\begin{proof} \ref{theorem:width.involutions} This follows from Theorem~\ref{theorem:decomposition}, Remark~\ref{remark:locally.special}\ref{remark:locally.special.1}, and Lemma~\ref{lemma:locally.special.decomposition}\ref{lemma:locally.special.decomposition.involutions}.
	
\ref{theorem:width.commutators} This is a consequence of Theorem~\ref{theorem:decomposition}, Remark~\ref{remark:locally.special}\ref{remark:locally.special.1}, and Lemma~\ref{lemma:locally.special.decomposition}\ref{lemma:locally.special.decomposition.commutators}.

\ref{theorem:width.word} This is due to Theorem~\ref{theorem:decomposition}, Remark~\ref{remark:locally.special}\ref{remark:locally.special.1}, and Lemma~\ref{lemma:locally.special.decomposition}\ref{lemma:locally.special.decomposition.word}. \end{proof}

\section{Steinhaus property}\label{section:steinhaus.property}

This section is dedicated to the proof of Theorem~\ref{theorem:194-Steinhaus}. Among other things, we will make use of the following two basic facts about metric spaces.

\begin{remark}\label{remark:metric.space} Let $X$ be a metric space. \begin{enumerate}
	\item\label{remark:uncountable.separable} If $X$ is separable, then every discrete subspace of $X$ is countable (see, e.g.,~\cite[4.1, Theorem~4.1.15, p.~255]{EngelkingBook}).
	\item \label{remark:neighborhood} A subset $U\subseteq X$ is a neighborhood of a point $x \in X$ if and only if, for every sequence $(x_{n})_{n \in \N}$ in $X$ converging to $x$, there is $m \in \N$ with~$x_{m} \in U$. While ($\Longrightarrow$) is trivial, the implication ($\Longleftarrow$) follows by contraposition, considering any sequence from the non-empty set $\prod_{n \in \N} \! \left\{y\in X\setminus U \left\vert \, d(x,y)<\tfrac{1}{n+1}\right\}\right.$. 
\end{enumerate}	\end{remark}
	
We now proceed to the proof of Theorem~\ref{theorem:194-Steinhaus}. The global strategy follows the ideas of \cite[Theorem~3.1]{KittrellTsankov} (see also~\cite[Proposition~A.1]{BenYaacovBerensteinMelleray}), but our setting requires a very careful analysis of several algebraic peculiarities.
	
\begin{definition} Let $R$ be a unital ring and $e\in\E(R)$. A subset $W\subseteq \GL(R)$ is called \emph{full} for $e$ if, for every $t \in \GL(eRe)$, there exists $s \in W$ such that $t = se = es$. \end{definition}

\begin{lem}\label{lemma:full.c_mW} Let $(R,\rho)$ be a complete rank ring and let $(e_{m})_{m \in \N} \in \E(R)^{\N}$ be pairwise orthogonal. Then the following hold. \begin{enumerate}
	\item\label{lemma:full.c_mW.A} For every $(t_{m})_{m \in \N} \in \prod_{m \in \N} \GL(e_{m}Re_{m})$, there exists $t \in \GL(R)$ such that \begin{displaymath}
			\qquad \forall m \in \N \colon \quad t_{m} = te_{m} = e_{m}t . 
		\end{displaymath}
	\item\label{lemma:full.c_mW.B} For every sequence $(W_{m})_{m \in \N}$ of subsets of $\GL(R)$ with $\GL(R) = \bigcup\nolimits_{m \in \N} W_{m}$, there exists $m \in \N$ such that $W_{m}$ is full for $e_{m}$.
\end{enumerate} \end{lem}

\begin{proof} \ref{lemma:full.c_mW.A} Let $(t_{m})_{m \in \N} \in \prod_{m \in \N} \GL(e_{m}Re_{m})$. By Lemma~\ref{lemma:local.decomposition} and Lemma~\ref{lemma:subgroup.unit.group}, \begin{displaymath}
	t \, \defeq \, \left(\sum\nolimits_{n \in \N} t_{n}\right) \! +\! \left(1-\sum\nolimits_{n \in \N} e_{n} \right)\! \, \in \, \GL(R).
\end{displaymath} For every $m \in \N$, we see that \begin{align*}
	&te_{m} \, = \, \! \left(\sum\nolimits_{n \in \N} t_{n} + 1 - \sum\nolimits_{n \in \N} e_{n} \right) \! e_{m} \, \stackrel{\ref{remark:properties.pseudo.rank.function}\ref{remark:rank.group.topology}}{=} \, t_{m}, \\
	&e_{m}t \, = \, e_{m}\!\left(\sum\nolimits_{n \in \N} t_{n} + 1 - \sum\nolimits_{n \in \N} e_{n}\right)\! \, \stackrel{\ref{remark:properties.pseudo.rank.function}\ref{remark:rank.group.topology}}{=} \, t_{m} .
\end{align*}

\ref{lemma:full.c_mW.B} Let $(W_{m})_{m \in \N}$ be a sequence of subsets of $\GL(R)$ with $\GL(R) = \bigcup\nolimits_{m \in \N} W_{m}$. For contradiction assume that, for each $m \in \N$, the set $W_{m}$ is not full for~$e_{m}$, i.e., there exists a sequence $(t_{m})_{m \in \N} \in \prod_{m \in \N} \GL(e_{m}Re_{m})$ such that \begin{equation}
	\forall m \in \N \ \forall s \in W_{m} \colon \quad \neg (t_{m} = se_{m} = e_{m}s) . \label{equation:full0}
\end{equation} Now, by~\ref{lemma:full.c_mW.A}, we find some $t \in \GL(R)$ such that $t_{m} = te_{m} = e_{m}t$ for each $m \in \N$. Due to~\eqref{equation:full0}, it follows that $t \notin W_{m}$ for every $m \in \N$. Hence, $t \notin \bigcup_{m\in \N} W_{m} = \GL(R)$, which is the desired contradiction. \end{proof}

\begin{lem}\label{lemma:full.W^2} Let $(R,\rho)$ be a complete rank ring, let $W\subseteq \GL(R)$ be symmetric and countably syndetic and let $(e_m)_{m\in \N}\in \E(R)^{\N}$ be pairwise orthogonal. Then there exists $m\in \N$  such that $W^2$ is full for $e_m$. \end{lem}

\begin{proof}  Since $W$ is countably syndetic, there exists a sequence $(c_{m})_{m\in \N} \in \GL(R)^{\N}$ such that $\GL(R) = \bigcup_{m \in \N} c_{m}W$. By Lemma~\ref{lemma:full.c_mW}\ref{lemma:full.c_mW.B}, there exists $m \in \N$ such that $c_{m}W$ is full for $e_{m} \eqdef e$. We will show that $W^{2}$ is full for $e$, too. To this end, let $t \in \GL(eRe)$. Since $c_{m}W$ is full for $e$, there exists $s\in c_{m}W$ such that \begin{equation}
	t \, = \, se \, = \, es . \label{eq--13}
\end{equation} Since $t \in \GL(eRe)$, we see that \begin{displaymath}
	st \, = \, set \, \stackrel{\eqref{eq--13}}{=} \, t^{2} \, \in \, \GL(eRe) .
\end{displaymath} Consequently, again by fullness of $c_{m}W$ for $e$, there exists $\tilde{s} \in c_{m}W$ such that \begin{equation}
	st \, = \, \tilde{s}e \, = \, e\tilde{s} . \label{eq--15}
\end{equation} Hence, the element $s^{-1}\tilde{s} \in (c_{m}W)^{-1}(c_{m}W) = W^{-1}W = W^{2}$ satisfies, as desired, \begin{displaymath}
	s^{-1}\tilde{s}e \, \stackrel{\eqref{eq--15}}{=} \, t \, \stackrel{\eqref{eq--15}}{=} \, s^{-1}e\tilde{s} \, = \, s^{-1}ess^{-1}\tilde{s} \, \stackrel{\eqref{eq--13}}{=} \, s^{-1}ses^{-1}\tilde{s} \, = \, es^{-1}\tilde{s} . \qedhere
\end{displaymath} \end{proof}
	
\begin{lem}\label{lemma:Gamma_{R}(e).subset.W^192} Let $R$ be a non-discrete irreducible, continuous ring, let $W \subseteq \GL(R)$ be symmetric and countably syndetic. Then there is $e \in \E(R) \setminus \{ 0 \}$ with $\Gamma_{R}(e) \subseteq W^{192}$. \end{lem}

\begin{proof} By Remark~\ref{remark:rank.function.general}\ref{remark:characterization.discrete} and Lemma~\ref{lemma:order}, there exists a sequence $(e_{m})_{m\in\N} \in \E(R)^{\N}$ of pairwise orthogonal non-zero elements with $\sum\nolimits_{m\in\N}e_{m} = 1$. Thanks to Theorem~\ref{theorem:unique.rank.function} and Lemma~\ref{lemma:full.W^2}, we find $m\in\N$ such that $W^{2}$ is full for $e_{m} \eqdef e$. Let \begin{displaymath}
	\theta \, \defeq \, \begin{cases}
					\, 4 & \text{if } \cha(R)=2 , \\
					\, 2 & \text{otherwise}.
				\end{cases}
\end{displaymath} We proceed in three intermediate steps, marked by the items~\eqref{claim1}, \eqref{claim2} and~\eqref{claim3}.
		
First we prove the existence of an element $s\in W^2\cap \I(\Gamma_{R}(e))$ such that \begin{equation}\label{claim1}
	0 \, < \, \rk_{R}(1-s) \, < \, \tfrac{\rk_{R}(e)}{\theta}.
\end{equation} By Remark~\ref{remark:eRe.non-discrete.irreducible.continuous}, Corollary~\ref{corollary:boolean.subgroup} and Lemma~\ref{lemma:subgroup.unit.group}, $\Gamma_{R}(e)$ admits an uncountable, separable subgroup $\Gamma$ consisting entirely of involutions. As $W$ is countably syndetic in $\GL(R)$, there exists a countable subset $C \subseteq \GL(R)$ such that $\GL(R) = CW$. Since $\Gamma$ is uncountable, there exists $c \in C$ such that $\Gamma \cap cW$ is uncountable. As $\Gamma$ is separable, Remark~\ref{remark:metric.space}\ref{remark:uncountable.separable} implies that $\Gamma \cap cW$ is not discrete, thus there exist $s_{1},s_{2} \in \Gamma \cap cW$ such that \begin{displaymath}
	0 \, < \, d_{R}(s_{1},s_{2}) \, < \, \tfrac{\rk_{R}(e)}{\theta} .
\end{displaymath} Consider $s\defeq s_{1}s_{2} \in \Gamma \subseteq \I(\Gamma_{R}(e))$. Moreover, \begin{displaymath}
	s \, = \, s_{1}s_{2} \, = \, s_{1}^{-1}s_{2} \, \in \, (cW)^{-1}(cW) \, = \, W^{-1}W \, = \, W^{2}
\end{displaymath} and \begin{displaymath}
	\rk_{R}(1-s) \, = \, \rk_{R}\!\left(1-s_{1}^{-1}s_{2}\right)\! \, = \, \rk_{R}(s_{1}-s_{2}) \, = \, d_{R}(s_{1},s_{2}),
\end{displaymath} thus $0<\rk_{R}(1-s)<\tfrac{\rk_{R}(e)}{\theta}$, which proves~\eqref{claim1}.
			
Second we show that \begin{equation}\label{claim2}
	\forall u \in \I(\Gamma_{R}(e)) \colon \quad \rk_{R}(1-u) = \rk_{R}(1-s) \ \Longrightarrow \ u \in W^6 .
\end{equation} Let $u \in \I(\Gamma_{R}(e))$ with $\rk_{R}(1-u) = \rk_{R}(1-s)$. We recall that $\Gamma_{R}(e) \cong \GL(eRe)$ due to Lemma~\ref{lemma:subgroup.unit.group}. Hence, by Remark~\ref{remark:eRe.non-discrete.irreducible.continuous} and Proposition~\ref{proposition:conjugation.involution.rank}, there exists $v \in \Gamma_{R}(e)$ such that $vsv^{-1} = u$. Since $ve \in \GL(eRe)$ by Lemma~\ref{lemma:subgroup.unit.group} and $W^{2}$ is full for $e$, we find $\tilde{v} \in W^{2}$ such that \begin{equation}\label{eqstar}
	ve \, = \, \tilde{v}e \, = \, e\tilde{v} .
\end{equation} We deduce that \begin{equation}\label{eqstarstar}
	\tilde{v}^{-1}e \, \stackrel{\eqref{eqstar}}{=} \, e\tilde{v}^{-1} \, \stackrel{\eqref{eqstar}}{=} \, v^{-1}e .       
\end{equation} Therefore, both \begin{align*}
	\tilde{v}s\tilde{v}^{-1}e \, &\stackrel{\eqref{eqstarstar}}{=} \, \tilde{v}sv^{-1}e \, \stackrel{sv^{-1}\in \Gamma_{R}(e)}{=} \, \tilde{v}e sv^{-1}e \, \stackrel{\eqref{eqstar}}{=} \, vesv^{-1}e \, \stackrel{sv^{-1}\in \Gamma_{R}(e)}{=} \, vsv^{-1}e=ue
\end{align*} and \begin{align*}
	\tilde{v}s\tilde{v}^{-1}(1-e) &\stackrel{\eqref{eqstarstar}}{=} \tilde{v}s(1-e)\tilde{v}^{-1} \stackrel{s\in\Gamma_{R}(e)}{=} \tilde{v}(1-e)\tilde{v}^{-1} \stackrel{\eqref{eqstarstar}}{=} \tilde{v}\tilde{v}^{-1}(1-e) \stackrel{u\in \Gamma_{R}(e)}{=} u(1-e) .
\end{align*} Consequently, \begin{displaymath}
	\tilde{v}s\tilde{v}^{-1} \, = \, \tilde{v}s\tilde{v}^{-1}e + \tilde{v}s\tilde{v}^{-1}(1-e) \, = \, ue + u(1-e) \, = \, u
\end{displaymath} and hence $u = \tilde{v}s\tilde{v}^{-1} \in W^{2}W^{2}W^{-2} = W^{6}$, as desired in~\eqref{claim2}.
			
Next we prove the existence of an element $f \in \E(R)\setminus\{0\}$ such that \begin{equation}\label{claim3}
	\I(\Gamma_{R}(f)) \, \subseteq \, W^{12}.
\end{equation} Note that \begin{displaymath}
	2\rk_{R}(1-s) \, \leq \, \theta\rk_{R}(1-s) \, \stackrel{\eqref{claim1}}{<} \, \rk_{R}(e).
\end{displaymath} By Lemma~\ref{lemma:invertible.rank.idempotent}, there is $f' \in \E(R)$ with $f' \leq e$, $s \in \Gamma_{R}(f')$ and $\rk_{R}(f') \leq \theta\rk_{R}(1-s)$. Furthermore, invoking Remark~\ref{remark:rank.function.general}\ref{remark:characterization.discrete} and Lemma~\ref{lemma:order}\ref{lemma:order.1}, we find $f \in \E(R)$ such that $f' \leq f \leq e$ and $\rk_{R}(f)=\theta\rk_{R}(1-s)$. Also, $\Gamma_{R}(f') \subseteq \Gamma_{R}(f)$ by Lemma~\ref{lemma:subgroup.unit.group}\ref{lemma:subgroup.unit.group.order}, whence $s \in \Gamma_{R}(f)$. Note that $s\ne 1$ by~\eqref{claim1}, which necessitates that $f\neq0$. Thanks to Remark~\ref{remark:eRe.non-discrete.irreducible.continuous} and the conjunction of Lemma~\ref{lemma:involution.rank.distance.char.neq2} and Lemma~\ref{lemma:involution.rank.distance.char=2}, \begin{align*}
	\I(fRf) \, &\subseteq \, \left\{gh\left\vert \, g,h\in\I(fRf),\, \rk_{fRf}(f-g)=\rk_{fRf}(f-h)=\tfrac{1}{\theta}\right\} \right. \\
	&\stackrel{\ref{remark:eRe.non-discrete.irreducible.continuous}}{=} \, \left\{gh\left\vert \, g,h\in\I(fRf),\, \rk_{R}(f-g)=\rk_{R}(f-h)=\rk_{R}(1-s) \right\} . \right.
\end{align*} Applying the isomorphism $\GL(fRf) \to \Gamma_{R}(f), \, a \mapsto a + 1-f$ from Lemma~\ref{lemma:subgroup.unit.group}, we conclude that \begin{align*}
	\I(\Gamma_{R}(f)) \, &\subseteq \, \{gh\mid g,h\in\I(\Gamma_{R}(f)), \, \rk_{R}(1-g)=\rk_{R}(1-h)=\rk_{R}(1-s)\}\\
		&\stackrel{\ref{lemma:subgroup.unit.group}\ref{lemma:subgroup.unit.group.order}}{\subseteq} \, \{gh\mid g,h\in\I(\Gamma_{R}(e)), \, \rk_{R}(1-g)=\rk_{R}(1-h)=\rk_{R}(1-s)\}\\
		&\stackrel{\text{\eqref{claim2}}}{\subseteq} \, W^{6}W^{6} \, = \, W^{12},
\end{align*} i.e., \eqref{claim3} holds.
			
Finally, by Theorem~\ref{theorem:width}\ref{theorem:width.involutions} and Remark~\ref{remark:eRe.non-discrete.irreducible.continuous}, every element of $\Gamma_{R}(f)\overset{\ref{lemma:subgroup.unit.group}}{\cong} \GL(fRf)$ is a product of $16$ involutions, whence \begin{displaymath}
	\Gamma_{R}(f) \, = \, \I(\Gamma_{R}(f))^{16} \, \overset{\eqref{claim3}}{\subseteq}\, W^{12\cdot 16} \, = \, W^{192}. \qedhere
\end{displaymath} \end{proof}		
	
\begin{lem}\label{lemma:GL(R).covered.by.c_nW} Let $R$ be an irreducible, continuous ring, let $W \subseteq \GL(R)$ be symmetric and countably syndetic, and let $e \in \E(R) \setminus \{ 0 \}$ and $\ell \in \N$ be such that $\Gamma_{R}(e) \subseteq W^{\ell}$. Then $W^{\ell+2}$ is an identity neighborhood in $\GL(R)$. \end{lem}
	
\begin{proof} By Remark~\ref{remark:metric.space}\ref{remark:neighborhood}, it suffices to check that, for every sequence $(t_{n})_{n \in \N}$ in $\GL(R)$ converging to $1$, there exists $m\in\N$ such that $t_{m} \in W^{\ell+2}$. To this end, let $(t_{n})_{n\in\N}\in\GL(R)^{\N}$ with $\lim_{n \to \infty} \rk_{R}(1-t_{n}) = 0$. Upon passing to a subsequence, we may and will assume that \begin{equation}\label{eq--18}
	\sum\nolimits_{n \in \N} \rk_{R}(1-t_{n}) \, < \, \tfrac{1}{2}\rk_{R}(e).
\end{equation} Let us note that \begin{equation}\label{eq--28}
	\forall a \in R \colon \quad \delta_{\latop(R)}(Ra) \, \stackrel{\ref{remark:rank.function.general}\ref{remark:duality}}{=} \, 1-\delta_{\lat(R)}(\rAnn(Ra)) \, \stackrel{\ref{theorem:unique.rank.function}+\ref{lemma:pseudo.dimension.function}\ref{lemma:rank.dimension.annihilator}}{=} \, \rk_{R}(a) .
\end{equation} Since $W$ is countably syndetic, there exists a sequence $(c_n)_{n\in \N}\in\GL(R)^{\N}$ such that $\GL(R)=\bigcup\nolimits_{n \in \N} c_nW$. We see that \begin{align}
	\delta_{\latop(R)}\!\left(\bigvee\nolimits_{n \in \N} R(1-c_{n}t_{n}c_{n}^{-1})\right) \,
		&\stackrel{\ref{proposition:dimension.function.continuous}}{=} \, \sup\nolimits_{m \in \N} \delta_{\latop(R)}\!\left(\bigvee\nolimits_{n=0}^{m}R(1-c_{n}t_{n}c_{n}^{-1})\right) \nonumber \\
		&\leq \, \sup\nolimits_{m \in \N} \sum\nolimits_{n=0}^{m}\delta_{\latop(R)}(R(1-c_{n}t_{n}c_{n}^{-1}))\nonumber\\
		&= \, \sum\nolimits_{n \in \N} \rk_{R}(1-c_{n}t_{n}c_{n}^{-1})\nonumber\\
		&= \, \sum\nolimits_{n \in \N} \rk_{R}(1-t_{n}) \label{eq--4} 
\end{align} and therefore \begin{align}\label{eq--19}
	\delta_{\lat(R)}\left(\bigwedge\nolimits_{n \in \N} \rAnn(R(1-c_{n}t_{n}c_{n}^{-1}))\right) \,
		&\stackrel{\ref{remark:bijection.annihilator}}{=} \, \delta_{\lat(R)}\left(\rAnn\left(\bigvee\nolimits_{n \in \N} R(1-c_{n}t_{n}c_{n}^{-1})\right)\right)\nonumber\\
		&\stackrel{\eqref{eq--28}}{=} \, 1-\delta_{\latop(R)}\left(\bigvee\nolimits_{n \in \N} R(1-c_{n}t_{n}c_{n}^{-1})\right)\nonumber\\
		&\stackrel{\eqref{eq--4}}{\geq} \, 1-\sum\nolimits_{n \in \N} \rk_{R}(1-t_{n})\nonumber\\
		&\stackrel{\eqref{eq--18}}{>} \, 1-\tfrac{1}{2}\rk_{R}(e).
\end{align} Consider \begin{displaymath}
	J \, \defeq \, \bigwedge\nolimits_{n \in \N} \rAnn(1-c_{n}t_{n}c_{n}^{-1}) \, \in \, \lat(R) .
\end{displaymath} Since $R$ is regular, there exists $\tilde{e}\in\E(R)$ such that $\tilde{e}R = J$. By \cite[Lemma~7.5]{SchneiderGAFA}, there exists $f\in\E(R)$ such that $1-\tilde{e}\leq f$ and \begin{displaymath}
	fR \, = \, (1-\tilde{e})R + \bigvee\nolimits_{n\in\N}(1-c_{n}t_{n}c_{n}^{-1})R
\end{displaymath} We observe that \begin{align*}
	\delta_{\lat(R)}((1-\tilde{e})R) \, &=\, \rk_{R}(1-\tilde{e}) \stackrel{\ref{remark:properties.pseudo.rank.function}\ref{remark:rank.difference.smaller.idempotent}}{=} \, 1-\rk_{R}(\tilde{e}) \, = \, 1-\delta_{\lat(R)}(J) \\
	& \stackrel{\eqref{eq--19}}{<} \, 1-\!\left(1-\tfrac{1}{2}\rk_{R}(e)\right)\! \, = \, \tfrac{1}{2}\rk_{R}(e) 
\end{align*} and hence \begin{align*}
	\rk_{R}(f) \, &= \, \delta_{\lat(R)}(fR) \, \leq \, \delta_{\lat(R)}((1-\tilde{e})R) + \delta_{\lat(R)}\!\left(\bigvee\nolimits_{n \in \N} (1-c_{n}t_{n}c_{n}^{-1})R\right) \\
	&< \, \tfrac{1}{2}\rk_R(e) + \delta_{\lat(R)}\!\left(\bigvee\nolimits_{n \in \N} (1-c_{n}t_{n}c_{n}^{-1})R\right) \\
	&\stackrel{\eqref{eq--4}}{\leq} \, \tfrac{1}{2}\rk_R(e) + \sum\nolimits_{n \in \N} \rk_R(1-t_{n}) \, \stackrel{\eqref{eq--18}}{<} \, \rk_{R}(e) .
\end{align*} By Lemma~\ref{lemma:order}\ref{lemma:order.1}, there exists $e_{0}\in\E(R)$ such that $e_{0} \leq e$ and $\rk_{R}(e_{0}) = \rk_{R}(f)$. By Lemma~\ref{lemma:conjugation.idempotent}, there exists $g\in\GL(R)$ such that $ge_{0}g^{-1} = f$. Since $\GL(R) = \bigcup\nolimits_{n=0}^{\infty} c_{n}W$, there is $m \in \N$ such that $g \in c_{m}W$. Consider $s \defeq c_{m}^{-1}g \in W$ and note that \begin{align}
	e_{0} \, = \, s^{-1}se_{0}s^{-1}s \, = \, s^{-1}c_{m}^{-1}ge_{0}g^{-1}c_{m}s \, = \, s^{-1}c_{m}^{-1}fc_{m}s .\label{something}
\end{align} Furthermore, we see that \begin{enumerate}
	\item[---\,] $1-\tilde{e}\leq f$,
	\item[---\,] $\tilde{e}\in\bigwedge_{n \in \N} \rAnn(1-c_{n}t_{n}c_{n}^{-1})\subseteq \rAnn(1-c_{m}t_{m}c_{m}^{-1})$,
	\item[---\,] $(1-c_{m}t_{m}c_{m}^{-1})R\subseteq\bigvee_{n \in \N} (1-c_{n}t_{n}c_{n}^{-1})R\subseteq fR$.
\end{enumerate} Consequently, Lemma~\ref{lemma:Gamma.annihilator.right-ideal} asserts that $c_{m}t_{m}c_{m}^{-1} \in \Gamma_{R}(f)$. In turn, \begin{displaymath}
	s^{-1}t_{m}s \, \stackrel{\ref{lemma:subgroup.unit.group}\ref{lemma:subgroup.unit.group.conjugation}}{\in} \, \Gamma_{R}(s^{-1}c_{m}^{-1}fc_{m}s) \, \stackrel{\eqref{something}}{=} \, \Gamma_{R}(e_{0}) \, \stackrel{\ref{lemma:subgroup.unit.group}\ref{lemma:subgroup.unit.group.order}}{\subseteq} \, \Gamma_{R}(e) \, \subseteq \, W^{\ell} 
\end{displaymath} and thus $t_{m} = ss^{-1}t_{m}ss^{-1} \in sW^{\ell}s^{-1} \subseteq W^{\ell+2}$, as desired. \end{proof}
	
We arrive at this section's main result.
	
\begin{thm}\label{theorem:194-Steinhaus} Let $R$ be an irreducible, continuous ring. Then the topological group $\GL(R)$ is $194$-Steinhaus. In particular, $\GL(R)$ has automatic continuity. \end{thm}

\begin{proof} Since any discrete group is even $0$-Steinhaus, the desired conclusion is trivial if $R$ is discrete. If $R$ is non-discrete, then the claim follows from Lemma~\ref{lemma:Gamma_{R}(e).subset.W^192} and Lemma~\ref{lemma:GL(R).covered.by.c_nW}. In turn, $\GL(R)$ has automatic continuity by~\cite[Proposition~2]{RosendalSolecki}. \end{proof}
	
\section{Proofs of corollaries}\label{section:proofs.of.corollaries}
	
\begin{proof}[Proof of Corollary~\ref{corollary:unique.polish.topology}] Note that $\GL(R)$ is closed in $(R,d_R)$ by Remark~\ref{remark:properties.rank.function}\ref{remark:GL(R).closed}. Since closed subspaces of complete metric spaces are complete, Theorem~\ref{theorem:unique.rank.function} entails that $(\GL(R),d_{R})$ is complete. Suppose now that $(R,d_{R})$ is separable. As subspaces of separable metric spaces are separable, we conclude that $(\GL(R),d_{R})$ is separable. Hence, the rank topology $\tau$ on $\GL(R)$, which is a group topology by Remark~\ref{remark:properties.pseudo.rank.function}\ref{remark:rank.group.topology}, is also Polish. Let $\sigma$ be another Polish group topology on $\GL(R)$. Then Theorem~\ref{theorem:194-Steinhaus} asserts that the identity homomorphism $\id\colon (\GL(R),\tau)\to (\GL(R),\sigma)$ is continuous, i.e., $\sigma\subseteq \tau$. By the open mapping theorem for Polish groups (see, e.g.,~\cite[Theorem~3]{RosendalSuarez}), $\id \colon (\GL(R),\tau)\to (\GL(R),\sigma)$ is also open, that is, $\tau\subseteq \sigma$. Therefore, $\sigma=\tau$. \end{proof} 
	
\begin{remark}\label{remark:separable} Let $R$ be an irreducible, continuous ring. With respect to the rank topology, $R$ is separable if and only if $\GL(R)$ is separable. The implication ($\Longrightarrow$) is obvious, as argued in the proof of Corollary~\ref{corollary:unique.polish.topology}. In order to prove ($\Longleftarrow$), consider any maximal chain $E$ in $(\E(R),{\leq})$, the existence of which is guaranteed by the Hausdorff maximal principle. Note that $\rk_{R}(E) = \rk_{R}(R)$ due to~\cite[Corollary~7.19]{SchneiderGAFA}, therefore~\cite[II.XVII, Theorem~17.1(d), p.~224]{VonNeumannBook} asserts that \begin{displaymath}
	\GL(R) \times E \times \GL(R) \, \longrightarrow \, R, \quad (u,e,v) \, \longmapsto \, uev
\end{displaymath} is surjective. Moreover, this map is continuous for the rank topology by Remark~\ref{remark:properties.pseudo.rank.function}\ref{remark:rank.group.topology}. Since ${{\rk_{R}}\vert_{E}}\colon (E,d_{R}) \to ([0,1],d)$ is isometric with regard to the standard metric $d$ on $[0,1]$ by Remark~\ref{remark:properties.pseudo.rank.function}\ref{remark:rank.order.isomorphism}, we see that $(E,d_{R})$ is separable. Consequently, if $\GL(R)$ is separable with respect to the rank topology, then so is $R$. \end{remark}	
	
\begin{proof}[Proof of Corollary~\ref{corollary:inert}] The homeomorphism group $\Homeo(X)$ of any metrizable compact space $X$, endowed with the topology of uniform convergence, constitutes a separable topological group (see, e.g.,~\cite[I.9, Example~9.B 8), p.~60]{KechrisBook}). Any action of $\GL(R)$ by homeomorphisms on a non-empty metrizable compact space $X$ gives rise to a homomorphism from $\GL(R)$ to $\Homeo(X)$, which then has to be continuous with respect to the rank topology due to Theorem~\ref{theorem:194-Steinhaus}, wherefore the action is (jointly) continuous. Hence, the claim follows by~\cite[Corollary~1.6]{SchneiderGAFA}. \end{proof}
	
\begin{proof}[Proof of Corollary~\ref{corollary:fixpoint}] Let $K$ be a finite field. As explained in the proof of Corollary~\ref{corollary:inert}, our Theorem~\ref{theorem:194-Steinhaus} implies that every action of $\GL(\M_{\infty}(K))$ by homeomorphisms on a metrizable compact space is continuous for the rank topology. Since the topological group $\GL(\M_{\infty}(K))$ is \emph{extremely amenable}~\cite{CarderiThom}, this entails the claim. \end{proof}						
		


\end{document}